\documentclass{siamltex}

\usepackage{bm}
\usepackage{bbm}
\usepackage{hyperref}
\usepackage{graphicx}
\usepackage{amsmath}
\usepackage{amsfonts}
\usepackage{amssymb}
\usepackage{exscale}
\usepackage{multirow}
\usepackage{stmaryrd}
\usepackage{tikz}
\usetikzlibrary{arrows,calc}
\usepackage{color}
\usepackage{soul}
\usepackage{cancel}
\usepackage{mathrsfs}
\usepackage{comment}

\usepackage{bm}
\usepackage{graphicx}

\usetikzlibrary{arrows,calc}
\numberwithin{figure}{section}

\newcommand{\ie}{{\it i.e.}}
\newcommand{\tie}{{that~is}}

\newcommand{\eg}{{\it e.g.}}

\newcommand{\dist}{{\rm{dist}}}

\newcommand{\gfrac}[2]{{{#1}/{#2}}}

\newcommand{\ffrac}{\displaystyle \frac}

\newcommand{\fns}{\footnotesize}

\newcommand{\rme}{\mathrm{e}}
\newcommand{\rmi}{{\mathrm{i}}}

\newcommand{\qfa}{\quad\text{{for all} }}

\def\bal{\begin{aligned}}

\def\ibt{\begin{itemize}}

\def\cds{{C_{\ddd,s}}}
\def\cdotsp{{\,\cdot\,}}

\def\cds{{C_{\ddd,s}}}

\def\ddd{{d}}
\def\D{{d }}

\def\del{{\delta}}
\def\dxb{{d \xb}} 
\def\dyb{{d \yb}} 
\def\dydx{{d \yb d \xb}} 
\def\eal{\end{aligned}}
\def\eit{\end{itemize}}

\def\fh{{\widehat{f}}}

\def\fhvoh{{(\fh,v)_\omgh}}
\def\fhvohd{{(\fh,v)_\omghd}}
\def\fvo{{( f, v)_\omg}}

\def\gam{{\gamma}}

\def\gamd{{\gamma_\del}}

\def\gams{{\gam_{s}}}

\def\gdd{{g_D^\del}}
\def\gdn{{g_N^\del}}
\def\gnv{{{\lrang{\gdn}{v}}_\omgidn}}

\def\lrang#1#2{\langle{#1},{#2}\rangle}

\def\mcA{{\mathcal A}}
\def\mcB{{\mathcal B}}
\def\mcBo{{{\mathcal B}_0}}
\def\mcBd{{{\mathcal B}_{\delta}}}
\def\mcBi{{{\mathcal B}_\infty}}
\def\mcBis{{{\mathcal B}_\infty^s}}

\def\mcDd{{{\mathcal D}_\del}}

\def\mcI{{\mathcal I}}

\def\mcL{{\mathcal L}}
\def\mcLd{{\mcL_\del}}

\def\mcLdr{{\mcL_{\del,\Omega}}}

\def\mcLi{{\mcL_\infty}}
\def\mcLis{{\mcL_\infty^s}}
\def\mcLz{{\mcL_0}}

\def\mcN{{\mathcal N}}
\def\mcNd{{\mcN_\del}}
\def\mcNdr{{\mcN_{\del,\omg}}}

\def\mcNis{{\mcN_\infty^s}}
\def\mcNi{{\mcN_{\infty}}}
\def\mcQ{{\mathcal Q}}

\def\omg{{\Omega}}
\def\omgb{{\overline{\omg}}}
\def\omgc{{\omg_{c}}}
\def\omgbc{{\omgb_{c}}}
\def\omgbd{{\omgb_{\del}}}

\def\omgh{{\widehat{\omg}}}

\def\omghd{{\omgh_\del}}
\def\omghdn{{\omgh_{\del_n}}}
\def\omgtd{{\omgt_\del}}

\def\omgid{{\omg_{\mcI_\del}}}
\def\omgidtwo{{\omg^2_{\mcI_\del}}}

\def\omgidd{{\omg^D_{\mcI_\del}}}
\def\omgidn{{\omg^N_{\mcI_\del}}}
\def\omgio{{\omg_{\mcI_0}}}
\def\omgii{{\omg_{\mcI_\infty}}}
\def\omgiid{{\omg^D_{\mcI_\infty}}}
\def\omgiin{{\omg^N_{\mcI_\infty}}}
\def\omginid{{\omg_{\mcI^0_\del}}}

\def\omghm{{\omghd^2\setminus \omg^2_{\mcI_\del}}}
\def\omgtd{{\widetilde{\omg}_\del}}
\def\omgti{{\widetilde{\omg}_\infty}}
\def\omgp{{\partial\omg}}
\def\omgpd{{\partial\omg^D}}
\def\omgpn{{\partial\omg^N}}

\def\Rd{{{\mathbb R}^\ddd}}

\def\Ro{{\mathbb R}}

\def\VVcd{{V_{c,\del}(\omghd)}}
\def\VVcdo{{V_{c,\del}(\omg)}}
\def\VVcdn{{V_{c,\del_n}(\omghdn)}}
\def\VVsd{{V^{\star}_{\del}(\omghd)}}

\def\VVdel{{V_{\del}(\omghd)}}

\def\VVt{{V_t(\omgidd)}}
\def\VVtp{{V_t(\omg')}}
\def\VVtd{{V_t(\omgidd)}}
\def\VVtn{{V_t(\omgidn)}}
\def\VVtds{{V_t^{\star}(\omgidd)}}
\def\VVtns{{V_t^{\star}(\omgidn)}}

\def\VVV{{V(\omgh)}}

\def\xb{{\bm x}}

\def\xib{{\boldsymbol\xi}}

\def\xyp{{(\xb,\yb)}}
\def\yxp{{(\yb,\xb)}}
\def\yb{{\bm y}}

\def\ymx{{\yb-\xb}}
\def\ymxa{{|\ymx|}}
\def\ymxap{{(|\ymx|)}}
\def\zb{{\bm z}}

\def\mcAd{{\mcA_\del}}

\def\edoc{\end{document}}

\newtheorem{thm}{Theorem}[section]
\newtheorem{prop}[thm]{Proposition}

\newtheorem{remark}[thm]{{Remark}}

\newtheorem{assu}[thm]{{Assumption}}

\newcommand{\ignore}[1]{}

\usepackage{subfigure}

\title{Nonlocal diffusion models with consistent local and fractional
limits\thanks{This work is to be published in {\em  Approximation, Applications, and Analysis of Nonlocal, Nonlinear Models} (The 50th John H. Barrett Memorial Lectures), 2022.
The research of Q.~Du was supported in part by NSF grant DMS-2012562 and DMS-1937254. The research of X.~Tian was supported in part by NSF grant DMS-2111608. The research of Z. Zhou is partially supported by Hong Kong Research
Grants Council (No. 15303122) and an internal grant of Hong Kong Polytechnic University (Project ID: P0031041, Programme: ZZKS)
}}

\author{Qiang Du\thanks{Department of Applied Physics and Applied Mathematics, and Data Science  Institute,
Columbia University, New York, NY 10027, USA (
{qd2125@columbia.edu}).}
       \and Xiaochuan Tian\thanks{Department of Mathematics, University of California, San Diego, CA 92093, USA  (
{xctian@ucsd.edu}).}
\and Zhi Zhou\thanks{Department of Applied Mathematics. Hong Kong Polytechnic University, 
Kowloon, Hong Kong, China
({zhizhou@polyu.edu.hk}).}
}

\begin{document}

\maketitle

\begin{abstract}
For some
{spatially nonlocal diffusion models} with a finite range of nonlocal interactions measured by a
positive parameter $\delta$,  we review their formulation 
defined on a bounded domain subject to various conditions that correspond  to  some inhomogeneous 
data. We consider their consistency to similar inhomogeneous  boundary value problems of classical
partial differential equation (PDE) models as the nonlocal interaction kernel gets localized in the
local $\delta\to 0$ limit, and at the same time, for rescaled fractional type kernels, to corresponding
inhomogeneous  nonlocal boundary value problems of fractional equations in the global $\delta\to \infty$ limit. Such discussions help to delineate issues related to nonlocal problems defined on a bounded domain with inhomogeneous data. 
\end{abstract}




\pagestyle{myheadings}
\thispagestyle{plain}
\markboth{DU AND TIAN AND ZHOU}{Inhomogeneous Nonlocal Models}

\section{Introduction}\label{intro}

There has been much interest in nonlocal models 
 \cite{amrt:10,bucur2016nonlocal,du19cbms,d2020numerical,gilboa:1005}.
  A characterization of nonlocal models can be found in  \cite{det19cm}, together with a study on the close connections between nonlocal modeling and other mathematical subjects like homogenization, model reduction, and coarse-graining.
 
 Motivated by studies from mechanics to image analysis and from traffic flows of autonomous vehicles to
anomalous diffusion,  a growing area of research is the study on models with nonlocal interactions that span a finite range  \cite{du19cbms,du12sirev}, measured by the horizon parameter $\delta\in\mathbb{R}$, as we will describe in details later.  It should be noted that in such models, the interaction range is generically nonzero (thus nonlocal) but, for problems on a bounded domain, may not be on the same scale as the size of the whole domain. 
In mathematical modeling, nonlocal models with a finite $\delta$,  simply  referred to as nonlocal models here, can play a number of roles \cite{du19cbms}:
\begin{itemize}
\item  They can either complement or serve as an alternative to
 traditional PDE-based local models, and they
 can account for nonlocal interactions explicitly and remain
valid for not only smooth but also singular solutions. Examples include peridynamic  (PD) models for fracture mechanics
\cite{dt-Silling16hbk,dtt18je} and nonlocal models of anomalous diffusion  and transport \cite{du12sirev,
MetzlerKlafter:2000,MeKl04,amrt:10}.

\item They can serve as tools to better understand and simulate existing local PDE models. Popular
numerical methods like Smoothed Particle Hydrodynamics (SPH) and Vortex-Blob methods are good examples of discretizing PDEs through nonlocal smoothing/averaging \cite{CoKo00,mono:05,DuTi19fcm,DuShi21}.

\item They can serve as bridges to build connections between continuum PDEs, discrete algebraic, graph and network models, metric geometry, and fractional differential equations \cite{burago2014graph,Coifman05,Coifman06,du17hbk4,TiDu19,tdg16acm,du19cbms}.
\end{itemize}

 The main goal here is related to the bridging roles mentioned above. That is, the discussion is centered on connecting local PDEs with fractional PDEs
 through nonlocal equations with a finite range of interactions via a unified description. In particular, we reexamine some concepts presented in \cite{du12sirev} on nonlocal diffusion models with a finite $\delta$, and make further investigations so that the theory can be consistent with both the local and fractional limits.
 While the related issues have been subject to earlier investigations, see e.g. \cite{du12sirev,tdg16acm,d2020numerical}, most known results are stated for problems with homogeneous boundary data (or nonlocal constraints).  By presenting problems in more general settings that involve inhomogeneous data, we also provide some insight on how to connect some of the relevant concepts concerning nonlocal boundary conditions (volumetric constraints \cite{du12sirev}). 

\section{Local, nonlocal and fractional models}\label{model}

For illustration, we consider problems defined on
 a bounded, open domain $\omg\subset\Rd$ with a Lipshitz boundary $\omgp$. 
   We adopt a set of notation that is largely similar to those used in \cite{d2020numerical} but with necessary modifications in some cases.
   
    Let the closure of $\omg$ be denoted by $\omgb=\omg\cup\partial \omg$.
The complements are denoted by
  $\omgc=\Rd\setminus\omgb$ and $\omgbc=\Rd\setminus \omg$ respectively. Note that $\omgbc=\overline{\Rd\setminus \omg}
  =\overline{\omgc}$. 

Given the domain $\omg$ and a positive parameter $\delta>0$, we define an {\em interaction layer of width $\delta$} by
\begin{equation}\label{str-inter}
\omgid := \{ \yb\in\omgc \;\mid \; \text{dist}(\yb, \omg) <\delta\,\}.
\end{equation}
We let $\overline{\omgid}$ denote the closure of
$\omgid$.
$\omghd=\omg\cup\omgid$
 and $\omgbd=\overline{\omg\cup\omgid}=\omgb\cup \overline{\omgid}$.
In this work, we are interested in some special cases and limiting regimes  
\begin{equation}\label{str-dcases}
\delta=0, \quad 0\leftarrow \del\ll \text{diam\,}\omg,\quad
\del \approx \text{diam\,}\omg,
\quad \text{diam\,}\omg \ll \del \to \infty, \quad \del = \infty.
\end{equation} 
Note that
 $\omgio=\emptyset$ and $\overline{\omgio}=\omgp$, while $\omgii=\omgc$ and $\overline{\omgii}=\omgbc$.
 
 We also define some domains for any pair of points $\xyp$. These include $\omg^2=\omg\times\omg$, which denotes the tensor product of any domain $\omg$ with itself and
 \[\omgtd=\omghm\quad\text{and}\quad 
\omgti=(\Rd)^2\setminus \omg_c^2.
 \]

\subsection{Local diffusion models}\label{pde}
We first consider a classical Poisson problem 
for an unknown solution  $u \colon  \Omega\subset\Rd\to \mathbb{R}$,
\begin{equation}\label{str-pde}
\begin{cases}
   -\mcLz u(\xb) :=
    -(\Delta u)(\xb)  = F(\xb,u(\xb)), & \quad\forall\,  \xb\in\omg,
\\
        \mcBo u(\xb)  = g_0(\xb), & \quad\forall\,  \xb\in \overline{\omgio}= \omgp,
\end{cases}
\end{equation}
where  $F(\xb, u)$ 
and $g_0(\xb)$ denote given functions defined on $\omg\times \Ro$ and $\partial\omg$, respectively. Here, we elect to work with the Laplacian $\Delta$ rather than more general diffusion operators. 
A special case is that $F(\xb,u)$ is independent of $u$, thus leading to a standard linear PDE on $\omg$.
For the boundary conditions operator $\mcBo$, we have the choices
\begin{equation}\label{str-bc}
\mcBo u(\xb) = \begin{cases}
u(\xb),   & \xb\in \partial\omg^D\; \mbox{(Dirichlet BC)},
\\
\ffrac{\partial u}{\partial \mathbf{n}}(\xb) &\xb\in\partial\omg^N\; \mbox{(Neumann BC)},
\end{cases}
\end{equation}
where $\omgpd$ and $\omgpn$ are subset of $\omgp$ such that
$\omgpd\cup\omgpn=\omgp$ with $\omgpd\cap\omgpn$ having zero $d-1$ dimensional measure. Note that
$\mcBo$ can also take on other forms, (e.g.,
\begin{equation}\label{str-robin}
\mcBo u(\xb) = \ffrac{\partial u}{\partial \mathbf{n}}(\xb) - \alpha u(\xb),\quad \forall
\xb\in\partial\omg^N 
\end{equation}
for a positive constant $\alpha>0$) 
to yield other (e.g., Robin type) boundary conditions.
In the case $\omgpd=\omgp$, we have a pure Dirichlet problem while $\omgpn=\omgp$ leads to a pure Neumann problem. In the latter case, compatibility conditions may be needed on the data, and an additional constraint (e.g., having a zero mean over $\omg$) on the solutions may be imposed to ensure uniqueness.
 
 \begin{remark}
 {Problem} \eqref{str-pde}, which is named the local diffusion problem here, is a well-studied local PDE model used for steady-state normal diffusion.
\end{remark}

\subsection{Fractional diffusion models}\label{fpde}
Fractional PDEs are PDEs with fractional-order partial derivatives of unknown solutions. They are nonlocal integral models with a global or infinite range of interactions.
Here, we consider a nonlocal fractional diffusion model associated with the {\em integral fractional Laplacian} $\mcLis
= - (-\Delta)^{s}${,} which, {for  $u \colon  \Rd\to \mathbb{R}$ and $s\in (0,1)$, is given by}
\begin{align}\label{intfl}
- \mcLis u(\xb)
=  (-\Delta)^{s} u ({\xb})  & =  \int_{\Rd} ( u( {\xb})-u( {\yb}) ) \gams\ymxap \dyb,
  \quad\quad\forall\,  \xb\in\Rd
  \end{align}
where the kernel function is defined by 
\begin{equation}\label{eq:gams}
 \gams(|\zb|)=\ffrac{\cds}{ |\zb|^{\ddd+2s}}, \quad \text{with} \quad  \cds = \ffrac{2^{2s}s\Gamma (s+\gfrac{\ddd}{2})}{\pi^{\ddd/2}\Gamma (1-s)},
\end{equation}
where $\Gamma$ denotes the {gamma} function.  Note that the singular integral in \eqref{intfl} should be interpreted in the {principal value} sense. 

The \emph{fractional diffusion problem} considered here is then given by
\begin{equation}\label{eq:fracPoisson}
\begin{cases}
-\mcLis u(\xb) = F(\xb,u(\xb)), &\quad\forall\, \xb\in\omg{,}
\\
     \mcBis u(\xb)=g_\infty(\xb),&\quad\forall\, \xb\in{\omgii}=\omgc ,
\end{cases}
\end{equation}
where $F(\xb, u)$ is as before and $g_\infty(\xb)$ is a given data defined for $\xb\in\omgii$. The operator $\mcBis$ is defined by
\begin{equation}\label{str-fc}
\mcBis u (\xb) = \begin{cases}
u(\xb),   & \xb\in \omgiid\; \mbox{(Dirichlet VC)},
\\
\mcNis u(\xb),  & \xb\in \omgiin\;  \mbox{(Neumann VC)},
\end{cases}
\end{equation}
where  $\omgiid$ and $\omgiin$ are two disjoint subdomains of
$\omgii$ such that $\overline{\omgiid\cup\omgiin}=\omgbc$.
$\mcNis$ is a suitably defined nonlocal  operator, for example,  given by  \cite{Dipierro:2017},
\begin{equation}\label{eq:mcns}
  \mcNis u({\xb})= - \int_{\omg} (u({\yb})-u({\xb}))  \gams\ymxap  {\dyb},  \quad\quad\forall\,  \xb\in{\omgii=\omgc}.
\end{equation}
In the above definition, $\mcNis$ is defined as an integral over the domain $\omg$ only, instead of the whole space $\Rd$. 

\begin{remark}
Note that the constraint $\mcBis u = g_\infty$ is applied on a domain having a nonzero volume in $\Rd$, thus representing a nonlocal boundary condition which is also called a {\em volume constraint} (VC) \cite{du12sirev}.
\end{remark}

The integral fractional Laplacian is sometimes called the whole space fractional Laplacian as the integral ranges the whole space $\Rd$.
It is a special case of the nonlocal  Laplacian defined later (see  \eqref{str-lap} and related discussions in section \ref{nonlocal}). It is also equivalent to the Fourier representation \cite{Valdinoci:2009}
\begin{align}\label{E:definition_spectral}
(-\Delta)^{s} u(\xb) &= \ffrac{1}{(2\pi)^\ddd} \int_{\Rd} |\xib|^{2s}  ( u , \rme^{-\rmi \xib \cdot \xb} ) \,\rme^{\rmi \xib \cdot \xb} \D \xib = \mathcal{F}^{-1} ( |\xib|^{2s} \mathcal{F}\{u\}(\xib) )(\xb),
\end{align}
where $\mathcal{F}$ denotes the Fourier transform. Formally if we let $s\to 1^-$ in the above, it recovers the Fourier spectral representation of the classical local Laplacian $-\Delta$.  Indeed, it is well known that  $(-\Delta)^s u(x) \rightarrow -\Delta u(x)$ as $s\to 1^-$ for any function $u$ with sufficient regularity; see \eg\ \cite[Theorems~3 and~4]{Stinga:2019} and \cite[Proposition~4.4]{FracGuide:2012}.

\begin{remark}
The fractional Laplacian \eqref{intfl} is often used to model superdiffusion. In probabilistic terms, in contrast to the local Laplacian that corresponds to the normal diffusion described by Brownian motion, the fractional Laplacian can be connected to superdiffusion associated with L\'evy flights in which the length of particle jumps follows a {heavy-tailed} power law distribution, reflecting the long-range interactions between particles, see \eg\ \cite{Dipierro:2017,MetzlerKlafter:2000,Sokolov:2022}.
\end{remark}

The problem \eqref{eq:fracPoisson} is called a fractional diffusion model here to distinguish it from other diffusion models. It is given in the distribution sense so that the equations are only specified over an open domain $\omg$.
More rigorous descriptions of \eqref{eq:fracPoisson} will be discussed later.

Other variants of the operators  $\mcLis$ and $\mcNis$ are discussed in section \ref{sec:regional}. For notation simplicity, we drop the superscript $s$ in $\mcLis$,
$\mcBis$ and $\mcNis$ whenever there is no ambiguity, that is 
 $\mcLis=\mcLi$,
$\mcBis=\mcBi$ and $\mcNis=\mcNi$.
Meanwhile, 
similar to the local case, one can also consider other volume constraints
such as 
$\mcBi u (\xb) = \mcNi u(\xb) - \alpha u (\xb) $ for $\xb \in \omgc$ to mimic a Robin type constraint.

\subsection{Nonlocal diffusion models}
\label{nonlocal}
To distinguish from the local and fractional counterparts, we consider a class of special nonlocal diffusion models here associated with a finite range of interactions, measured by the so-called horizon parameter (interaction radius) $\delta>0$.

For $\del>0$, we consider the following nonlocal problem for a scalar-valued unknown function $u(\xb)$ defined on $\omghd$, 
\begin{equation}\label{str-vcp}
\begin{cases}
  - \mcLd u =  F(\xb, u(\xb)), & \quad\forall\,  \xb\in\omg,
\\
  \mcBd u = g_\del(\xb), &\quad\forall\,  \xb\in\omgid,
\end{cases}
\end{equation}
with $F$ and $g_\delta$ given and a {\em nonlocal diffusion operator} defined by
\begin{equation}\label{str-lap-e}
\mcLd u {(\xb)} := \int_{\Rd} (u(\yb)-u(\xb)) \gamd\xyp \dyb,
\quad\quad\forall\,  \xb\in\Rd,
\end{equation}
where the kernel function $\gamd\xyp$ is symmetric, 
\ie, 
$$\gamd\xyp=\gamd\yxp, \quad \forall \xb, \yb\in \Rd,$$ 
non-negative, and with a compact support in $\ymx$ in $B_\delta({\bf 0})$,
the $\delta$-ball of the origin.
\tie, 
$$\gamd\xyp=0 \text{ whenever } |\ymx|>\del.$$ 
Again, the integral in \eqref{str-lap-e} (and the one in \eqref{str-lap}, which will be introduced shortly) should be understood in the sense of principal value throughout the paper. Moreover, due to the compact support of $\gamd$, for $\xb\in\omg$, we can rewrite \eqref{str-lap-e} as
\begin{equation}\label{str-lap-re}
\mcLd u (\xb) := \int_{\omghd} (u(\yb)-u(\xb)) \gamd\xyp \dyb,
\quad\quad\forall\,  \xb\in\omg.
\end{equation}

Analogous to \eqref{str-bc}, we have the choice of volume constraints 
\begin{equation}\label{str-vc}
\mcBd u (\xb) = \begin{cases}
u(\xb)  & \xb\in \omgidd\; \mbox{(Dirichlet VC)},
\\
\mcNd u(\xb),  & \xb\in \omgidn\;  \mbox{(Neumann VC)},
\end{cases}
\end{equation}
where  $\omgidd$ and $\omgidn$ are two disjoint subdomains of
$\omgid$ such that $\overline{\omgidd\cup\omgidn}=\overline{\omgid}$.
The  operator $\mcNd$ can be defined, for example, analogously to \eqref{eq:mcns} by
\begin{equation}\label{str-flux}
\mcNd u(\xb)  := - \int_\omg (u(\yb)-u(\xb)) \gamd\xyp \dyb,
 \quad\forall\,  \xb\in\omgid. 
\end{equation}
Note that generally, $\mcNd u$ only needs to be defined and applied on
 $\omgidn$. Its definition on $\omgid$ is given for the convenience of notation only.

Rather than considering kernels in the general form, we focus on
those translation-invariant (\ie, functions of $\ymx$) and radial symmetric (\ie, depending only on $\ymxa$).
In particular, we consider a more specific re-scaled form given by
\begin{equation}
 \label{eq:scale}
  \gamd\xyp  =
 C_d(\del) \gamma\biggl(\ffrac{|\ymx|}{\del}\biggr)
 \end{equation}
 for a kernel $\gamma( {|\cdot |})$ that satisfies
\begin{equation}
\label{eq:scal0}
\left\{ \begin{array}{l}
\gamma(s)\geq 0, \text{ and non-increasing, }\forall s\in [0,1), \;\;
 \gamma(s)=0,\forall s\geq 1,\\
\widehat \gamma (|\zb |)=|\zb|^2\gamma(|\zb|)\in L^1_{\text{loc}}(\Rd).
\end{array}\right.
\end{equation}
For such kernels, we see that $\mcLd$ can also be written as
\begin{equation}\label{str-lap}
\mcLd u {(\xb)} :=  C_{d}(\del)\int_{\Rd} (u(\yb)-u(\xb)) \gam \left(\frac{\ymxa}{\delta}\right) \dyb,
\quad\quad\forall\,  \xb\in\Rd.
\end{equation}

The scaling factor $C_d(\del)$ is of special interest in this work and is defined differently depending on the regimes that we study. We give two examples below.
\begin{itemize}
  \item 
For $0<\del\ll \text{diam\,}(\omg)$,
$C_d(\del)$ is chosen to satisfy
\begin{equation}\label{eq:scal1}
 \int_{\Rd} 
 \ymxa^2
 \gamd\xyp d\yb = 
{\del^{d+2}}C_d(\del)  \int_{{B_1({\bf 0})} } 
 |\bm \xi|^2
 \gamma(|\bm\xi|) \D \bm\xi=2d.
\end{equation}
so that as $\del\to 0$, the local limit of $\mcLd$ becomes the standard Laplacian $\mcLz=\Delta$ \cite{du12sirev,du19cbms,md15non,md16na}.
\item For
$\del\gg \text{diam\,}(\omg)$, the factor and kernels are chosen to satisfy
\begin{equation}\label{eq:scal2a}
C_d(\del)=
\frac{C_{d,s}}{\del^{d+2s}}
 \quad\text{for}\quad \gam(t)=t^{-\ddd-2s} \mathbbm{1}_{(0,1)},
\end{equation}
so that $C_d(\del)\gam(t/\del)=C_{d,s}t^{-d-2s}\mathbbm{1}_{(0,\del)}$,
\tie, $\gamd\xyp$ is a truncation of the 
fractional kernel $\gams$ defined in \eqref{eq:gams}
\begin{equation}\label{eq:scal2}
\gamd\xyp=\begin{cases}
\gams\ymxap
=\ffrac{\cds}{\ymxa^{\ddd+2s}}, &
\quad \forall\: \ymx
\in B_\del({\bf 0}),\\
0, & \quad \text{otherwise.}
\end{cases}
\end{equation}
Consequently, we have $\mcLd\to \mcLis=-(-\Delta)^s$ as $\del\to \infty$
\cite{MG:2013,burko1,tdg16acm}.
\end{itemize}

\begin{remark}
Although $\mcLd u(\xb)$ is defined for any  $\xb\in\Rd$, independent of any spatial domain, for problem \eqref{str-vcp}, we only need 
$\mcLd u {(\xb)}$ for $\xb\in\omg$. Then, 
with a compactly supported $\gamd$, the integral in \eqref{str-lap} can be equivalently defined over  $\omghd$ instead, \ie,
\begin{equation}\label{str-lap-equ}
\mcLd u {(\xb)} := \int_{\omghd} (u(\yb)-u(\xb)) \gamd \xyp \dyb,
\quad\quad\forall\,  \xb\in\omg.
\end{equation}
\end{remark}

\begin{remark}
The operator $\mcLd$ defined by \eqref{str-lap}, with a kernel scaled by \eqref{eq:scal1} or \eqref{eq:scal2}, 
is a nonlocal analog of the 
partial differential 
operator $\mcLz=\Delta$ 
or the fractional operator $\mcLis=-(-\Delta)^s$.
Likewise, the nonlocal model \eqref{str-vcp} is a nonlocal analog of the local diffusion model \eqref{str-pde} 
and the fractional diffusion model \eqref{eq:fracPoisson}.
Nonlocal diffusion problems have also been extensively studied in recent years, we refer to the literature review given in \cite{du12sirev} and \cite{d2020numerical}. In addition, one can find additional references in a few more recent works \cite{d2020physically,Foss21,pang2020npinns,zhang2021second}.
\end{remark}

\begin{remark}
We note that, different from \cite{d2020numerical,du13m3as}, the operators $\mcLd$ and
$\mcNd$, 
involving  the respective domains of integration in their definitions, are adopted as in \cite{du19cbms,djlq20sirev}
to be consistent with the fractional version, see e.g. \cite{Dipierro:2017,dyda2019function}, given by
\eqref{eq:mcns}. We also consider their variants, as well as 
the variants of the fractional operators
later in the paper.
\end{remark}

\begin{remark}
\label{rmk:nlrobindef}
Concerning the operator $\mcBd$ in \eqref{str-vc},
one may consider more general forms. For example,
with $\omgid=\omgidd$, we may have
\begin{equation}
\label{nonlocalrobin}
\mcBd u (\xb) = \mcNd u(\xb) -
{\mcDd}
u(\xb), \;\forall  \xb\in \omgid
\end{equation} where the linear operator 
{$\mcDd$} is given by
\begin{equation}\label{str-gvc}
{\mcDd} u(\xb) = \int_{\omgid} u(\yb)\sigma_\delta(\xb,\yb)d\yb, \;
\forall \xb\in \omgid=\omgidd
\end{equation}
where the 
$\sigma_\delta(\xb,\yb)$ can be a bounded  and symmetric kernel
such that as  $\delta\to 0$, 
\begin{equation}
\label{str-robin-vc}
\int_{\omgid} {\mcDd} u(\xb) v(\xb) d\xb \to \int_{\partial\omg} u(\xb) v(\xb) d\xb ,
\end{equation}
for any pair of smoothly defined functions $u$ and $v$ on $\omg$. In this case, we get a problem analogous to the local Robin boundary value problem.
\end{remark}

\begin{remark}
While we focus on nonlocal diffusion operators
 defined by \eqref{str-lap} with a symmetric kernel
 for illustration, we note that more general nonlocal diffusion-convection operators with a nonsymmetric kernel have also been studied, see, for instance, \cite{du14dcdsb,huang15thesis,du15cmame,tjd17cmame,DElia2017nlcd}.
\end{remark}

Note that for the pure Neumann case, \ie, $\omgidn={\omgid}$,
we impose a constraint on the solution $u$ to have mean $0$ over $\omghd$ and the equations in \eqref{str-vcp} are modified to 
\begin{equation}\label{str-vcpn}
\begin{cases}
  - \mcLd u =  F(\xb, u(\xb)) - \lambda_N, & \quad\forall\,  \xb\in\omg,
\\
  \mcNd u = \gdn(\xb) -\lambda_N, &\quad\forall\,  \xb\in\omgid,
\end{cases}
\end{equation}
where $\lambda_N$ is a constant Lagrange multiplier satisfying
\begin{align}\label{str-vcpn-c}
    \int_\omg F(\xb,u(\xb))\,\dxb + \int_\omgid \gdn(\xb) = \lambda_N |\omghd|,
\end{align}
where
$|\omghd|$ 
denotes the nonzero measure of $\omghd$.

\subsection{Regional fractional and nonlocal  Laplacians}\label{sec:regional}

An operator related to the fractional Laplacian, but defined with respect to the domain $\omg$, is the {\em regional fractional Laplacian} given by \cite{BogdanBurdzyChen:2003,ChenKim:2002}
\begin{equation}\label{intflreq}
  (-\Delta) _{\Omega}^{s} u ({\xb})  =     \int_{\omg} (u( {\xb})-u( {\yb}))  \gams\xyp \dyb{,}\quad \forall\, \xb\in \omg.
\end{equation}
It can be used to define a problem given by
\begin{equation}\label{eqn:reg-fractional}
  (-\Delta) _{\Omega}^{s} u ({\xb})  = 
 F(\xb, u(\xb)),\quad \quad\forall\,  \xb\in\omg,
\end{equation}
{which} is another approach to generalize the local Poisson Neumann problem to the nonlocal fractional case \cite[Section 7]{Dipierro:2017}. Note that in comparison with the fractional Laplacian $\mcLis$ defined in \eqref{intfl},
the only difference is that the regional operator in \eqref{intflreq} involves an integral over $\omg$ instead of $\Rd$.

Similarly, a nonlocal problem  on $\omg$  related to the regional fractional Laplacian  is
\begin{equation}\label{eqn:reg-nonlocal}
  - \mcLdr u (\xb) = F(\xb, u(\xb)), \quad  \quad\forall\,  \xb\in\omg,
\end{equation}
where the {\em regional nonlocal diffusion operator} $\mcLdr$ is given by
\begin{equation}\label{eqn:reg-nonlocal-op}
  \mcLdr u(\xb) = \int_{\omg} (u( {\yb})-u( {\xb}))  \gamd \xyp \dyb{,}  \quad  \quad\forall\,  \xb\in \omg.
\end{equation}
Then \eqref{eqn:reg-fractional} is a special case of \eqref{eqn:reg-nonlocal} by choosing special fractional kernels as in \eqref{eq:scal2a}
and setting $\delta=\infty$. Hence, {for kernels of the form
\eqref{eq:scal2a},} we may see $\mcLdr$ as a ``regional" nonlocal Laplacian 
with respect to the domain $\omg$. 

Indeed, let $\omg=\omg_1\cup\omg_2$ with $\omg_1\cap\omg_2=\emptyset$,
the problem \eqref{eqn:reg-nonlocal}, by a simple change of notations, can be equivalently written as the nonlocal problem:
\begin{equation}\label{str-vcr}
\left\{\begin{array}{ll}
  - \mcLdr u(\xb) =  F(\xb, u(\xb)) \quad & \quad\forall\,  \xb\in \omg_1,\\
  \mcNdr  u (\xb) =  G(\xb, u(\xb))  \quad & \quad\forall\,  \xb\in
  \omg_2.\end{array}\right.
\end{equation}
Here, a nonlocal operator $\mcNdr$ is introduced, which
is defined by
\begin{equation}\label{str-fluxt}
   \mcNdr u(\xb)  := -\mcLdr u(\xb)=\int_{\omg} (u(\xb)-u(\yb)) \gamd\xyp \dyb, \quad \forall
   \xb\in
  \omg_2,
\end{equation}
and $G(\xb, u(\xb)) = F(\xb, u(\xb))$ on 
$\omg_2$.
$\mcNdr$ and $G(\xb, u(\xb))$
only carry symbolic meanings in this case and the equivalence of
\eqref{eqn:reg-nonlocal} and \eqref{str-vcr} holds independently of how
$\omg$ is decomposed into $\omg_1$ and $\omg_2$.  For example, 
 a special choice of the domain $\omg_2$, for $\delta$ sufficiently small,  is the interior $\del$-layer
of $\omg$ given by
\begin{align}
    \label{interior-del}
\omginid=\{ \xb\in \omg \:\mid\: \text{dist}(\xb,\omgp) < \del\,\}.
\end{align}

\begin{remark} Note that the definition \eqref{str-fluxt} of $\mcNdr$ has been adopted in discussions in \cite{d2020numerical,du13m3as}, except for the use of different notation for the geometric domains. It can be seen that $\mcNdr$ adopts the same definition as $-\mcLdr$ except that while 
$-\mcLdr u(\xb)$ is applied on $\omg_1$, 
$\mcNdr u(\xb)$ inherits the definition on
$\omg_2$. 
\end{remark}

\begin{remark} We note that the reformulation \eqref{str-vcr} of
\eqref{eqn:reg-nonlocal}  resembles more closely, in appearance, the conventional way of presenting the problem as an equation coupled with a boundary condition (or say volume constraint). Yet, the equivalence of the formulations also highlights the fact that for nonlocal problems, the notion of boundary conditions can be superfluous. A more general way to understand the formulation is that a nonlocal model is largely described by the law of nonlocal interactions that could be altered due to the presence of a boundary.
\end{remark}

\begin{remark}
We also note that the nonlocal diffusion operator $\mcLd$ reformulated in \eqref{str-lap-equ}, with the integral being over $\omghd$ instead of $\Rd$ for compactly supported kernels of an interaction range $\delta$, may have the appearance of being associated with a ``regional" nonlocal 
{diffusion operator}
defined on the combined domain $\omghd$. However, this association can only be meaningful in the case where the equation \eqref{str-lap-equ}, rather than being confined to $\omg$, is extended to all $\xb\in \omghd$.
\end{remark}

Let us examine some limiting cases formally. First, with the kernel $\gamd\xyp$ properly scaled with respect to the horizon parameter $\delta$ as in \eqref{eq:scal1}, we see that  \eqref{eqn:reg-nonlocal}, thus \eqref{str-vcr}, is expected to converge as $\del\to 0$ to a Neumann problem of the local diffusion operator with the corresponding boundary condition determined by the limiting behavior of $F$ on the interior $\del$-layer
$\omginid$ defined in \eqref{interior-del}.

Likewise, with $\gamd\xyp$ chosen as in \eqref{eq:scal2}, then  \eqref{str-vcr} is expected to recover
exactly the regional fractional Laplacian Neumann problem \eqref{eqn:reg-fractional} if $\del>\text{diam}(\omg)$. We thus see the {regional} nonlocal Laplacian again provides a bridge between
the local Laplacian and the regional fractional Laplacian.  

We note that spectral decomposition is another
approach to define fractional differential operators and nonlocal integral operators.
Generically, such spectrally represented operators might not yield operators with a finite range of interactions spatially. In many cases, the resulting nonlocal kernels for their spatial integral formulations might not have simple analytical representations in real space. We thus do not consider further in this direction, interested readers can find additional discussions in  \cite{Abatangelo:2017,bonito:2018,CaSi07,dyda2019function,Nochetto:2015,
stinga2010fractional,Servadei:2014,Antil:2018,Cusimano:2018}.

\subsection{Nonlocal flux and Nonlocal Green's identities}
 In \cite{du13m3as,d2020numerical}, the operator $\mcNdr$ is also used to provide a nonlocal Green's formula that shares more symbolic
resemblance to the classical Green's formula for the local Laplacian and for defining a nonlocal flux.  We provide a more rigorous statement and derivation here for clarity.

\begin{prop}
 \label{prop:green}
Given the nonlocal operators defined by a rescaled kernel
\eqref{eq:scale} with the conditions 
\eqref{eq:scal0} and
\eqref{eq:scal1} satisfied, 
for functions $u$ and $v$ in $C^2(\overline{\omg\cup\omg'})$, we have the nonlocal Green's first identity given by 
\begin{align}\label{str-rgfi0}
& \frac{1}{2}\int_{(\omg\cup\omg')^2} (v(\yb)-v(\xb))(u(\yb)-u(\xb)) \gamd\xyp
\dydx 
=-\int_{\omg\cup\omg'} v(\xb) \mcL_{\del,\omg\cup\omg'} u (\xb)\dxb 
\notag  \\*
&\quad  =-\int_{\omg} v(\xb) \mcL_{\del,\omg\cup\omg'} u (\xb)\dxb + \int_{\omg'} v(\xb) \mcN_{\del,\omg\cup\omg'} u (\xb)\dxb.
\end{align}
This also readily implies the also the nonlocal Green's second identity
\begin{align}\label{str-rgfi}
&\int_{\omg} v(\xb) \mcL_{\del,\omg\cup\omg'} u (\xb)\dxb -\int_{\omg} u(\xb) \mcL_{\del,\omg\cup\omg'} v (\xb)\dxb  \notag \\*
&\quad  = - \int_{\omg'} u (\xb) \mcN_{\del,\omg\cup\omg'} v (\xb)\dxb + \int_{\omg'} v(\xb) \mcN_{\del,\omg\cup\omg'} u (\xb)\dxb.
\end{align}
\end{prop}

\begin{proof} To prove the Proposition \ref{prop:green}, the key is to show that the integrals are well-defined so that the classical Fubini theorem can be applied to get \eqref{str-rgfi0}. For 
$u\in C^2(\overline{\omg\cup\omg'})$, we now check that 
$\mcL_{\del,\omg\cup\omg'} u \in L^1(\omg\cup\omg')$, which is a sufficient condition to achieve this. To simplify the notation, we use $\omg$ to replace $\omg\cup\omg'$ so that
\begin{align*}
\mcL_{\del,\omg} u (\xb) = \int_\omg (u(\yb) -u(\xb)) \gamd\xyp  d\yb. 
\end{align*}
Then it suffices to show that $\mcL_{\del,\omg} u \in L^1(\omg) $ for all $u\in C^2(\overline{\omg})$. 
 
By the regularity of $u$,  we have for $\xb\in \omg$, 
\[
\begin{split}
\mcL_{\del,\omg} u (\xb) &= \int_{\omg \backslash B_{\dist(\xb,\partial\omg)} (\xb)} (u(\yb) -u(\xb)) \gamd\xyp d\yb \\
&\qquad + \int_{B_{\dist(\xb,\partial\omg)} (\xb)} (u(\yb) -u(\xb)) \gamd(|\yb-\xb|) d\yb \\
&\leq C \int_{\omg \backslash B_{\dist(\xb,\partial\omg)} (\xb)} |\yb-\xb| \gamd(|\yb-\xb|) d\yb \\
&\qquad  + C  \int_{B_{\dist(\xb,\partial\omg)} (\xb)} |\yb-\xb|^2 \gamd(|\yb-\xb|) d\yb \\
& \leq C \left(\int_{\dist(\xb,\partial\omg) < |\zb| <\del } |\zb|\gamd(|\zb|) d\zb  +1\right)
\end{split}
\]
where the constant $C$ depends on $\| u\|_{C^2(\overline{\omg})}$. Thus, the task remains is to show that \[
w(\xb):=\int_{\dist(\xb,\partial\omg) < |\zb| <\del } |\zb|\gamd(|\zb|) d\zb  \in L^1(\omg)\,.\] 

Let $\omg^\epsilon$ denote an inner layer of width $\epsilon>0$ surrounding $\partial\omg$ where $\epsilon >0$ is a small fixed number.
Since it is obvious that 
$w\in L^1(\omg\backslash \omg^\epsilon)$, we only need to show 
$w \in L^1(\omg^\epsilon)$.  Define for $t\in (0,\epsilon)$,
\[
\partial \omg_t := \{ \xb\in \omg ~| ~\dist(\xb,\partial\omg) = t\}. 
\]
Then, there exists a general constant $C>0$,
\[
\begin{split}
\int_{\omg^\epsilon}\int_{\dist(\xb,\partial\omg) < |\zb| <\del } |\zb|\gamd(|\zb|) d\zb & \leq  C \int_0^\epsilon \int_{\partial\omg_t} \int_{t< r <\del } r^d \gamd(r) dr d\xb dt \\
&\leq C \int_0^\epsilon \int_{t< r <\del } r^d\gamd(r) dr dt\\
& =  C \int_0^\del \int_0^{\min(r,\epsilon) }   r^d\gamd(r) dtdr \\
& \leq  C \int_0^\del   r^{d+1}\gamd(r) dr   < \infty. 
\end{split}
\]
Thus, we get the desired result that 
$\mcL_{\del,\omg} u \in L^1(\omg) $. 
\end{proof}

\begin{remark}
Note that by density argument, one can show that the above identities hold in the closure of $C^2(\overline{\omg\cup\omg'})$ with respect to suitable norms.
\end{remark}

\begin{remark}
In \cite{du13m3as} and \cite{du19cbms}, the identity 
 \eqref{str-rgfi} is stated as
for $\omg=\omg_1\cup\omg_2$,
\begin{align} \label{str-fgfi2re}
\int_{\omg} v(\xb) \mcLdr u (\xb)
\dxb =\int_{\omg}  u(\xb) \mcLdr v (\xb) \dxb.
\end{align}
This relation, as a consequence of Fubini's theorem, can be seen as
a nonlocal version of integration by parts. 
\end{remark}

It is worth pointing out that a similar role  as
 $\mcNdr$ can be played by the operator $\mcNd$ defined earlier in \eqref{str-flux}. In fact, corresponding to the function $u$ defined on a union $\omg\cup\omg'$ of domains $\omg$ and $\omg'$, let $\mcNd u (\xb)$ be defined on $\omg'$ in the same way as \eqref{str-flux}, the so-called nonlocal flux of $u$ 
 between the domain $\omg$ and another domain $\omg'$, which is induced from
$\mcNd$, can be specified as
\begin{align*}
  \mathcal{F}(\omg,\omg')
  =\int_{\omg'} \int_{\omg}
(u(\xb)-u(\yb)) \gamd\xyp 
  \,d\yb \, d\xb = \int_{\omg'}
  \mcNd u (\xb)\, d\xb  .
\end{align*}
Then for a symmetric kernel $\gamd\xyp$, \ie, $\gamd\xyp=\gamd\yxp$, we have the following properties.
\begin{itemize}
    \item Anti-symmetry:
\begin{align*}
  \mathcal{F}(\omg,\omg')= - \mathcal{F}(\omg',\omg)
=  - \int_{\omg}\int_{\omg'}
(u(\yb)-u(\xb)) \gamd\yxp \,\dxb \,\dyb,\quad
\end{align*}
holds formally (assuming that Fubini's theorem is applicable).

\item Additivity:
\begin{align*}
\mathcal{F}(\omg_1,\omg')
+\mathcal{F}(\omg_2,\omg')
=\mathcal{F}(\omg,\omg')\quad \text{for }\; 
\omg=\omg_1\cup\omg_2, \; \omg_1\cap\omg_2=\emptyset.
\end{align*}
\end{itemize}

\begin{remark}
With $\mathcal{F}(\omg_1,\omg_1)=0$, we see that 
\begin{align*}
  \mathcal{F}(\omg_1,\omg_2)
 & = \mathcal{F}(\omg_1,\omg_2) +  \mathcal{F}(\omg_2,\omg_2)
 =\mathcal{F}(\omg_1\cup\omg_2,\omg_2)\\
&  =\int_{\omg_2} \int_{\omg}
(u(\xb)-u(\yb)) \gamd\xyp 
  \,d\yb \, d\xb = \int_{\omg_2}
  \mcNdr u (\xb)\, d\xb.
\end{align*}
This also implies that $\mcNd$ induces the same flux functional as $\mcNdr$ for any symmetric kernel. 
\end{remark}

Next, let us state 
the nonlocal Green's first identity given by 
\begin{align}\label{str-gfi0}
& \frac{1}{2}\int_{(\omg\cup\omg')^2\setminus \omg'^2} (v(\yb)-v(\xb))(u(\yb)-u(\xb)) \gamd\xyp
\dydx \notag  \\*
&\quad  =-\int_{\omg} v(\xb) \mcL_{\del,\omg\cup\omg'} u (\xb)\dxb + \int_{\omg'} v(\xb) \mcNd u (\xb)\dxb.
\end{align}
In particular, if $\omg'=\omgid$, then with $\omgtd=\omghm$, we have
\begin{align}\label{str-gfi1}
& \frac{1}{2}\int_{\omgtd} (v(\yb)-v(\xb))(u(\yb)-u(\xb)) \gamd\xyp
\dydx \notag  \\*
&\quad  =-\int_{\omg} v(\xb) \mcLd u (\xb)\dxb + \int_{\omgid} v(\xb) \mcNd u (\xb)\dxb.
\end{align}

Note that the identity \eqref{str-gfi1} is different from \eqref{str-rgfi0} as given in \cite{du13m3as,d2020numerical}, but it is a more natural nonlocal analog of the (generalized) local Green's first identity
\begin{align}\label{str-gflo}
\int_\omg \nabla v (\xb)\cdot \nabla u(\xb)  \dxb = -\int_\omg   v(\xb)\Delta u (\xb) \dxb + \int_{\partial\omg} v(\xb) \frac{\partial u}{\partial \mathbf{n}} (\xb)  \dxb,
\end{align}
as well as the fractional version given by
\begin{align}\label{str-gfi1f}
& \frac{1}{2}\int_{\omgti} (v(\yb)-v(\xb))(u(\yb)-u(\xb)) \gams\ymxap
\dydx \notag  \\*
&\quad  =-\int_{\omg} v(\xb) \mcLis u (\xb)\dxb + \int_{\omgc} v(\xb) \mcNis u (\xb)\dxb,
\end{align}
where $\omgti=(\Rd)^2\setminus \omg_c^2$.

For smooth function $u$, by simple Taylor expansion, we have $\mcLd u (\xb) \to \Delta u (\xb)$ in $\omg$ if we choose the kernels according to \eqref{eq:scal1}. In this case, 
from the first Green's identity \eqref{str-gfi1}, we may formally
let $\delta \to 0$ to get that
\begin{align}\label{eq:flux-conv}
\int_{\omgid} v(\xb) \mcNd u (\xb)\dxb \to \int_{\partial \omg} v(\xb) \frac{\partial u}{\partial \mathbf{n}} d\xb.
\end{align}
This offers some justification to the statement that the local version \eqref{str-gflo} of \eqref{str-gfi1f} can be seen as the $\del\to 0$ limit of \eqref{str-gfi1}. 
In this case, from the above formal argument, we expect that
the Neumann operator $\mcNd$ converges to the local one in the distribution sense.
Meanwhile,
the fractional version can be seen as the $\del\to \infty$ limit if the kernel is given according to \eqref{eq:scal2}. 
{
Furthermore, as $\delta\to 0$, formally  we may also note that
\begin{align}
\label{eq:flux-conv-reg}
\int_{\omgid} v(\xb) \mcN_{\delta,\omghd} u (\xb)\dxb \to \int_{\partial \omg} v(\xb) \frac{\partial u}{\partial \mathbf{n}} d\xb.
\end{align}
Indeed, the difference between the left-hand side integrals of \eqref{eq:flux-conv} and \eqref{eq:flux-conv-reg} satisfies 
\begin{align*}
&\int_{\omgid} v(\xb) \int_{\omgid} (u(\xb)-u(\yb)) \gamma_\delta(|\xb-\yb|) \dyb \dxb\\
&\quad
=\frac12  \int_{\omgid}  \int_{\omgid}(u(\xb)-u(\yb))
(v(\xb)-v(\yb)) \gamma_\delta(|\xb-\yb|) \dyb \dxb
=O(\delta)
\end{align*}
for smooth functions $u$ and $v$ defined on the whole space. This observation shows the formal consistency using the different Neumann operators in the local limit. However,
 We note that consistency to the fractional limit, unlike the case of $\mcNd$ is not shared by the nonlocal analog presented in \cite{du13m3as,d2020numerical} based on the operator $\mcN_{\delta,\omghd}$.
}

Likewise, the nonlocal analog of
\[ \int_\omg \left[
v(\xb) \Delta u (\xb)
-  u(\xb)\Delta v (\xb) \right]
\dxb = \int_{\partial\omg} \left[
v(\xb) \frac{\partial u}{\partial \mathbf{n}} (\xb)
- u(\xb) \frac{\partial v}{\partial \mathbf{n}} (\xb)\right]
\dxb
\]
is the nonlocal Green's second identity given by
\begin{align}\label{str-gfi2}
\int_{\omg} [ v(\xb) \mcLd u (\xb) -  u(\xb) \mcLd v (\xb)] \dxb  = \int_{\omgid} [v(\xb) \mcNd u (\xb) - u(\xb) \mcNd v (\xb)] \dxb. 
\end{align}
The factional analog can be easily obtained
in the limit $\delta\to \infty$:
\begin{align}\label{str-fgfi2}
\int_{\omg} [ v(\xb) \mcLis u (\xb) -  u(\xb) \mcLis v (\xb)] \dxb  = \int_{\omgc} [v(\xb) \mcNis u (\xb) - u(\xb) \mcNis v (\xb)] \dxb.
\end{align}

\begin{remark}
Rigorous proofs of the nonlocal Green's identities and their
connections to the local and fractional analog in their full generality depend on the characterization of the trace spaces, and we will leave the discussions to future works, though we note some related works in the fractional case \cite{dyda2019function} and recent works on the nonlocal trace theorems in \cite{Trace21}.
\end{remark}

\section{Well-posed inhomogeneous nonlocal variational problems}\label{varia}

We use linear problems where $F(\xb,u)=f(\xb)$ as illustrations for well-posed nonlocal models with a horizon parameter $\del$ and inhomogeneous nonlocal constraints.
We consider weak formulations corresponding to \eqref{str-vcp} and make connections to a minimization principle.  We pay particular attention to the results that are independent of $\del$ so that the local and fractional cases can then become special cases, which enforces the generality of the nonlocal models.

\subsection{Function spaces and norms}\label{sec:spanorm}

Given a generic domain $\omg$, a parameter $\delta>0$ and the extended domain $\omghd$,
let us first define {an} {\em `energy' space} 
\begin{equation}\label{weak-espace}
   \VVdel =  \{  v(\xb)\in L^2(\omghd) \colon \| v \|_{\VVdel} <\infty\},
\end{equation}
where $\| \cdot \|_{\VVdel} $ denotes an {\em `energy' norm} induced by the inner product  
\[
{\langle \cdot,\cdot\rangle_{\VVdel}} = (\cdot,\cdot)_{\VVdel} 
+ (\cdot, \cdot)_{\omghd}\] 
with 
$(\cdot,\cdot)_{\VVdel} $ being defined as:
\begin{align}\label{weak-ip}
  (u,v)_{\VVdel} 
  =  \frac{1}{2}
  \int_{\omgtd} (v(\yb)-v(\xb))  (u(\yb)-u(\xb))   \gamd\xyp  \,\dydx\,,
\end{align}
for any $u,v\in\VVdel$. Here $\omgtd=
\omghm$.

\begin{remark}
A different choice of the nonlocal space and its norm, defined over a domain $\omgh$, corresponds to 
the inner product
\begin{align}
    \label{weak-nnorm}
    {\langle u, v \rangle}_{\VVV}  =
   (u, v )_{\VVV} +(u,v)_{\omgh}
\end{align}
 where
\begin{align}
    \label{weak-nnorm-1}
(u, v )_{\VVV}
= \frac{1}{2}
\int_{\omgh}\int_{\omgh} 
(u(\yb) -u(\xb)) (v(\yb) -v(\xb))   \gamd\xyp  \dydx 
\end{align}
Such space and norm can be associated with the regional version of the nonlocal and fractional versions of diffusion discussed in Section \ref{sec:regional}. 
When $\omgh$ is replaced by $\omghd$, the main difference with the norm of $\VVdel$ induced from \eqref{weak-ip}
is the inclusion of additional nonlocal interactions in $\omgidtwo$
for the norm in $V(\omghd)$. Note that this difference disappears for functions constrained to have zero values on $\omgid$.
\end{remark}

The space $\VVdel$ is by definition a subspace of $L^2(\omghd)$
and 
can be shown, for suitably chosen kernels like those specified by \eqref{eq:scal1} or \eqref{eq:scal2},  to be a Hilbert space.

The {\em constrained energy space} is a closed subspace of $\VVdel$
defined by
\begin{equation}\label{weak-espacec}
   \VVcd =  \{  v(\xb)\in \VVdel \colon  v(\xb) = 0 \ \mbox{for $\xb\in\omgidd$} \}.
\end{equation}
Note that if $\omgidd=\emptyset$, we may pick a nontrivial subset $K$ of $\omghd$ 
to define the space by 
space $ \VVcd$ by
\begin{equation}\label{weak-espacen}
   \VVcd =  \{  v(\xb)\in \VVdel\: \colon \:  \int_K v(\xb) d\xb  = 0 \;\}.
\end{equation}
Some natural choices include $K=\omg$ and $K=\omghd$.
The dual space of $\VVcd$ is denoted by $\VVsd$ with $(\cdot,\cdot)_\omghd$ being the duality pairing which, for functions vanishing on $\omgidd$, is also used to denote the equivalent $L^2(\omg\cup\omgidn)$ inner product.

For convenience, we also define the energy seminorm
\begin{equation}\label{weak-enormn}
| v |_{\VVcd} = \biggl( \frac{1}{2}
\int_{\omgtd}  (v(\yb)-v(\xb)) ^2  \gamd\xyp 
 \dydx \biggr)^{1/2},
\end{equation}
which can be a full norm on $\VVcd$ as shown later.

\begin{remark}
A natural question is to ask what assumptions on $\omgidd$ and $\omgidn$ and $\gdd$ and $\gdn$ are needed to recover the appropriate local limit. For example,
under what assumptions on $\omgidd$ and $\omgidn$, one can have bounded extension of functions $\VVcd$ with respect to the energy, in the same form but over some larger domain, independent of $\delta$ for $\delta$ small.
Sometimes, it is possible that the above extension cannot be made: for example, if $\omgidd$ is an inner 
layer surrounded by $\omgidn$. For an attempt to put assumptions on $\omgidd$, $\omgidn$, $\gdd$ and $\gdn$, we refer to the discussion in Section \ref{local-limit}.
\end{remark} 

\paragraph{Compactness and equivalent norms} 
The work of Bourgain, Brezis and Mironecu \cite{BBM01} characterizes the limiting space of the energy space $\VVdel$ as $\delta\to0$ when the kernel $\gamd$ satisfies 
\eqref{eq:scal0} and \eqref{eq:scal1}. Compactness is gained through taking the limit $\delta\to0$. The original result in \cite{BBM01} requires additionally that $\gam(t)t^2$ is a nonincreasing function for $t\in (0,1)$. It was found later that the assumption could be replaced with $\gam(t)$ being nonincreasing instead, see e.g., discussions in \cite{md15non,td15sinum} and \cite{dmt22compact}.  \cite{ponce2004} showed a different proof of the Bourgain-Brezis-Mironescu compactness result which removed the nonincreasing assumption on $\gam$ for $d\geq 2$. We quote the compactness result below. 
\begin{lemma}\label{lem:compact1}
Assume that $\del_n\to 0$ and the kernels $\{\gamma_{\delta_n}\}_{n\in \mathbb{N}_{+}}$ satisfy \eqref{eq:scal0} and \eqref{eq:scal1}. Let $\omg_1 \subset \Rd$ be a bounded Lipschitz domain.  If $\{v_n \}\subset L^2(\omg_1)$ is a bounded sequence such that
\[ \int_{\omg_1} \int_{\omg_1} \gamma_{\delta_n}(\xb,\yb) (v_n(\yb) -v_n(\xb))^2 \dxb\dyb \leq C_0,\]
 then $\{v_n\}$ is relatively compact in $L^2(\omg_1)$ and any limit function $v$ belongs to $H^1(\omg_1)$ with
\[
\|v \|_{H^1(\omg_1)} \leq C_0 .
\]
\end{lemma}

\begin{remark} It is possible to have more general compactness results for a sequence of kernels $\{\gamma_{\delta_n}\}_{n\in \mathbb{N}_{+}}$ converging to a limit $\gamma_\delta$, as $\delta_n\to \delta$, under suitable conditions on these kernels and the corresponding energy spaces $\VVcdn$ and $\VVcd$. We refer to \cite{td15sinum,dmt22compact}.
\end{remark}

\begin{remark}
Note that by classical results of fractional Sobolev spaces, for a kernel $\gamma_\del$ that satisfies \eqref{eq:scal2}, we have the
pre-compactness of $\VVcd$ in $L^2(\omgh)$ for $\delta>0$.
\end{remark}

Using the compactness result, we can prove a general Poincar\'e inequality  \cite{du12sirev,du13dcdsb,du19cbms,mengesha2012} for kernels satisfying either \eqref{eq:scal0}-\eqref{eq:scal1} or \eqref{eq:scal2} as $\delta\to 0$ and $\delta\to \infty$ respectively. 

\begin{lemma}\label{lem:poincare}
There exists a constant $C>0$, independent of $\del$, such that
\begin{equation}\label{poincare-uni}
\|v\|_{L^2(\omghd)} \leq C |v |_{\VVcd},\quad \forall\; v\in \VVcd.
\end{equation}
\end{lemma}

Lemma \ref{lem:poincare} unifies the classical version of the Poincar\'e inequality for the usual energy function spaces associated with the local Laplacian,  the fractional and the nonlocal versions.

\begin{remark}
 We see that $|\cdot|_{\VVcd}$
is a norm on $\VVcd$ equivalent to the norm $\|\cdot\|_{\VVcd}$ uniformly in $\del$ for kernels specified in 
\eqref{eq:scal1} and \eqref{eq:scal2}
as $\delta\to 0$ and $\delta\to \infty$ respectively. Correspondingly, we may view $\VVcd$ as a Hilbert space equipped with the inner product
\begin{align}\label{weak-ip1}
  (u,v)_{\VVcd} 
  =  \frac{1}{2}
\int_{\omgtd}  (v(\yb)-v(\xb))  (u(\yb)-u(\xb))   \gamd\xyp
 \dydx
\end{align}
for all $u,v\in\VVcd$.
\end{remark}

For more detailed discussions of the conditions on the kernels and rigorous proofs of the properties on these spaces, we refer to
  \cite{md15non,md16na,du19cbms}.

For any subdomain $\omg'\subset \omgid$ with a nonzero volume (\ie, $d$ dimensional measure),  we define the
space $\VVtp$,
\begin{equation}\label{weak-trace}
\VVtp = 
\{ v = w|_{\omg'} \; \text{for some}\; w\in\VVdel \},
\end{equation} 
which involves restrictions to the domain $\omg'$ having nonzero volume in $\Rd$. A norm on $\VVtp$ can be  defined by
\begin{align}\label{weak-normt}
\| v \|_{\VVtp} 
  &= \inf\{ \|w\|_{\VVdel}\; |\; w\in\VVdel,  w|_{\omg'}=v.\}
\end{align}

\begin{remark}
  One can view $\VVtp$, induced by $\VVdel$, as a nonlocal analogue of the trace space $H^{1/2}(S)$ induced by $H^1(\omg)$ for $S\subset \partial \omg$. Some rigorous characterizations of the trace spaces for special classes of kernels can be found in \cite{Trace21}.
 \end{remark}
 
 We are particularly interested in the 
  spaces $\VVtd$ and $\VVtn$ and their dual spaces
  $\VVtds$ and $\VVtns$  corresponding to 
  cases $\omg'=\omgidd$ and $\omg'=\omgidn$ respectively. We use $\gnv$ to denote the inner product of $\gdn$ and $v$ in $L^2(\omgidn)$.

\subsection{Weak formulations and {well-posedness}}
For illustration, we consider the case {$F(\xb,u)=-q(u)+f(\xb)$} where $q(u)=Q'(u)$
with $Q=Q(u)=\alpha |u|^m/m$  with some $m\in (2, \min\{4, 2d/(d-2)\})$ and $\alpha\geq  0$.
We note that the problem reduces to a linear one if $\alpha=0$.
For any $u\in L^m(\omg)$, we define a linear functional 
$$\mcQ(u,v)=\int_\omg q(u(\xb)) v(\xb)  d\xb, \quad\forall v \in L^m(\omg).$$ 

Following \cite{du19cbms}, we consider
a weak formulation {of problem} \eqref{str-vcp} with $\mcBd u={g_\del}$. Let $\gdd$ and $\gdn$ denote the restriction of ${g_\del}$ on $\omgidd$ and $\omgidn$ respectively for domains that have nonzero volume
($d$ dimensional measure). We define
\begin{equation}
    \label{eq:fhat}
    \fh(\xb)=\begin{cases}
    f(\xb), &\quad \xb\in \omg,\\
    0, &\quad \xb\in \omgidd,\\
    \gdn(\xb), &\quad \xb\in \omgidn.
    \end{cases}
\end{equation}
We assume that $\fh\in\VVsd$, which also implies $\gdn\in\VVtns$. Moreover, we assume that $\gdd\in \VVtd$.
By the definition of $\VVtd$, we can define an extension of $\gdd$ on $\omghd$, denoted by $w_D^\del$, such that 
\begin{align}
\label{eqn:extd}
w_D^\del\in \VVdel, \quad w_D^\del\mid_{\omgidd}=\gdd, \quad \|w_D^\del\|_{\VVdel} \leq C \|\gdd\|_{\VVtd},
\end{align}
for some constant $C>0$, independent of $\delta$. Without loss of generality, we can set $C=2$. {To deal with the special nonlinear term, we make an additional assumption
that $w_D^\del$ also satisfies
\begin{align}
\label{eqn:extdp}
 \alpha w_D^\del\in \alpha L^m(\omg), \quad  \quad \mcQ(w_D^\del,w_D^\del)  \leq \tilde{q}(\|\gdd\|_{\VVtd}),
\end{align}
for a given smooth function $\tilde{q}=\tilde{q}(r)$.
Here we use the abbreviation $\alpha L^m(\omg)$ as the space of all functions $v$ such that $\alpha v \in L^m(\omg)$, \ie, 
$\alpha L^m(\omg)$ can include all functions if $\alpha=0$ but 
$\alpha L^m(\omg)=L^m(\omg)
$ if $\alpha\neq 0$.
}

The weak formulation can then be stated as follows: given $\gdd\in\VVtd$,
$\fh\in\VVsd$, and $w_D^\del$ satisfying \eqref{eqn:extd}
and \eqref{eqn:extdp},
find $u$ such that $u-w_D^\del\in\VVcd \cap \alpha L^m(\omg)$ and
\begin{equation}\label{weak-weakf}
{\mcAd}(u,v) + {\mcQ(u,v) }
= \fhvohd ,  \quad \forall  v\in\VVcd {,}
\end{equation}
where the bilinear form  ${\mcAd}(\cdot,\cdot)$ is given by
\begin{align}\label{weak-bform}
  &{\mcAd}(u,v) =  (u,v)_{\VVcd} ,\quad
\forall\: u, v\in\VVdel,
\end{align}
and the linear functional
\begin{equation}\label{weak-lfunc}
\fhvohd=\int_{\omg} f(\xb)v(\xb) \dxb  + \int_{\omgidn} \gdn(\xb)v(\xb) \dxb \,.
\end{equation}

Note that if $\omgidd=\omgid$, then for any $ v\in\VVcd$, we have
\begin{equation}
\label{weak-bform-equ}
\begin{aligned}
  {\mcAd}(u,v) =  (u,v)_{\VVcd} 
 & =\frac{1}{2}\int_{\omgtd}  (v(\yb)-v(\xb))  (u(\yb)-u(\xb))  \gamd\xyp  \dydx\\
  & =\frac{1}{2}\int_{\omghd} \int_{\omghd} 
    (v(\yb)-v(\xb))  (u(\yb)-u(\xb))  \gamd\xyp
 \dydx
\end{aligned}
\end{equation}

{The weak formulation is equivalent to a minimization problem:
Given $\gdd\in\VVtd$,
$\fh\in\VVsd$, and $w_D^\del$ satisfying \eqref{eqn:extd}
and \eqref{eqn:extdp},
find $u$ that minimizes
\begin{equation}\label{weak-min}
\min_{v\in w_D^\del+ \VVcd\cap \alpha L^m(\omg)}
 E(v):={\frac{1}{2}
\mcAd}(v,v) +  \frac{1}{m} \mcQ(v,v) -\fhvohd , 
\end{equation}
}

\begin{theorem} 
The {problem} \eqref{weak-weakf}   has a unique solution
$u\in w_D^\del+\VVcd\cap \alpha L^m(\omg)$,
 for any given data $\fh$,  $\gdd$ and 
 $w_D^\del$ satisfying the conditions specified earlier.
 Moreover, there is a constant $C>0$, independent of $\del$, such that the following {\it a~priori} estimates for the solution $u$ are satisfied:
 {
\begin{equation}\label{weak-est}
 \| u \|_\VVdel^2
+ \mcQ (u, u) \le C [ \|\fh\|_\VVsd ^2+  \| \gdd\|_\VVt^2 + \tilde{q}(\|\gdd\|_\VVt) ].
\end{equation}}
\end{theorem}

\begin{proof}
Using the Poincar\'{e} inequality 
 in Lemma \ref{lem:poincare}, we see that the functional is coercive and continuous in the set $w_D^\del+\VVcd\cap\alpha L^m(\omg)$ with respect to the norm
 $\|\cdot\|_{L^2_0(\omg)}+ \alpha \|\cdot\|_{L^m(\omg)}$.
 Thus, there exists a minimizing sequence
 which remains uniformly bounded in  
 $w_D^\del+\VVcd\cap\alpha L^m(\omg)$. One can take a uniformly bounded and weakly convergent  
 sequence  $v_n=w_n + w_D^\del$, $w_n \in \VVcd\cap\alpha L^m(\omg)$ such that
 $$\lim_{n\to \infty} E(v_n)  = \inf_{v\in w_D^\del+ \VVcd\cap \alpha L^m(\omg)} E(v)> -\infty.$$
 Denote $u$ the limit of $v_n$, then $u \in  w_D^\del+ \VVcd\cap \alpha L^m(\omg)$, and
  we may use the weak lower semicontinuity of the functional $E=E(v)$ to get
 $$
  E(u) \leq \underline{\lim}_{n\to \infty}   E(v_n)  =   \inf_{v\in w_D^\del+ \VVcd\cap \alpha L^m(\omg)} E(v),
 $$
 which gives $u$ as the minimizer of $E(v)$ in $ w_D^\del+ \VVcd\cap \alpha L^m(\omg)$. One can easily see that the
 weak formulation is simply the Euler-Lagrange equation. We thus have 
 the existence of the solution. Since $E(v)$ is strictly convex, we also get the uniqueness of the solution.
 The {\it a~priori} estimates can be obtained from the weak formulation. Indeed, 
 \[
{\mcAd}(u,u-w_D^\del) + {\mcQ(u,u-w_D^\del) }
= (\widehat{f}, u-w_D^\del)_{\omghd} . 
\]
 {Then Young's inequality implies
\begin{align*}
\frac{1}{\alpha}Q(u,u-w_D^\del) & =\int_{\omg} u^{m}(\xb) d\xb-\int_{\omg} u^{m-1}(\xb) w_D^\del(\xb) d\xb\\
& \geq  \int_{\omg}  u^{m}(\xb) d\xb- \int_{\omg}  \frac{m-1}{m} u^{m}(\xb) +\frac{1}{m} (w_D^\del)^{m}(\xb) d\xb\\
& = \frac{1}{m} \int_{\omg} u^{m}(\xb) d\xb-\frac{1}{m}\int_{\omg} (w_D^\del)^{m}(\xb) d\xb .
\end{align*}
}
Note that for the linear case with $\alpha=0$, we can also use
the Lax--Milgram theorem to show the well-posedness of the  {problem} \eqref{weak-weakf}.
 \end{proof}

\begin{remark}[{a pure nonlocal Neumann problem}]\label{rem:neumann}\
For the case $\omgidn=\omgid$ and $\alpha=0$, \tie, a pure linear Neumann problem, due to the compatibility condition imposed on $\VVcd$, we can recover the original strong form
\eqref{str-vcp} from the linear functional $\fhvoh$ with
the following compatibility condition:
\begin{align}\label{weak-compat}
\int_{\omghd} \fh(\xb) \dxb = \int_{\omg} f(\xb) \dxb + \int_{\omgid} \gdn(\xb)\dxb =0. 
\end{align}
Meanwhile, for the nonlinear problem \eqref{str-vcpn}, the
condition is equivalently \eqref{str-vcpn-c}.
Similarly, a compatibility condition is also required for the ``regional" nonlocal diffusion problem given in \eqref{eqn:reg-nonlocal} or the equivalent form in \eqref{str-vcr}:
\begin{align}\label{weak-compat-reg}
\int_{\omg}  F(\xb, u(\xb)) d\xb = \int_{\omg_1}  F(\xb, u(\xb)) d\xb
+ \int_{\omg_2}  G(\xb, u(\xb)) d\xb = 
0
\end{align}
where $\omg_1$ and $\omg_2$ represent a decomposition of $\omg$ as in \eqref{str-vcr}.
The equation \eqref{weak-compat-reg} follows easily from the symmetry assumption on the nonlocal interaction kernel.
\end{remark}

\begin{remark}
 From the discussion here, we see that in the weak formulation, different parts of the data $\fh$ play different roles, in particular when considering the limiting regimes of $\delta\to 0$ or $\delta\to \infty$.
 That is, $f$ symbolically corresponds to the right-hand side of the equation while $\gdn$ corresponds to the Neumann data.
  \end{remark}

\begin{remark}
We note that other ways of defining a nonlocal Neumann problem 
  {can also be used.}
For additional
discussions about nonlocal Neumann problems, see \eg\  \cite{cortazar,amrt:10,du19cbms,md15non,md16na,shi2017convergence,li2019point,ttd17amc,d2020physically,Dengw:2018}.
In addition, we may also consider problems associated with the regional nonlocal and fractional Laplacian defined by \eqref{eqn:reg-nonlocal-op} and\eqref{intflreq} on the domain $\omg$.
{In particular, the bilinear form can be modified as 
\begin{align}\label{weak-ip1-regional}
\mcA_{\delta,\omg} (u,v)=
(u,v)_{\VVcdo} 
  =  \frac{1}{2}
\int_{\omg}\int_{\omg}  (v(\yb)-v(\xb))  (u(\yb)-u(\xb))   \gamd\xyp
 \dydx
 \end{align}
for all $u, v \in \VVcdo$. 
}
  \end{remark}

\begin{remark}
\label{rmk:nlrobin}
Recall the Robin type constraint specified by \eqref{nonlocalrobin} with
\eqref{str-gvc}, 
we can get similar well-posedness in such a case by modifying the bilinear form to
\begin{align}\label{weak-bform-robin}
  {\mcAd}(u,v) =  (u,v)_{\VVcd} + (\mcDd u,v)_{\omgid} ,\quad
\forall\: u, v\in\VVdel,
\end{align}
For given $\del>0$, if $(\mcDd u,u)_{\omgid} =0$ for $u\in \VVdel$ holds only if $u\mid_\omgid=0$, then the coercivity of the modified bilinear form can also be derived (e.g., by following similar discussions on nonlocal Poincar\'e inequality given in \cite{md15non}).
\end{remark}

To conclude this section, we note that 
one can establish the maximum principle for linear nonlocal diffusion problems as for the local case; see discussions in  \cite{dy19jsc,djlq20sirev}.

\vspace{5pt}
\section{Special examples and limiting cases}\label{fract}

In this section, we consider some special cases, including
nonlocal models with an integrable kernel that are of particular interest in applications like peridynamics.
Then,
 we consider the local ($\delta\to 0$) and fractional ($\delta\to \infty$) limits of nonlocal diffusion models parameterized by $\delta>0$.
 This again is to demonstrate the bridging role of the nonlocal models with a finite $\delta>0$, \ie, as illustrated in Figure.~\ref{fig:diagram}, which generalizes some similar pictures in \cite{tdg16acm} for homogeneous Dirichlet problems.
 
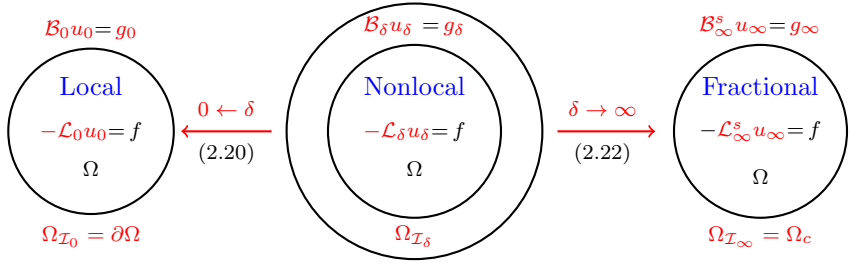
\begin{figure}
\centering
\begin{tikzpicture}[thick,scale=1, every node/.style={scale=1.}]
       \node[draw, circle, minimum width=22mm] at ({0}:0cm) {};
         \node(v3) at (0,-0.5)  {\fns{$\omg$}}; 
       \node(v1) at (0.,-1.4) {\textcolor{red}{\fns{$\omgio=\omgp$}}};
    \node(v2) at (0, 0) {\textcolor{red}{\fns{$-\mathcal{L}_{0} u_{0}$}}\fns{=$\,{f}$}};
             \node(v4) at (0,0.6)  {\textcolor{blue}{{Local}}}; 
      \node(v5) at (0,1.35) {\textcolor{red}{\fns{$ \mcBo{u}_{0}$}}\fns{=$\,\textcolor{red}{g_0}$}};
       \node(v4) at (1.8,0.3) {\textcolor{red}{\fns{$0 \leftarrow \delta$}}}; 
         \node(v16) at (1.8,-0.3) {\fns{\eqref{eq:scal1}}}; 
  \draw[red, thick, <-]  (1.2,0.) -- (2.4,0); 
      \node[draw, circle, minimum width=34mm] at ({0}:43mm) {};
 \node[draw, circle, minimum width=23mm] at ({0}:43mm) {};
    \node(v6) at (4.3,-0.5) {\fns{$\omg$}}; 
  \node(v7) at (4.3,-1.4) {\textcolor{red}{\fns{${\omg_{\mathcal{I}_\delta}}$}}};
    \node(v8) at (4.3,0.) {\textcolor{red}{\fns{$- \mcLd {u}_{\delta}$}}\fns{=$\,{f}$}};
                 \node(v4) at (4.3,0.6)  {\textcolor{blue}{Nonlocal}}; 
        \node(v9) at (4.3,1.4) {\textcolor{red}{\fns{$\mcBd{u}_{\delta}$}} \fns{=$\,\textcolor{red}{g_\delta}$}};
 \node(v10) at (6.8,0.3) {\textcolor{red}{\fns{$\delta \rightarrow \infty$}}}; 
  \node(v20) at (6.8,-0.3) {\fns{\eqref{eq:scal2}}}; 
  \draw[red, thick, ->]  (6.2,0.0) -- (7.5,0); 
    \node[draw, circle, minimum width=23mm] at ({0}:89mm) {};
         \node(v13) at (8.9,-0.6)  {\fns{$\omg$}};
                \node(v11) at (8.9,-1.4) {\textcolor{red}{\fns{$\omgii=\omgc$}}};
         \node(v9) at (8.9,1.4) {\textcolor{red}{\fns{$\mcBis {u}_{\infty}$}}\fns{=$\,\textcolor{red}{g_\infty}$}};
        \node(v12) at (8.9,0) {\fns{$-$}\textcolor{red}{\fns{$\mcLis {u}_{\infty}$}}\fns{={$\,{f}$}
        }};
                \node(v14) at (8.9,0.6)  {\textcolor{blue}{{Fractional}}}; 
\end{tikzpicture}\\
\caption{
Limits of nonlocal models with a finite interaction radius $\delta$ as $\delta\to 0$ (local limit) with a kernel given by \eqref{eq:scal1} and as $\delta\to\infty$ (fractional limit) with a kernel given by \eqref{eq:scal2}.}
\label{fig:diagram}
\end{figure}

\subsection{Nonlocal Laplacian with integrable kernels}\label{integrable}

We consider the case with a kernel satisfying
\begin{equation}\label{eq:kernel-integrable}
 \int_{\Rd} 
 \gamd\xyp d\yb = 
{\del^d}C_d(\del)  \int_{{B_1({\bf 0})} } 
 \gamma(|\zb|)\, d\zb < \infty.
\end{equation}

In this case, for any given $\delta>0$,
$ \mcLd $ and $\mcNd$ become bounded operators on the $L^2$ spaces for functions defined on respective domains.
One can see that the functions spaces are essentially the  $L^2$ spaces on the corresponding domains \cite{md14prsa}.
Then for the problem, 
\begin{equation}\label{str-vcp-more}
\begin{cases}
  - \mcLd u =  f(\xb), & \quad\forall\,  \xb\in\omg,
\\
  u = \gdd(\xb), &\quad\forall\,  \xb\in\omgidd,\\
 \mcNd u = \gdn(\xb), &\quad\forall\,  \xb\in\omgidn,
\end{cases}
\end{equation}
the basic existence of solution in $L^2(\omghd)$ can be obtained, similar to earlier discussion, for  $f\in L^2(\omg)$, $\gdd\in L^2(\omgidd)$ and
$\gdn\in L^2(\omgidn)$.

We present this special case to highlight the fact that the solutions to the nonlocal diffusion problem may not share the smoothing properties enjoyed by the solutions to typical elliptic PDEs; that is, the solutions may not have a higher order of differentiability than the data.

\subsection{Local limit}\label{local-limit}

We present some results concerning the local limits. We begin with a formal discussion.
For example, equation \eqref{weak-weakf}  is a nonlocal analog of the weak formulation corresponding to \eqref{str-pde}
with ${\mcB_0} u=u$
on $\omgpd$  and ${\mcB_0}  u=\frac{\partial u}{\partial \mathbf{n}}$
on $\omgpn$ with the weak form
given by
\begin{equation} \label{local-weakform}
    \int_\omg \nabla u \cdotsp \nabla v \dxb  { + \mcQ(u,v) }= \int_\omg fv\dxb + \int_{\omgpn}g^0_N v\dxb.
\end{equation}
For a pure {linear Neumann problem with $\omgpn=\omgp$, \ie, $\alpha=0$ in \eqref{local-weakform}}
the functions $f$ and $g_N$ satisfy
 the conventional compatibility condition 
\[
{\int_\omg f\dxb + \int_{\partial\omg}g^0_N \dxb=0.}
\] 

Next, we present some rigorous results concerning the limit of $\delta\to 0$ of \eqref{str-vcp} for kernels
satisfying \eqref{eq:scal0} and \eqref{eq:scal1}.
To properly discuss the local limit, we need to make several assumptions on the boundary domains and boundary data.  

\begin{assu}
\label{assumption-on-data}
We first make the following assumptions on the domains.
\begin{itemize}
\item[$\bullet$] The domains $\omgidd$ and $\omgidn$ are monotone decreasing sets as $\del\to0$ and $\overline{\omgidd}\supset \omgpd$ and $\overline{\omgidn}\supset \omgpn $. 
\item[$\bullet$] Moreover, for any $\del_1>\del_2>0$, we assume that $(\omg^D_{\mathcal{I}_{\del_1}}\backslash \omg^D_{\mathcal{I}_{\del_2}}) \cap \omg_{\mathcal{I}_{\del_2}} =\emptyset$. 
\end{itemize}
Next, we present the assumptions on the data.
\begin{itemize}
\item[$\bullet$]
For a constant $\del_0>0$, we assume that for all $\del<\del_0$, the function $w_D^\del$ in \eqref{eqn:extd}
{and\eqref{eqn:extdp}} is extended to $\omgh_{\del_0}$ with the extension bounded in $V_{c,\del_0}(\omgh_{\del_0})\cap \alpha L^m(\omg)$ and $w_D^\del$ converges in $L^2(\omgh_{\del_0})$ to $w_D^0 \in H^1(\omgh_{\del_0})$ as $\del\to0$, where $w_D^0|_{\omgpd}= g_D^0$. 
\item[$\bullet$]
We also assume that the Neumann data $\gdn$, as a function defined on $\omghd$, is uniformly bounded in the dual space of the energy space. Moreover,
\begin{align}
\label{eqn:neuconv}
\lim_{\del\to 0}
\int_{\omgidn} \gdn(\xb) v(\xb) \dxb = \int_{\omgpn} g_N^0(\xb) v(\xb) \dxb \,,\quad   \forall v\in C^\infty(\Rd).
\end{align}
\end{itemize}
\end{assu}

\begin{remark}
{Let us note that under the assumption on $m$ and the Sobolev imbedding theorem, we have $\alpha w_D^0\in \alpha L^m(\omg)$. }
\end{remark}

\begin{remark}
\label{rem:domains}
Notice the special geometric assumptions on the domains $\omgidd$ and $\omgidn$ in the above discussions. In particular, we require $\{ \overline{\omgidd} \}_{\del>0}$ and $\{ \overline{\omgidn} \}_{\del>0}$ to be sequences of nested sets that decrease monotonically to $\omgpd$ and $\omgpn $, respectively, as $\delta\to0$. These assumptions rule out the more complicated situations where $\omgidd$ or $\omgidn$ may not be attached to the boundary $\partial\omg$. 
{Some numerical experiments are provided in cases these geometric assumptions are violated.}
\end{remark}

With these assumptions, we present the convergence theorem for $\delta\to0$. 

\begin{thm}\label{thm:convtolocal}
Let $u_\del$ and $u_0$ be the weak solutions to the nonlocal and local boundary value problems, respectively. Assume that $\|\fh\|_\VVsd$ and $\| \gdd\|_\VVt$ are bounded uniformly in $\del$, then we have 
\[
\| u_\del - u_0\|_{L^2(\omg)} \to 0\,,
\]
as $\del\to0$. 
\end{thm}
\begin{proof}
From the well-posedness theorem, it is easy to see that $\| u_\del \|_{\VVdel}$,
$\| u_\del - w_D^\del \|_{\VVdel}$ and  $\mcQ(u_\del,u_\del)$
are uniformly bounded and $(u_\del - w_D^\del)|_{\omgidd} =0 $. Take $\del_0>0$ and let $\omg_1 := \overline{\omg\cup\omg^D_{\mathcal{I}_{\del_0}}}$. By the assumptions on the boundary sets, for any $\del<\del_0$, we can always do a zero extension of the function $u_\del - w_D^\del$
onto $\omg_1$, and 
\[
\begin{split}
&\int_{\omg_1}\int_{\omg_1} \gamd(\xb, \yb) \left( (u_\del - w_D^\del)(\yb) - (u_\del - w_D^\del)(\xb) \right)^2 \dyb\dxb  \\
\leq & \int_{(\omg\cup\omgidd)^2 \backslash (\omgidd)^2}\gamd(\xb, \yb) \left( (u_\del - w_D^\del)(\yb) - (u_\del - w_D^\del)(\xb) \right)^2 \dyb\dxb \\
\end{split}
\]
We can then use Lemma \ref{lem:compact1} to conclude that any $L^2(\omg_1)$-limit (up to a subsequence) of $u_\del - w_D^\del$ as $\del\to 0$ belongs to $H^1(\omg_1)$. Let $v^\ast$  denote the limit function and $u^\ast:= v^\ast+w_D^0$. By the assumption on the Dirichlet data, $u^\ast$ is the $L^2(\omg_1)$-limit of $u_\del$. We want to show $u^\ast =u_0$.  First, since $(u_\del -w_D^\del)|_{\omg^D_{\mathcal{I}_{\del_0}}} =0 $ for all $\del<\del_0$, we then have $v^\ast|_{\omg^D_{\mathcal{I}_{\del_0}}} =0 $. This implies that $u^\ast\in H^1(\omg)$ and $u^\ast|_{\omgpd} = w_D^0|_{\omgpd} = g_D^0$. 
Now we only need to show that $u^\ast$ satisfies the weak form \eqref{local-weakform} for any $v\in C^\infty(\overline{\omg})$ with  $v|_{\omgpd} =0$. Assume that the extension of $v$ is made such that $v\in C^\infty(\overline{\omg\cup \omg_{\mathcal{I}_{\del_0}}})$ and $v|_{\omg^D_{\mathcal{I}_{\del_0}}}=0$. Since $u_\del$ is the weak solution to the nonlocal problem, we have
\begin{equation} \label{eq:convthm_eqn1}
   {\mcAd}(u_\del,v) {+\mcQ(u_\del,v)}= (f, v)_{\omg} + (g^\del_N , v)_{\omgidn}\,. 
\end{equation}
The right-hand side of \eqref{eq:convthm_eqn1} converges to the right-hand of \eqref{local-weakform} from the assumption on the Neumann boundary data. Using the same argument in \cite[Theorem 3.5]{ttd17amc}, one can show that as $\del\to0$,
\[
\int_{\omg}\int_{\omg} \gamd(\xb, \yb) (u_\del(\yb)-u_\del(\xb)) (v(\yb)-v(\xb)) \dyb \dxb \to \int_{\omg} \nabla u^\ast \cdot\nabla v \dxb \,. 
\]
Note also that 
\[
\begin{split}
&\left|\iint_{(\omgid\times \omg)\cup(\omg\times \omgid) }\gamd(\xb, \yb) (u_\del(\yb)-u_\del(\xb)) (v(\yb)-v(\xb)) \dyb \dxb \right|\\
\leq & \| u_\del\|^2_{\VVdel} \cdot \iint_{(\omgid\times \omg)\cup(\omg\times \omgid)} \gamd(\xb, \yb)(v(\yb)-v(\xb))^2 \dyb \dxb \\
\leq & C \int_{\omgid} |\nabla v|^2 \dxb \to 0 \,,
\end{split}
\]
as $\del\to 0$. We then have the convergence of ${\mcAd}(u_\del,v)$ to the integral of $\nabla u^\ast \cdot\nabla v$ on $\omg$. {Meanwhile, by the assumption on $m$, we have $u_\ast\in L^{m}(\omg)$. Moreover, we observe that there is a constant $C>0$ only depends on $\omg$ and $m$ such that
\begin{align*}
\|u_\del- u^\ast\|_{L^{m-1}(\omg)}^{m-1}
& \leq \|u_\del- u^\ast\|_{L^2(\omg)}
\|u_\del- u^\ast\|_{L^{2(m-2)}(\omg)}^{m-2}\\
&\leq C \|u_\del- u^\ast\|_{L^2(\omg)}
\|u_\del- u^\ast\|_{L^{m}(\omg)}^{m-2}\\
&\leq C \|u_\del- u^\ast\|_{L^2(\omg)}
(\|u_\del\|_{L^{m}(\omg)}+\| u^\ast\|_{L^{m}(\omg)})^{m-2},
\end{align*}
where we have used $2(m-2)\leq m$ by the choice of $m$.
Thus, we get the convergence of $u_\del$ to $ u^\ast$ in $L^{m-1}(\omg)$ (and consequently the convergence of $q(u_\del)$ to $ q(u^\ast)$ in $L^{1}(\omg)$) from the 
convergence of $u_\del$ to $ u^\ast$ in $L^{2}(\omg)$.
We thus have the convergence of  $\mcQ(u_\del, v)$ to $\mcQ(u_\ast, v)$.} 
Therefore $u^\ast$ is the weak solution of the local boundary value problem, and this implies $u^\ast=u_0$. 
\end{proof}

\begin{remark}
An example for data satisfying \eqref{eqn:neuconv} for $\omg$ being the unit ball and
$\omgidn=\omgid$ is {the case of $\gdn(\xb)$ being given by 
$\delta^{-1} g_N^0(\xb/|\xb|)$} for any $\xb\in\omgid$.
Similar constant extensions along normal directions of the boundary $\partial\omg$ can be constructed for smooth $\omg$.
Other extensions can also be considered, 
see, for example, \cite{d2020physically,you2020asymptotically}. 
One may also examine how the convergence order depends on the data and the solution; see, for example, \cite{dy19jsc,du2019uniform,d2020physically}. It is no surprise that the constant extension may not be a desirable choice in order to achieve higher convergence order. Moreover, such a constant extension only makes sense if the resulting nonlocal data satisfies conditions specified in the assumption \ref{assumption-on-data}.
\end{remark}

\begin{remark}
A direct connection between the fractional Laplacian and
the local Laplacian is the case of $s\rightarrow 1^-$.
It is known that, see \eg\ \cite{Biccari:2018},  the solution of integral fractional diffusion model \eqref{eq:fracPoisson} strongly converges to the solution of the local diffusion problem in $H^{1-\epsilon}(\Omega)$. Related discussions on Neumann boundary problems can be found in \cite{Gounoue20,foghem2022general}.
\end{remark}

\begin{remark}
We may also consider the limit of Robin type problems considered in 
the remark \ref{rmk:nlrobindef} and \ref{rmk:nlrobin} 
assuming that \eqref{str-robin-vc} holds.
Then one may also consider the 
local $\delta \to 0$ limit to the corresponding boundary value problem of the local Laplacian with the local Robin boundary condition.
\end{remark}

\begin{remark}
The special example considered here corresponds to variational problems associated with convex energy so that the solutions are unique. One can consider extensions to problems associated with non-convex energy, see related discussions in \cite{md15non} using $\Gamma$-convergence.
\end{remark}

\subsection{Fractional limit}
Here, we make more studies on treating fractional PDEs as either specialized nonlocal models or limiting cases.

We can let $\delta=\infty$ for kernels of the types in \eqref{eq:scal2}. We identify some of the spaces in this case: 
\[V_\infty(\omgh_\infty)=H^s(\Rd),\quad
V_{c,\infty}(\omgh_\infty)=  {H}^{s}_c(\Rd):=\{u\in H^{s}(\Rd)  \colon  u=0 \ \; \forall\,  \xb\in {\omgiid}\}
\]
\begin{remark}
We consider the case that $\omgiid=\omgii$. For $s>1/2$, then ${H}^{s}_c(\omg)$ coincides with the space $H_{0}^{s}(\omg)$ that is the closure of $C_{0}^{\infty}(\omg)$ with respect to the $H^{s}(\omg)$-norm, whereas for $s<1/2$, ${H}^{s}_c(\omg)$ is isomorphic to $H^{s}(\omg)$ (\ie, functions in
 $H^{s}(\omg)$ with zero extension to $\omg_c$).
In the critical case $s=1/2$, ${H}^{s}_c(\omg)\subsetneq H^{s}_{0}(\omg)$. See \eg\ \cite[Chapter~3]{Mclean:2000} for a detailed discussion.
\end{remark}

Let us first consider an example of a pure Dirichlet type problem.
For  $\omgiid=\omgii=\omgc$,  the following weak formulation has been given in \cite{AinsworthGlusa2018_TowardsEfficientFiniteElement}: 
\begin{equation}\label{eq:fracPoissonVariational}
  \begin{minipage}[b]{25pc}
  {Find} $ u(\xb)\in {H}^{s}(\Rd)$  such that $ u(\xb)=g_D^\infty(\xb) $ for all $ \xb\in\omgc $  and
  \[
  \mcA_s(u,v)=\fvo {+} (g_D^\infty,v)_{{\omg,\omgc}}\qfa v\in {H}^{s}_c(\Rd),
  \]
  \end{minipage}
\end{equation}
where
\begin{align}
\label{aaaaaa}
 \mcA_s(u,v)
  &=
  \frac{1}{2} \int_{(\Rd\times\Rd)\setminus (\omgc\times\omgc)} \;   (u({\yb})-u({\xb}))  (v({\yb}) -v({\xb}) )  \gams\xyp \dydx \notag\\
  & = 
  \underbrace{ {\frac{1}{2}} \int_{\omg}\int_{\omg} \;   (u({\yb})-u({\xb}))  (v({\yb}) -v({\xb}) )  \gams\xyp \dydx}_{\mcA_{\omg,\omg}(u,v)} \notag \\
    &\quad\, + \underbrace{
    \int_{\omg}u({\xb})v({\xb})\int_{{\omgc}} \gams\xyp \dydx
    }_{\mcA_{\omg,\omgc}(u,v)}
\end{align}
and
\begin{equation}\label{rhsfl}
(g_D^\infty,v)_{\omg,\omgc} = \int_\omg v(\xb) \int_\omgc g_D^\infty(\yb)\gams\xyp \dydx.
\end{equation}

It is easy to see that due to the Dirichlet constraint on $\omgiid=\omgii=\omgc$, $u(\xb)=g_D^\infty(\xb)$ and $v(\xb)=0$ for $\xb\in\omgc$. Thus,
\begin{align*}
(g_D^\infty,v)_{\omg,\omgc} & = \int_\omg v(\xb) \int_\omgc g_D^\infty(\yb)\gams\xyp \dydx\\
&= \int_\omg v(\xb) \int_\omgc u(\yb)\gams\xyp \dydx
\end{align*}
Meanwhile, assuming that Fubini's theorem holds, then
\begin{align*}
 \mcA_{\omg,\omgc}(u,v) - (g_D^\infty,v)_{\omg,\omgc} & = \int_\omg \int_\omgc v(\xb) (u(\xb)- u(\yb))\gams\xyp \dydx\\
& = \int_\omg \int_\omgc (v(\xb) - v(\yb)) (u(\xb)- u(\yb))\gams\xyp \dydx\\
& = \int_\omgc \int_\omg (v(\xb) - v(\yb)) (u(\xb)- u(\yb))\gams\xyp \dydx\\
& = \frac{1}{2}\int_\omg \int_\omgc (v(\xb) - v(\yb)) (u(\xb)- u(\yb))\gams\xyp \dydx\\
&\quad + \frac{1}{2}\int_\omgc \int_\omg (v(\xb) - v(\yb)) (u(\xb)- u(\yb))\gams\xyp \dydx
\end{align*}

Thus, \eqref{eq:fracPoissonVariational} is equivalent to that defined \eqref{weak-weakf} for $\delta=\infty$ with $\mcQ =0$.

\begin{remark}
One may also enforce the volume constraints via Lagrange multipliers, see e.g.  \cite{Acosta:2019}.
\end{remark}

One can derive some rigorous results on the limit of \eqref{str-vcp} as $\delta\to \infty$ for kernels satisfying \eqref{eq:scal2}, particularly for the case of pure Dirichlet condition. For example, one can extend the result given \cite{tdg16acm} for the case of
$\gdd\equiv 0$ and $g_D^\infty \equiv 0$ (homogeneous Dirichlet problem) to inhomogeneous data.

\begin{thm}
For kernels satisfying \eqref{eq:scal2},
let $u_\del$ and $u_\infty$ be the weak solutions to the nonlocal and fractional Dirichlet boundary value problems associated with
an element $w\in H^{s}(\Rd)$ and
$\gdd=w\mid_{\omgid}$ and
$g_D^\infty=w\mid_{\omgc}$ respectively. Assume that $\|\fh\|_\VVsd$ is bounded uniformly in $\del$, then we have 
\[
\| u_\del - u_\infty\|_{H^{s}(\omg)} \to 0\,,
\]
as $\del\to \infty$. 
\end{thm}

Indeed, one can get the above conclusion by considering the corresponding problems for
$u_\del -w$ and $u_\infty-w$
with homogeneous Dirichlet data and noticing that 
$\mcAd(w,v) \to \mcA_s(w,v)$
as $\del\to \infty$ for any 
$v\in H^s_c(\Rd)$. Then the convergence of $u_\del \to u_\infty$ can be derived from a similar convergence result for the homogeneous Dirichlet problem, see \cite{tdg16acm} for more discussions on the latter.

\begin{remark}
In the more general case, further studies are needed to characterize the solution spaces and the weak form as $\del\to \infty$  for kernels satisfying \eqref{eq:scal2}. For example, in \cite{Dipierro:2017},  some weighted norm based on the Neumann data $g$ is introduced, which is involved in the definition of weak solution.
\end{remark}

\section{Numerical experiments and discussion}

In this section, we shall present some numerical examples to show the behaviours of the solution to
the nonlocal model \eqref{str-vcp} with inhomogeneous boundary conditions. 
Since the Dirichlet boundary conditions have been intensively studied in existing works,
we focus on the cases for
the pure Neumann or mixed type boundary conditions here.

We consider a one-dimensional nonlocal problem in the unit interval $\Omega = (0, 1)$
with boundary region $\omgid =(-\delta,0)\cup(1,1+\delta)$. Throughout the section,
we choose the kernel functions in the fractional type:
\begin{equation*} 
\gamma_\delta(x,y) = C_{s,\delta} |x-y|^{-1-2s} \chi_{B_\delta(0)}(x-y)
\end{equation*} 
Here the normalization constant $C_{s,\delta}(\delta)$ is defined by
\begin{equation*} 
C_{s,\delta} =\begin{cases}
  (2-2s)\delta^{-2+2s},\quad& \text{for}~~\delta\le \delta_s,\\
   \ffrac{2^{2s}s\Gamma (s+\gfrac{1}{2})}{\pi^{\frac12}\Gamma (1-s)},\quad&\text{for}~~ \delta\ge \delta_s,
\end{cases} \quad \text{with}\quad 
\delta_s = \Big( \ffrac{(2-2s)\pi^{\frac12}\Gamma(1-s)}{2^{2s}s\Gamma(s+1/2)} \Big)^{\frac1{2-2s}}.
\end{equation*} 
In our computation, we use the standard Galerkin finite element method with uniform mesh size $h=1/M$. Without loss of generality, we assume that $h\ll\delta$ and $r=\delta/h$ is an integer. Define the grid points $\{x_r\}_{i=-r}^{M+r}$ such that $x_r = rh$, we then define $X_h$ to be the set of functions in $\VVcd$ (see the definition in \eqref{weak-espacec}) which are 
linear when restricted to the
subintervals $[x_i, x_{i+1}]$, $i=-r,\ldots,M+r$. 
We fix a sufficiently small mesh size $h=10^{-4}$.

\vskip5pt

\textbf{Example 1. Local limit: pure Neumann boundary conditions.}
To begin with, we examine the nonlocal model  \eqref{str-vcp} with the inhomogeneous Neumann boundary conditions. We define the source term 
$ f(x) = x(1-x)$ for all $x\in \Omega$
and the Neumann type volumetric constraint
$$g_N^\delta(x) =\begin{cases}
  -\ffrac1{6\delta},\quad &x\in(-\delta,0),\\
  0,\quad &x\in(1,1+\delta).
\end{cases} $$
Note that the source term $f$ and boundary data $g_N^\delta$ {satisfy the compatibility condition \eqref{weak-compat}}, which implies the existence of the weak solution.
We shall test two types of Neumann Boundary conditions that correspond to the bilinear forms \eqref{weak-bform}
and \eqref{weak-ip1-regional} respectively, and thus implying implicitly the use of the operators $\mcNd$ and $\mcN_{\delta,\omghd}$ respectively, where
 \begin{equation}\label{eqn:n1}
 \mcNd u(x)  := - \int_\omg (u(y)-u(x)) \gamma_\delta(|x-y|)\,\mbox{d} y,
  \quad\forall\,  x\in\omgid. 
   \end{equation}
and
\begin{equation}\label{eqn:n2}  
\mcN_{\delta,\omghd} u(x)  := - \int_\omghd (u(y)-u(x)) \gamma_\delta(|x-y|)\,\mbox{d} y, \quad\forall\,  x\in\omgid.  
\end{equation}
For both cases, the local limit
is the steady diffusion problem
\begin{equation}\label{eqn:pde-local}
    -u''(x) = f(x)~~ \text{for} ~~x\in \Omega,
    \qquad \text{with}~~u'(0)=1/6~~\text{and}~~u'(1)=0.
\end{equation}

With the nonlocal Neumann boundary condition $\mcNd u = g_N^\delta$,
the finite element scheme for the nonlocal problem reads:
find $u_h \in X_h$ such that for all $v_h \in X_h$,
\begin{equation}\label{fem1}
\begin{split}
{\mcAd(u_h, v_h)=}
  \frac{1}{2}
\int_{\omgtd}  (v_h(y)-v_h(x)) & (u_h(y)-u_h(x)) \gamma_\delta(x-y)
 \,dy\,dx \\
 &= \int_\Omega f(x) v_h (x) \, d x + \int_{\Omega_{\mathcal{I}_\delta}} g_N^\delta v_h(x)\,dx.
 \end{split}
\end{equation}
Meanwhile, with the second kind of 
Neumann boundary condition $
\mcN_{\delta,\omghd} u = g_N^\delta$,
the corresponding finite element approximation reads: 
find $u_h \in X_h$ such that for all $v_h \in X_h$,
\begin{equation}\label{fem2}
\begin{split}
{\mcA_{\delta,\omghd} (u_h,v_h)=}
  \frac{1}{2}
\int_{\omghd\times\omghd} & (v_h(y)-v_h(x))  (u_h(y)-u_h(x)) \gamma_\delta(x-y)
 \,dy\, dx \\
 &= \int_\Omega f(x) v_h (x) \, d x + \int_{\Omega_{\mathcal{I}_\delta}} g_N^\delta v_h(x)\,dx.
  \end{split}
\end{equation}
Both finite dimensional problems \eqref{fem1} and \eqref{fem2} are well-posed by the Lax--Milgram theorem and the nonlocal version of Poincar\'e inequality in Lemma \ref{lem:poincare}.

{Figure \ref{fig:Neumann} presents the numerical solutions in $(0,1)$
with various $\delta$ and $s$, while Figure \ref{fig:Neumann-x0} zooms in 
on the solution profile near the boundary $x=0$. In the case that $s=-1$, the solution is nonsmooth 
and we observe discontinuity at $x=0$, 
while the solution is continuous in $\omghd$ in cases that $s=0.25$ and $0.75$.} As $\delta \rightarrow 0$, the numerical results clearly indicate that the solution to the
nonlocal diffusion problem converges to $u_0$, the solution to the local problem. This fully supports our theoretical finding in Theorem \ref{thm:convtolocal}. Meanwhile, with the relatively large nonlocal horizon $\delta$, our empirical experiments also show that the solution of the boundary operator $\mcN_{\delta,\omghd}$ is closer to the local solution $u_0$. Finally, in Table \ref{tab:N-rate}, we present the convergence rate as 
$\delta \rightarrow 0$ for both cases. These interesting phenomena warrant further investigation in our future studies.

\begin{figure}[hbt!]
\setlength{\tabcolsep}{0pt}
   \centering
   \begin{tabular}{ccc}
   \includegraphics[width=0.33\textwidth]{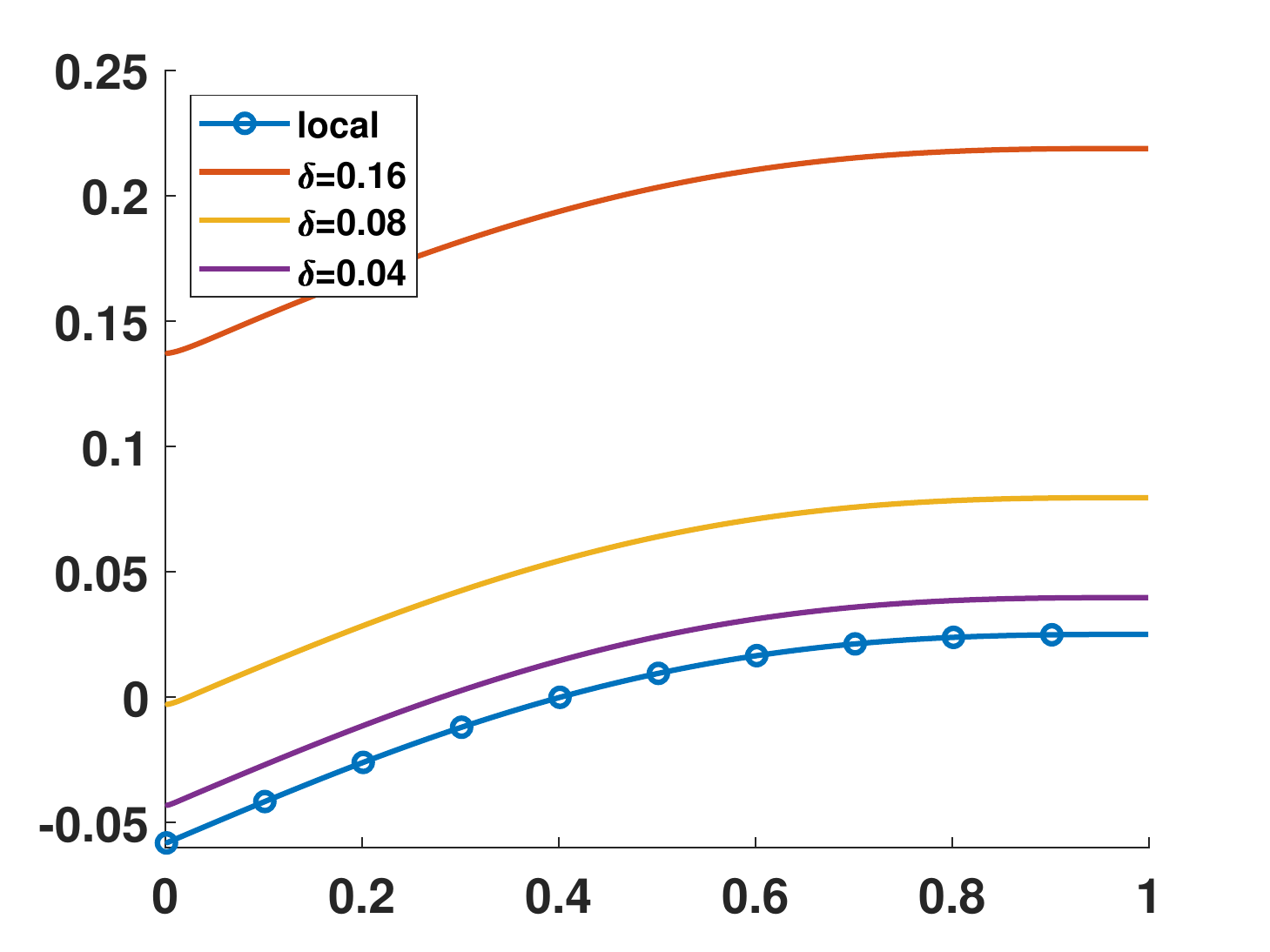}&
   \includegraphics[width=0.33\textwidth]{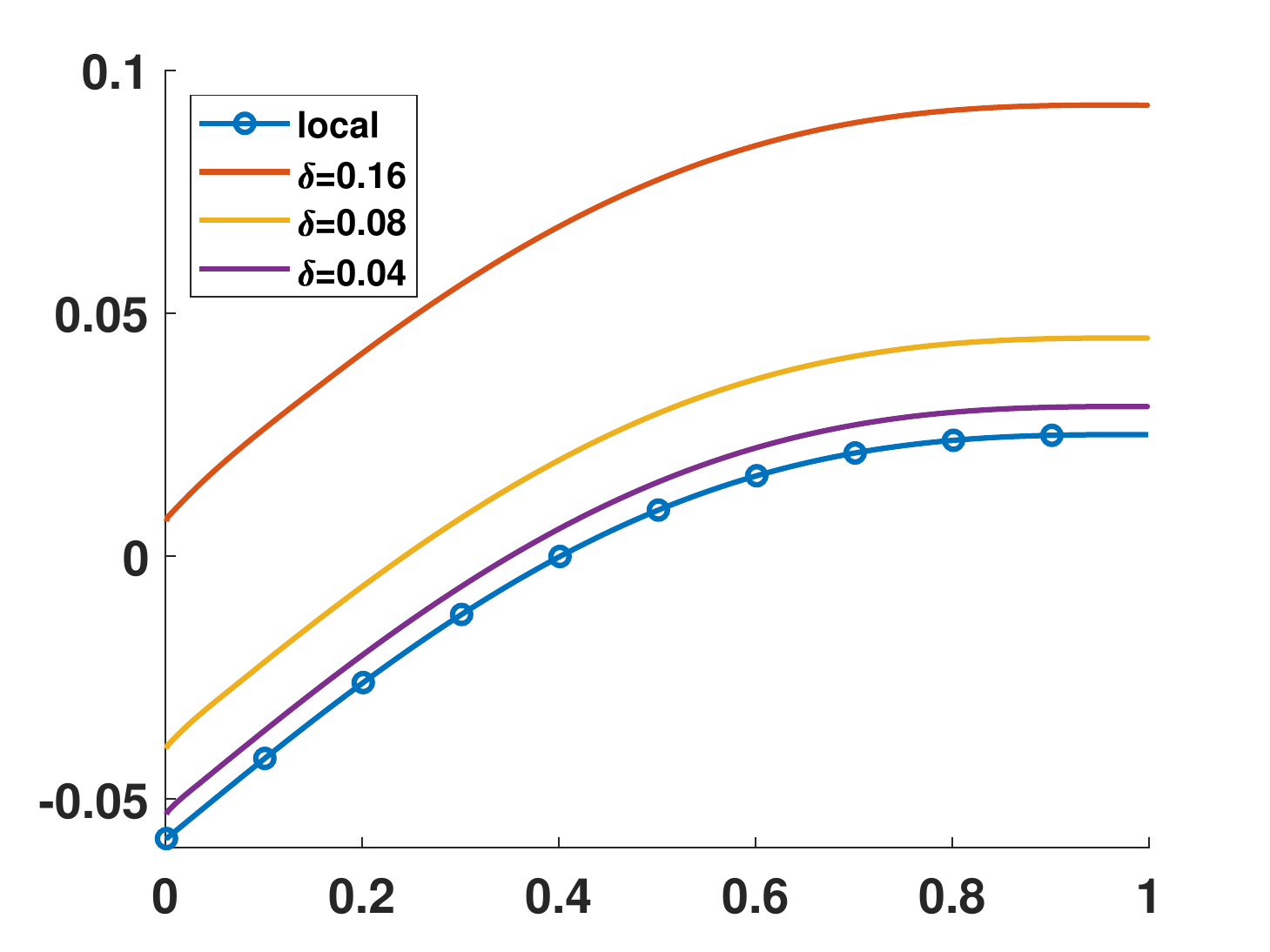} & \includegraphics[width=0.33\textwidth]{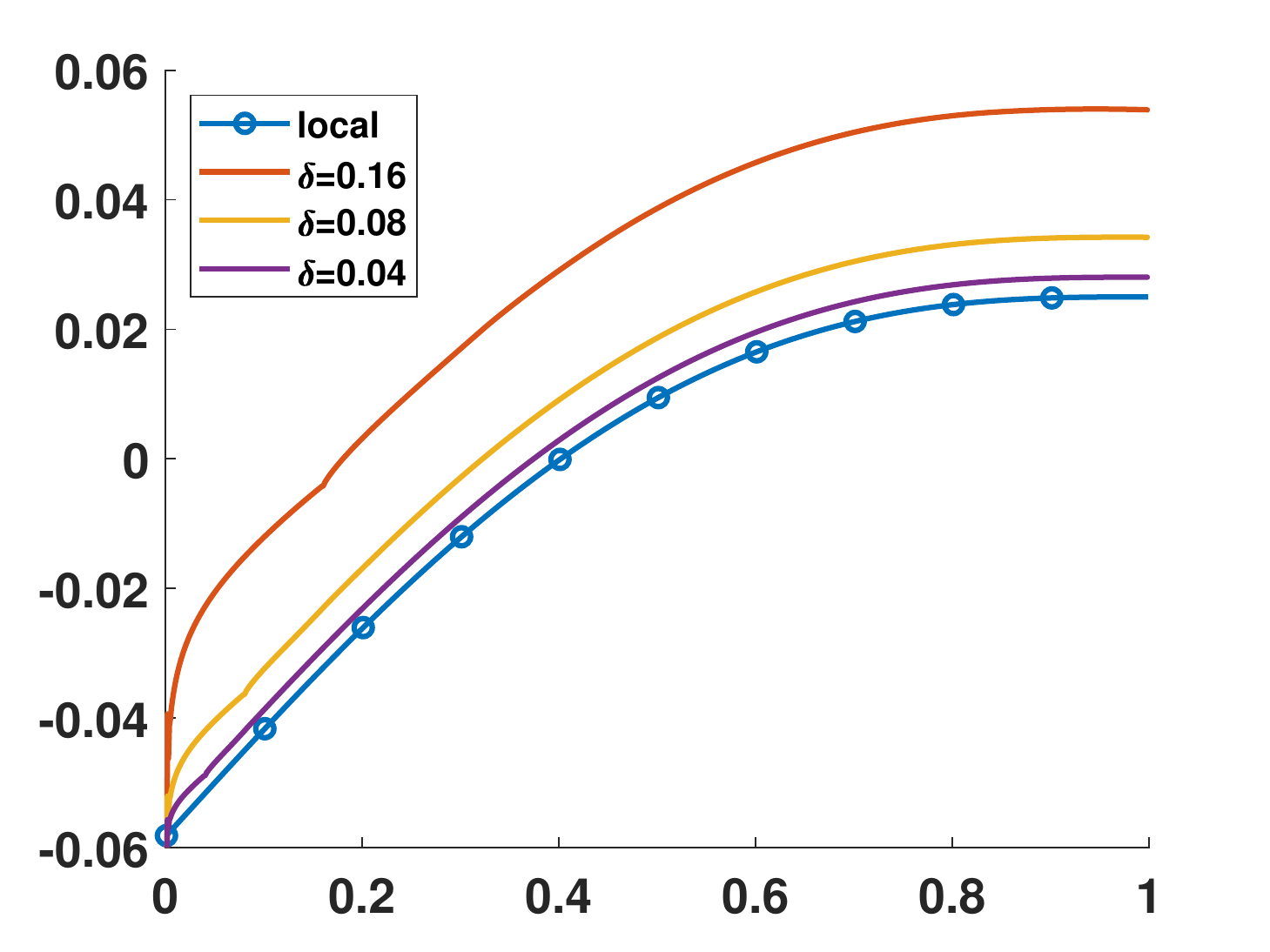}\\
      \includegraphics[width=0.33\textwidth]{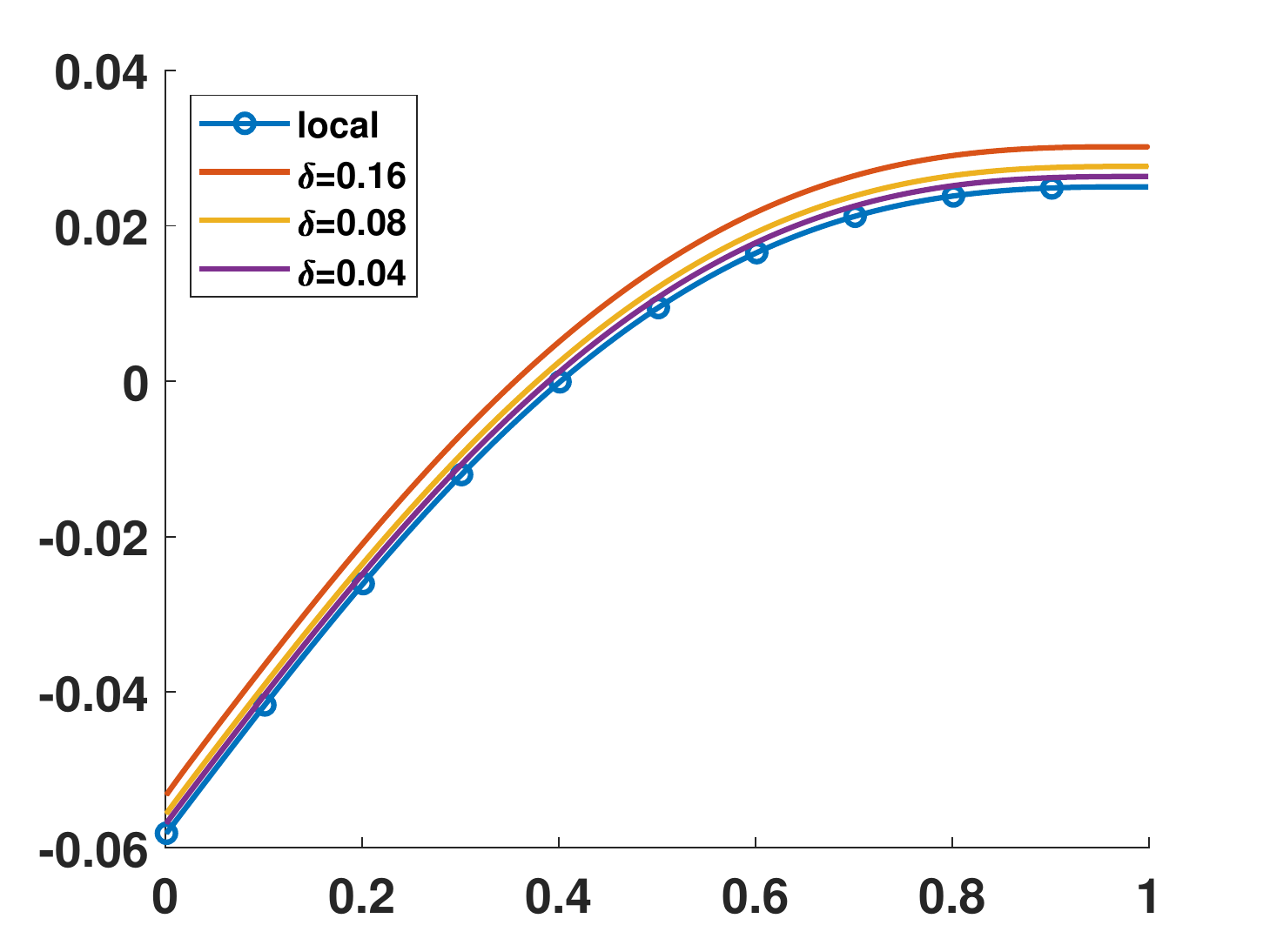}&
   \includegraphics[width=0.33\textwidth]{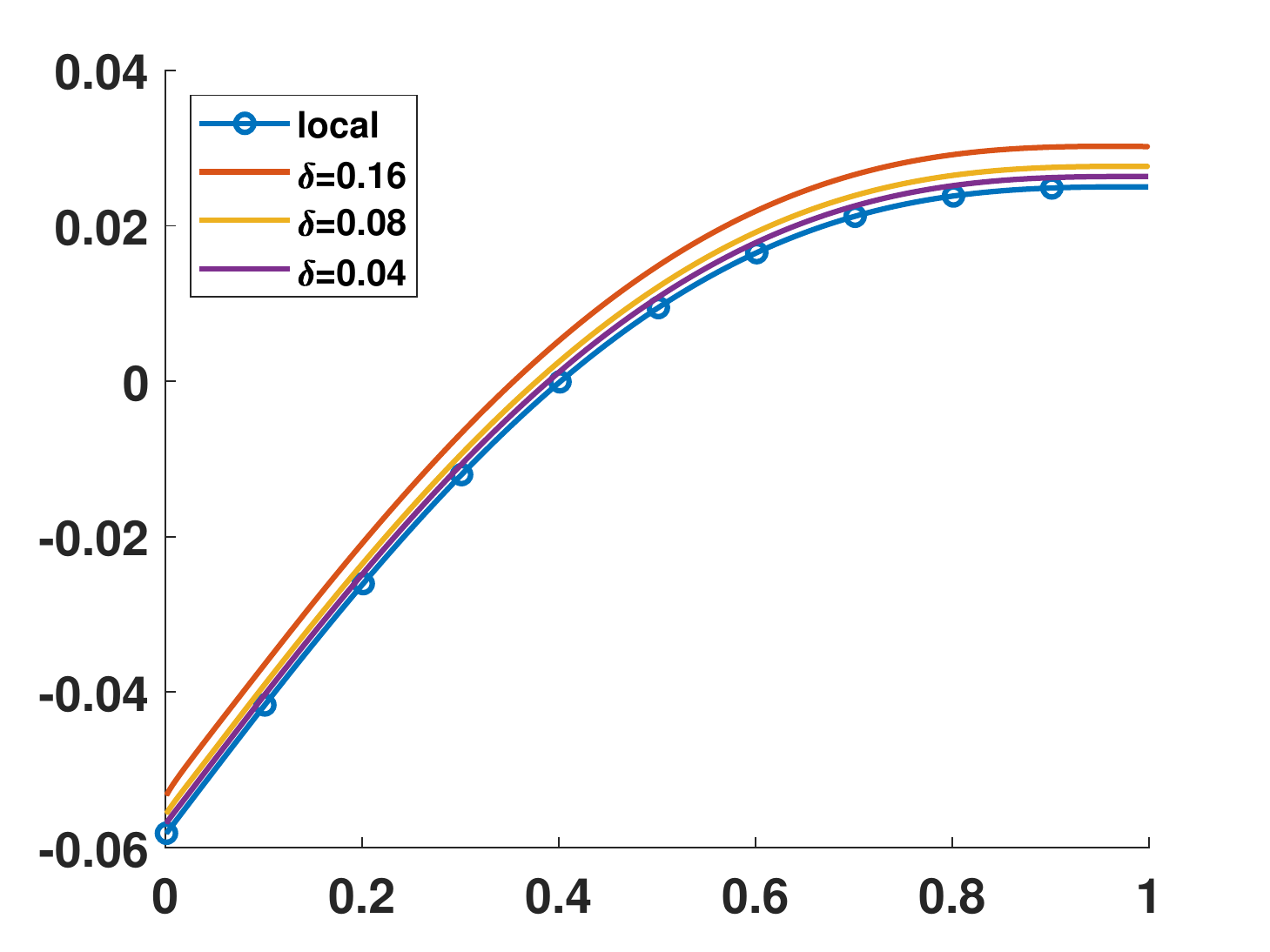} & \includegraphics[width=0.33\textwidth]{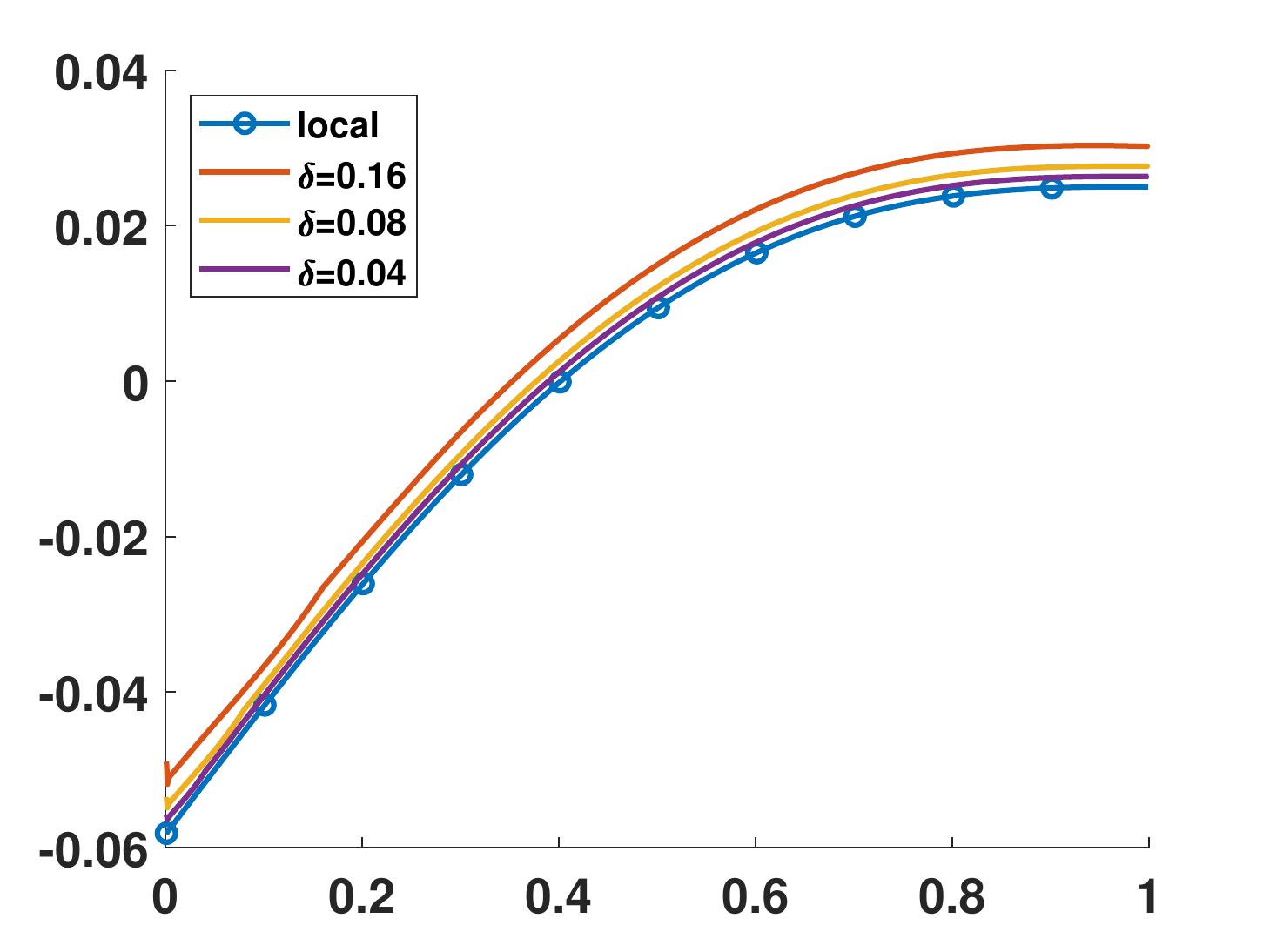}\\
       $s=0.75$  &  $s=0.25$ &   $s=-1$
   \end{tabular}
   \caption{Example 1. Solutions to the nonlocal diffusion problem with inhomogeneous Neumann type boundary conditions, with operators $\mcNd$ (upper) and $\mcN_{\delta,\omghd}$ (lower).}
   \label{fig:Neumann}
\end{figure}

\begin{figure}[hbt!]
\setlength{\tabcolsep}{0pt}
   \centering
   \begin{tabular}{ccc}
   \includegraphics[width=0.33\textwidth]{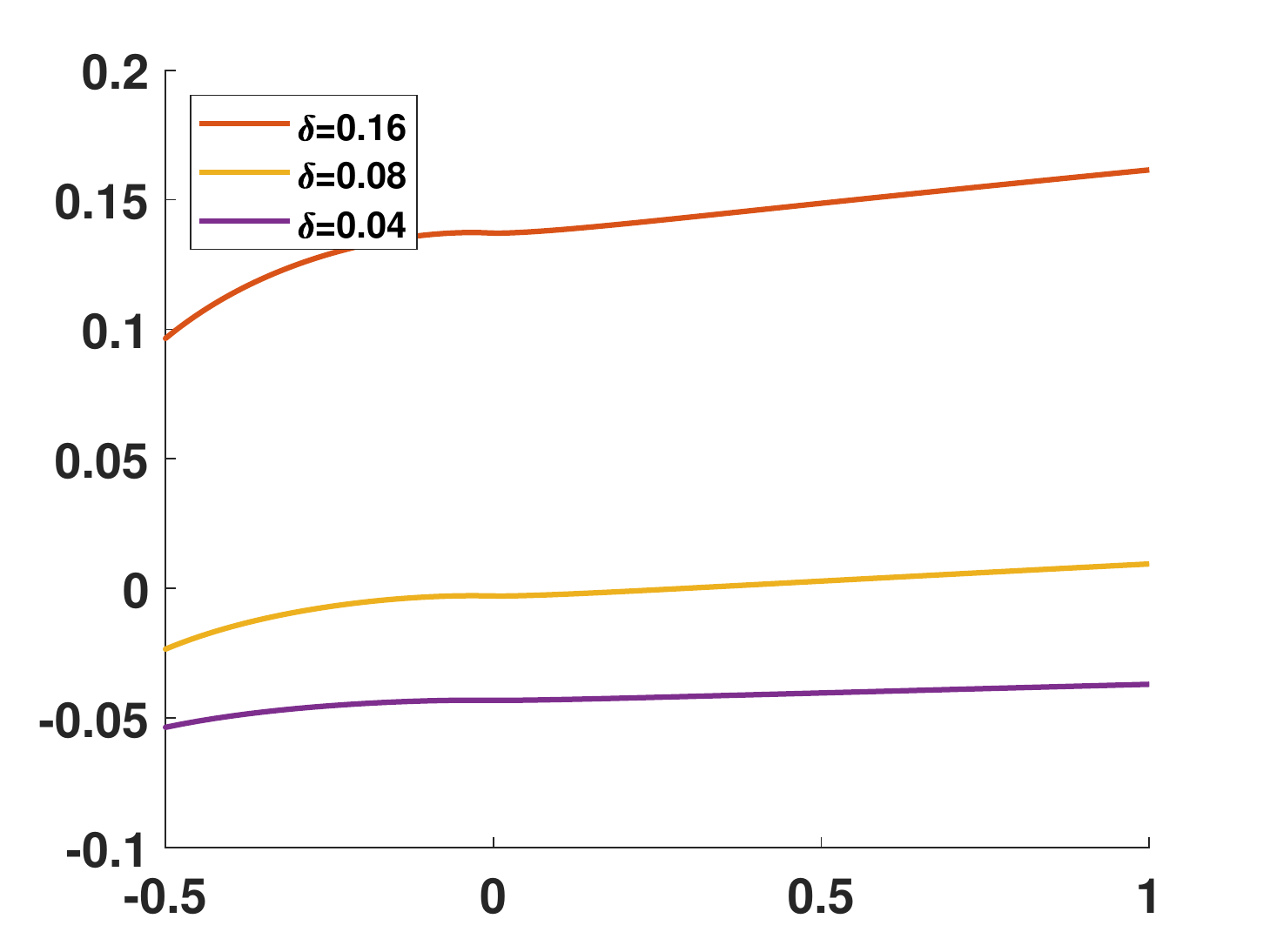}&
   \includegraphics[width=0.33\textwidth]{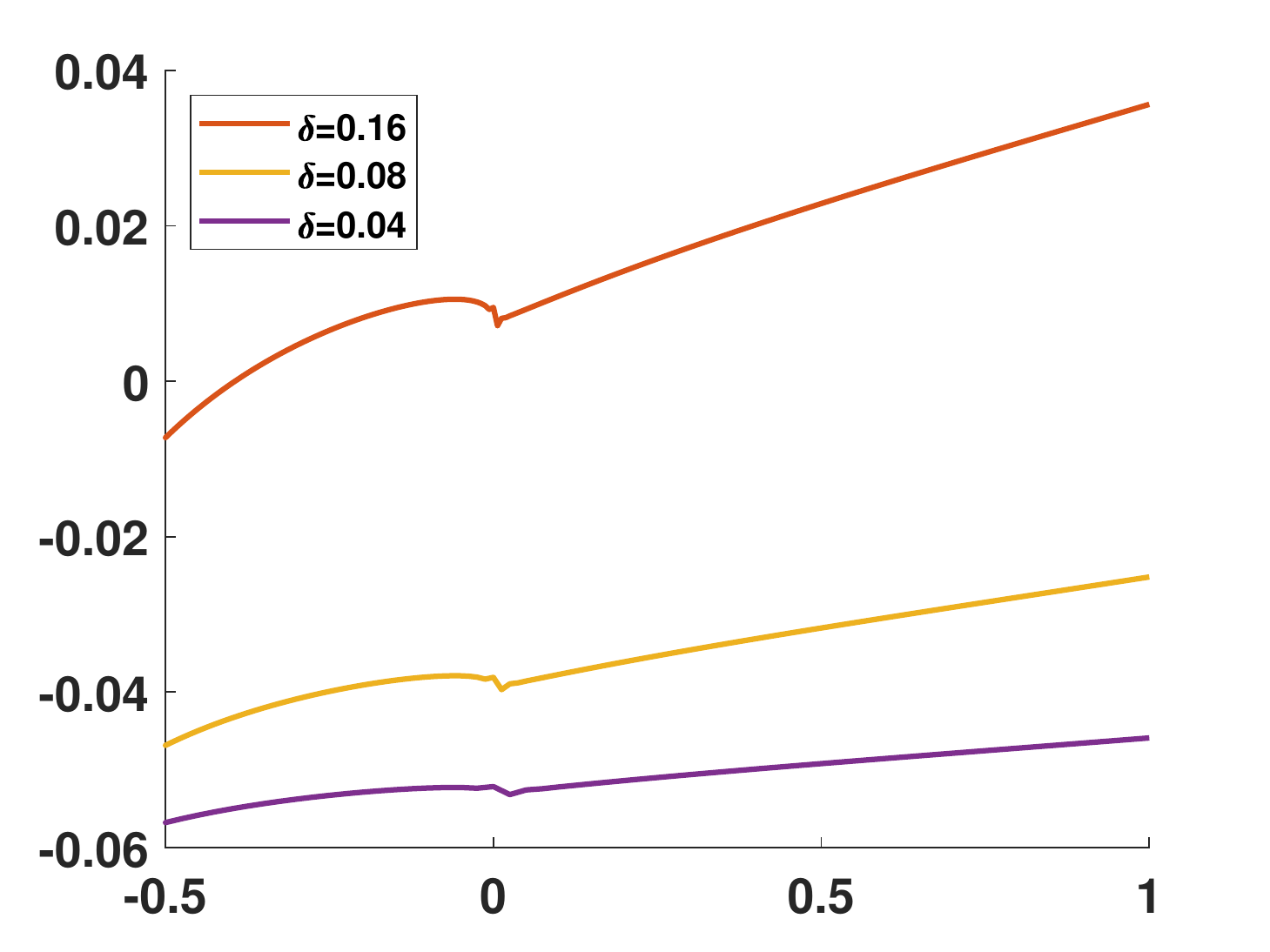} & \includegraphics[width=0.33\textwidth]{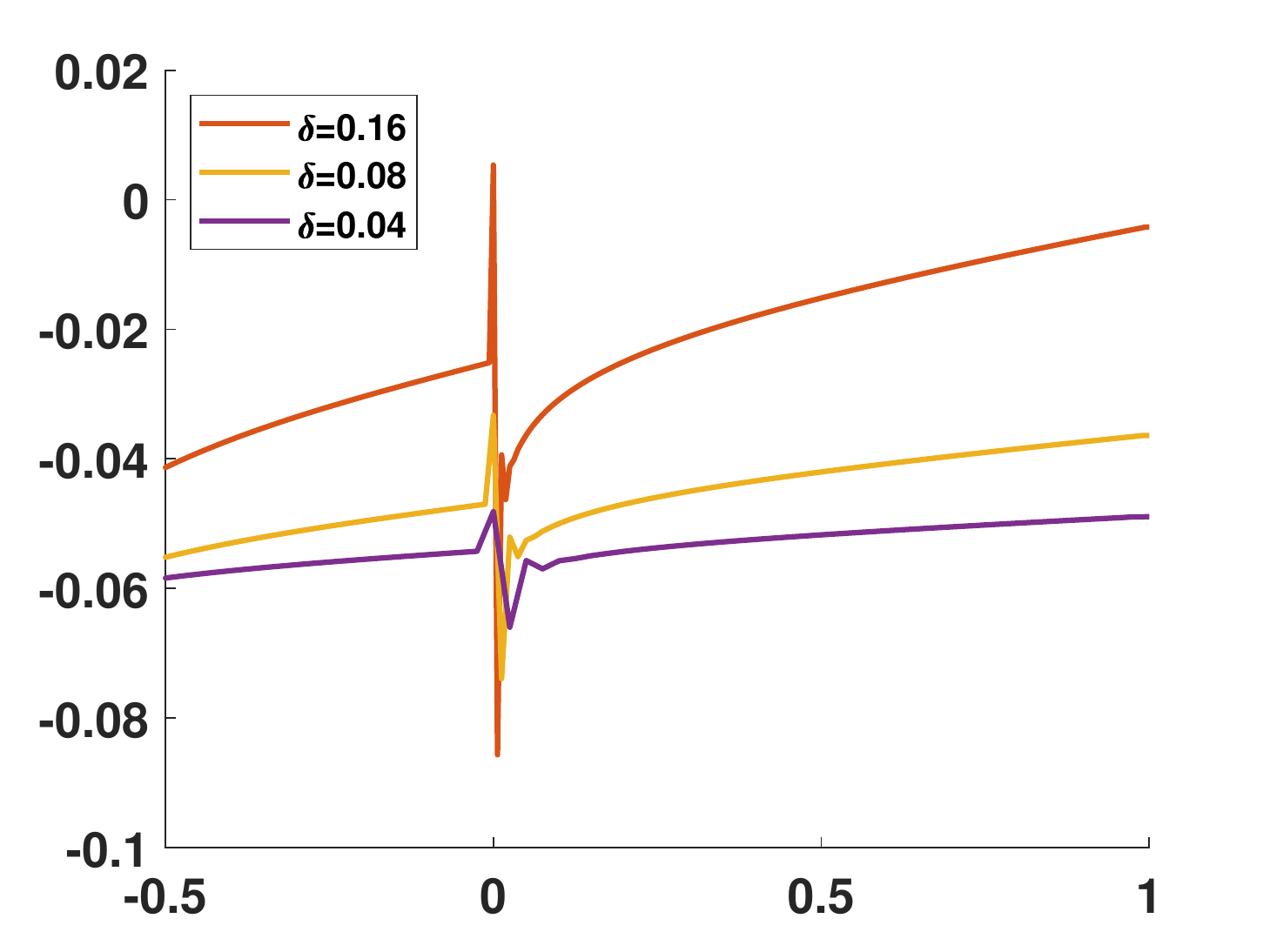}\\
      \includegraphics[width=0.33\textwidth]{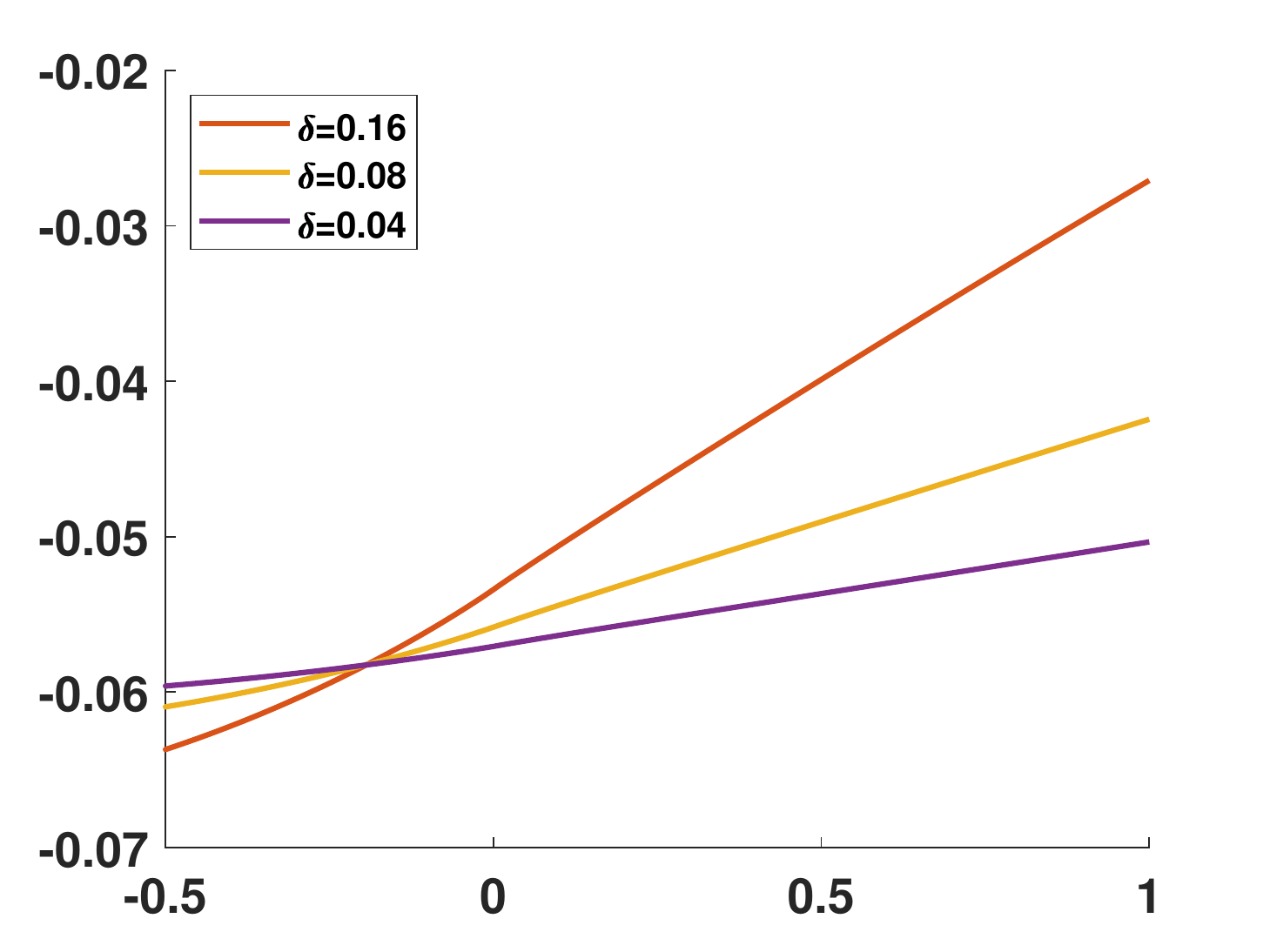}&
   \includegraphics[width=0.33\textwidth]{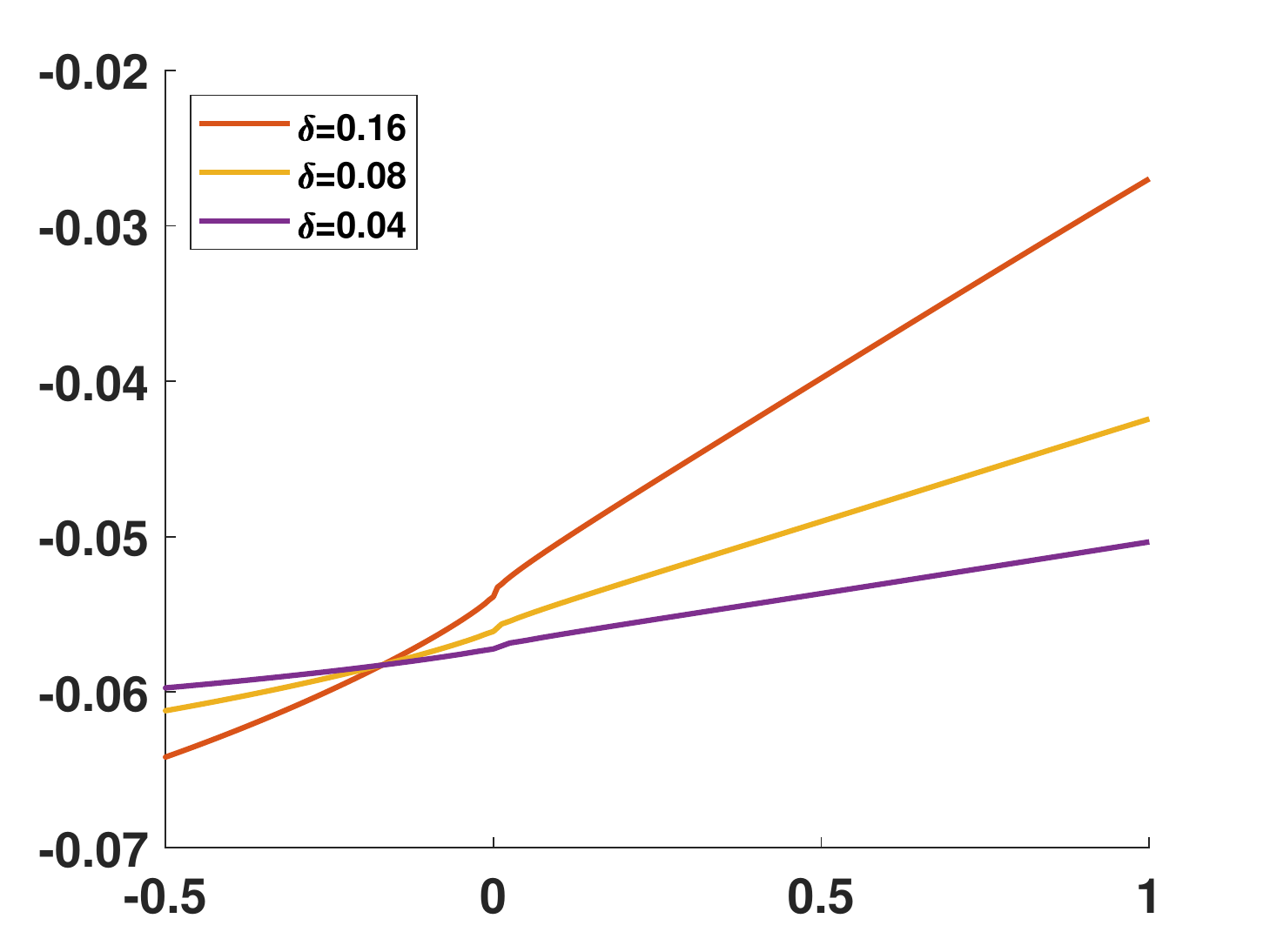} & \includegraphics[width=0.33\textwidth]{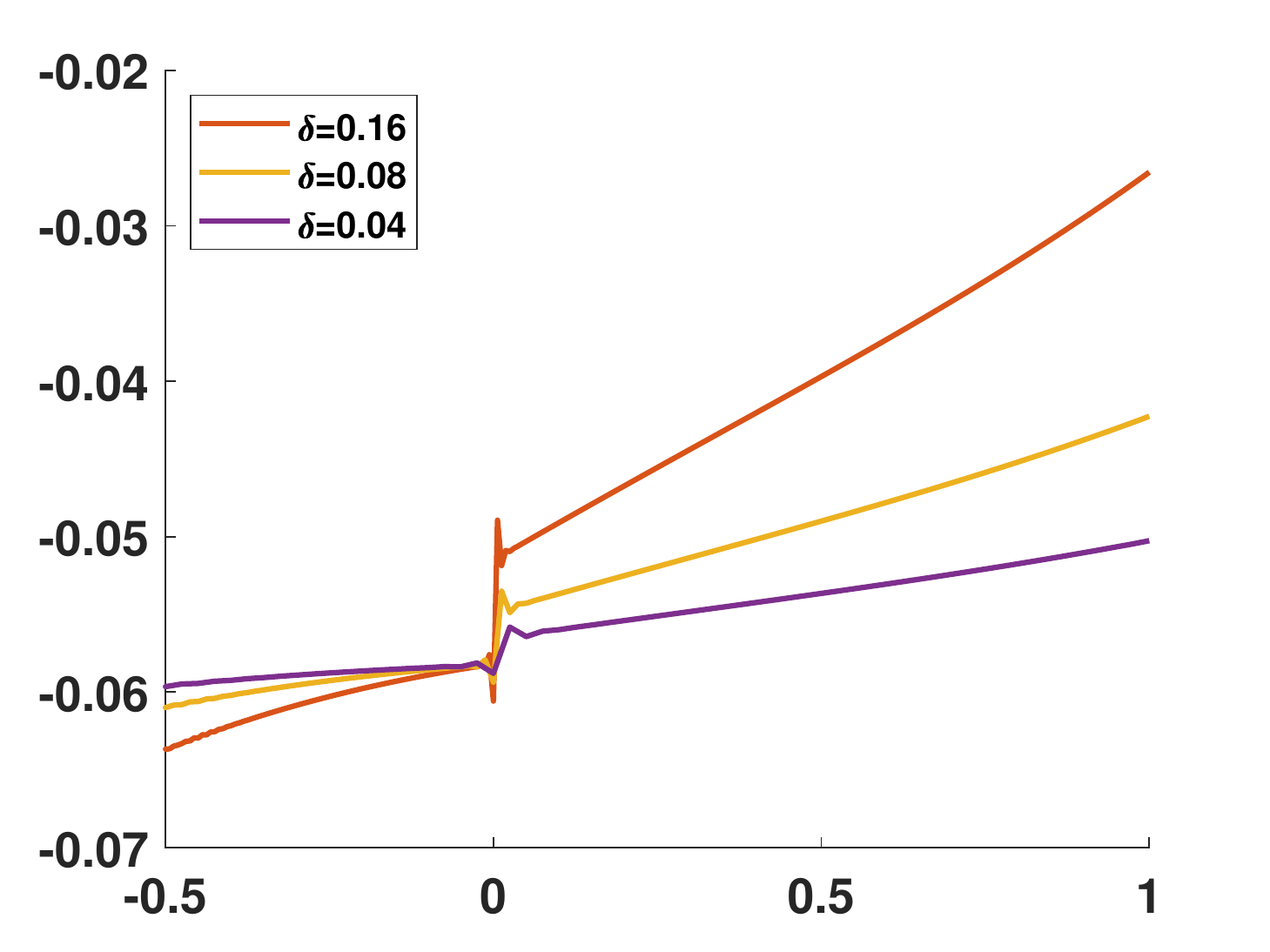}\\
       $s=0.75$  &  $s=0.25$ &   $s=-1$
   \end{tabular}
   \caption{Example 1. Profiles of $u_\delta(\delta y)$, $y\in(-1/2,1)$, with various  horizon parameters $\delta$ and fractional powers $s$, with operators $\mcNd$ (upper) and $\mcN_{\delta,\omghd}$ (lower).}
   \label{fig:Neumann-x0}
\end{figure}

\begin{table}[htbp]
   \caption{Example 1. $\|u_\delta - u_0\|_{L^2(\Omega)}$ with $s=0.75$, $0.25$ and $-1$.}\label{tab:N-rate}
   \vskip-5pt
\centering
\begin{tabular}{|c|c|ccccc|c|}
\hline
&$s\backslash\delta$& 0.08 & 0.04 & 0.02 & 0.01 &0.005  &Rate ($\delta$)\\
\hline
& $-1$ & 9.33E-3 & 3.02E-3  & 1.14E-3 & 4.35E-4  & 1.97E-4& 1.27\\
$\mathcal N_\delta$ 
& $0.25$ &2.15E-2  & 6.29E-3 & 1.92E-3 & 6.54E-4 & 2.48E-4& 1.47 \\
& $0.75$  & 5.97E-2 & 1.63E-2 & 4.44E-3 & 1.27E-3 &  3.95E-4& 1.75 \\
\hline
& $-1$ &  2.71E-3 & 1.33E-3  & 6.63E-4& 3.27E-4 & 1.60E-4  & 1.02\\
$\mcN_{\delta,\omghd}$ 
& $0.25$ &2.65E-3  & 1.32E-3 & 6.60E-4 & 3.26E-4 & 1.60E-4 & 1.01 \\
 &$0.75$  &2.62E-3 & 1.31E-3 & 6.57E-4 & 3.26E-4 & 1.60E-4& 1.01 \\
\hline
\end{tabular}
\end{table}

\vskip5pt
\textbf{Example 2. Local limit: mixed boundary conditions.}
In this example, we consider zero source term, i.e. $f(x)=0$,
and the following two types of boundary conditions:
\begin{equation*}
\begin{split}
      &
      \mathcal{N}_\delta u (x)
      = -\frac{1}{3\delta} ~~ \text{for}~~
     x\in (-\delta,-\delta/2),\quad
u(x) = 1~~ \text{for}~~ x\in[-\delta/2,0),\\
&u(x) = 1~~ \text{for}~~x\in(1,1+\delta), \qquad 
\textbf{Example 2(a)},
\end{split} 
\end{equation*}
and 
\begin{equation*}
\begin{split}
         &u(x) = 1 ~~ \text{for}~~
     x\in (-\delta,-\delta/2),\quad
\mathcal{N}_\delta u(x) = -\frac{1}{3\delta}~~ \text{for}~~ x\in[-\delta/2,0),\\
&u(x) = 1~~ \text{for}~~x\in(1,1+\delta), \qquad 
\textbf{Example 2(b)}.
\end{split} 
\end{equation*}
Noting that the main difference between Examples 
2(a) and 2(b) is the order of Neumann and Dirichlet boundary conditions in the interval $(-\delta,0)$.  This leads to the case that a part of the data (Neumann data for the former, and Dirichlet data for the latter respectively) is defined in the interval disconnected from $(0,1)$.
It is interesting to observe that, in Figure \ref{fig:m1}, both solutions converge to the same local limit, i.e., the solution to the steady diffusion problem: $-u''(x)= 0$ in $\Omega=(0,1)$ with  Dirichlet boundary conditions $u(0) = u(1) = 1$. 
{The solution profiles near $x=0$, i.e. $u_\delta(\delta y)$ with $y\in (-1, 2)$, are plotted in Figure  \ref{fig:m1-near0}. We observe the discontinuity at $x=-\delta/2$, which might be due to the incompatibility between the nonlocal Neumann and Dirichlet boundary conditions.}
In Table \ref{tab:Ex2-rate}, we present the difference between $u_\delta$ and $u_0$ in the $L^2$ sense. Our numerical experiments show that the convergence of Example 2 (b) is slower than that of Example 2 (a).
These observations deserve further theoretical analysis in the future.

\begin{figure}[hbt!]
\setlength{\tabcolsep}{0pt}
   \centering
   \begin{tabular}{ccc}
   \includegraphics[width=0.33\textwidth]{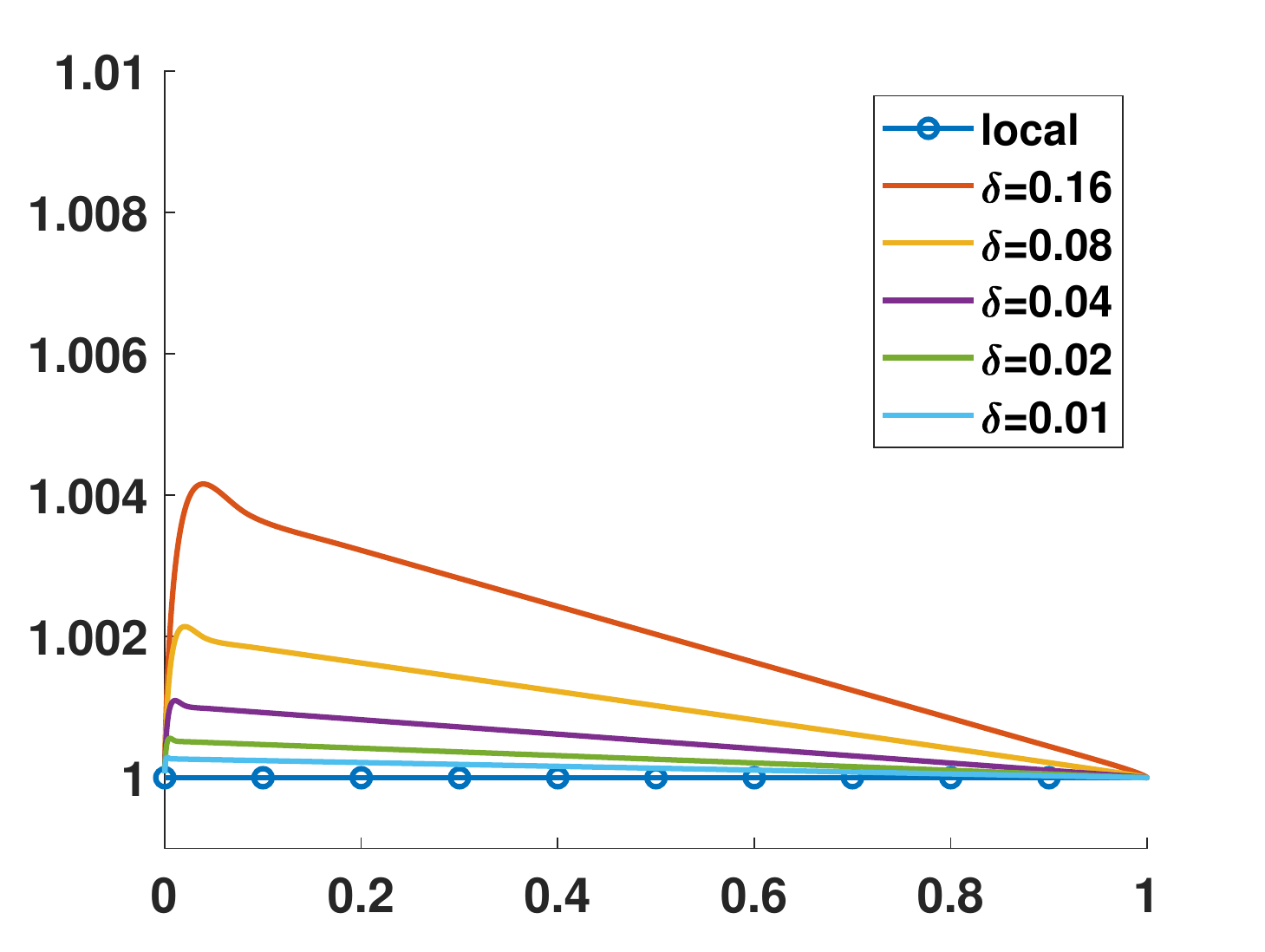}&
   \includegraphics[width=0.33\textwidth]{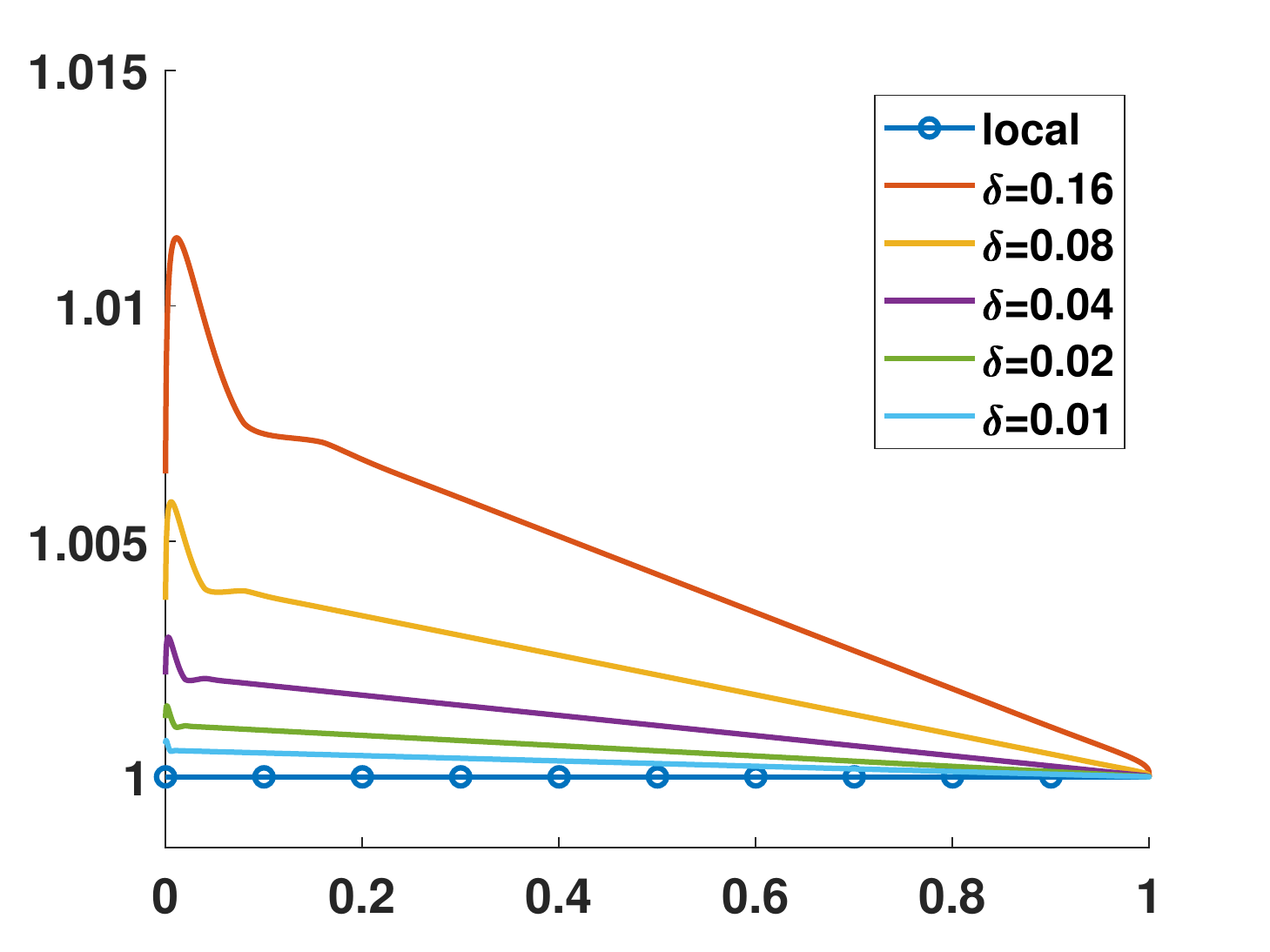} & \includegraphics[width=0.33\textwidth]{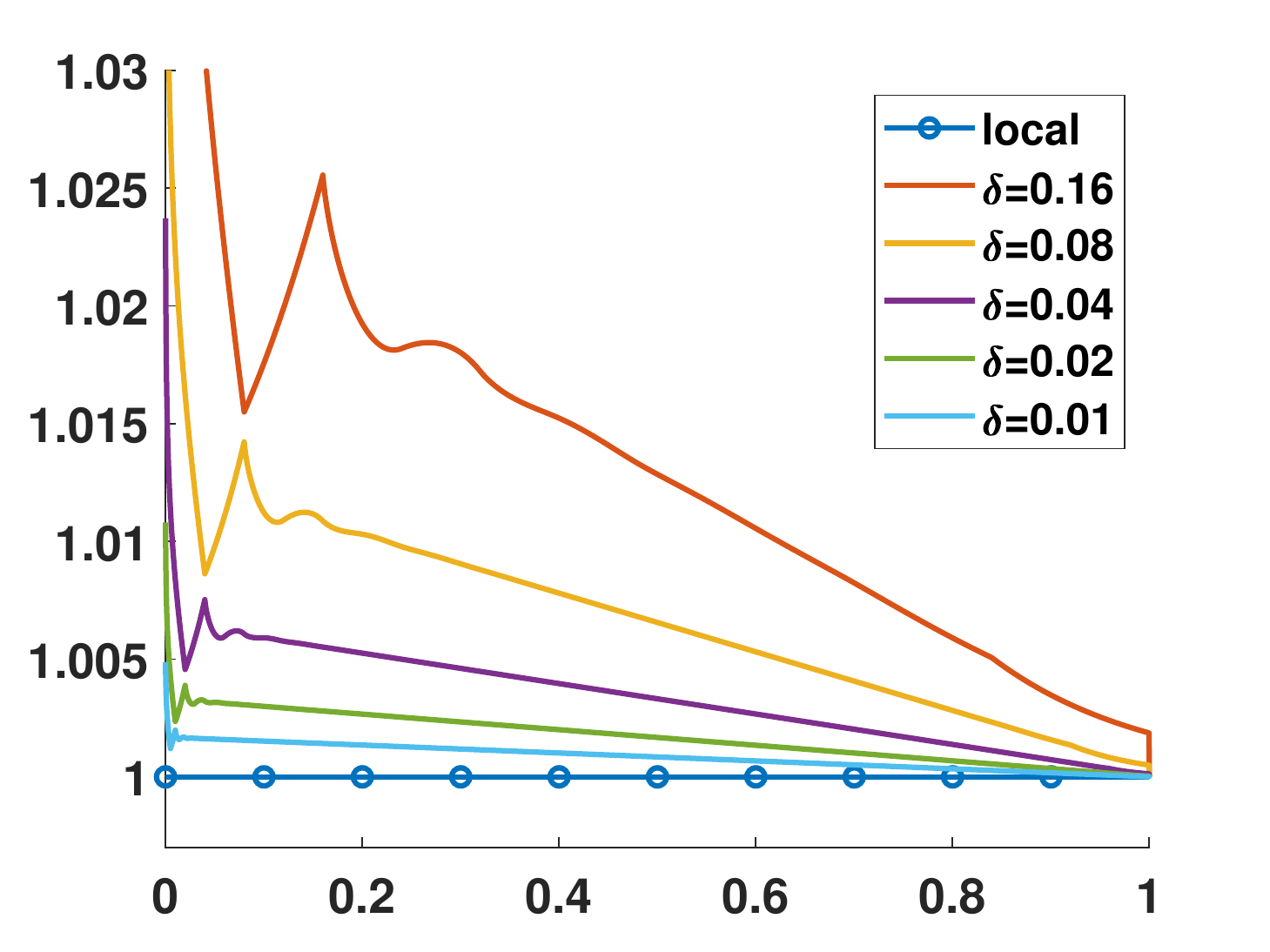}\\ 
   \includegraphics[width=0.33\textwidth]{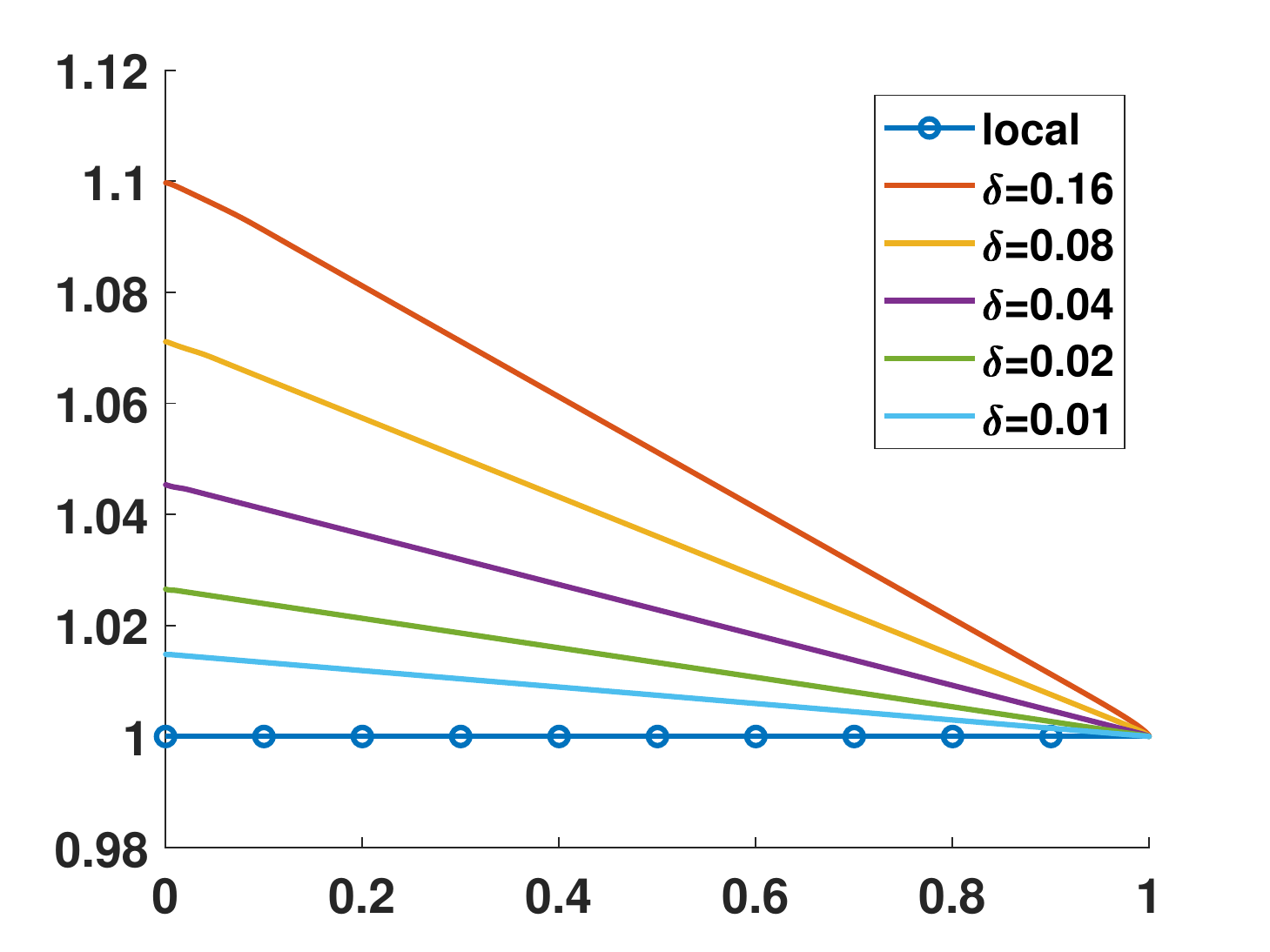}&
   \includegraphics[width=0.33\textwidth]{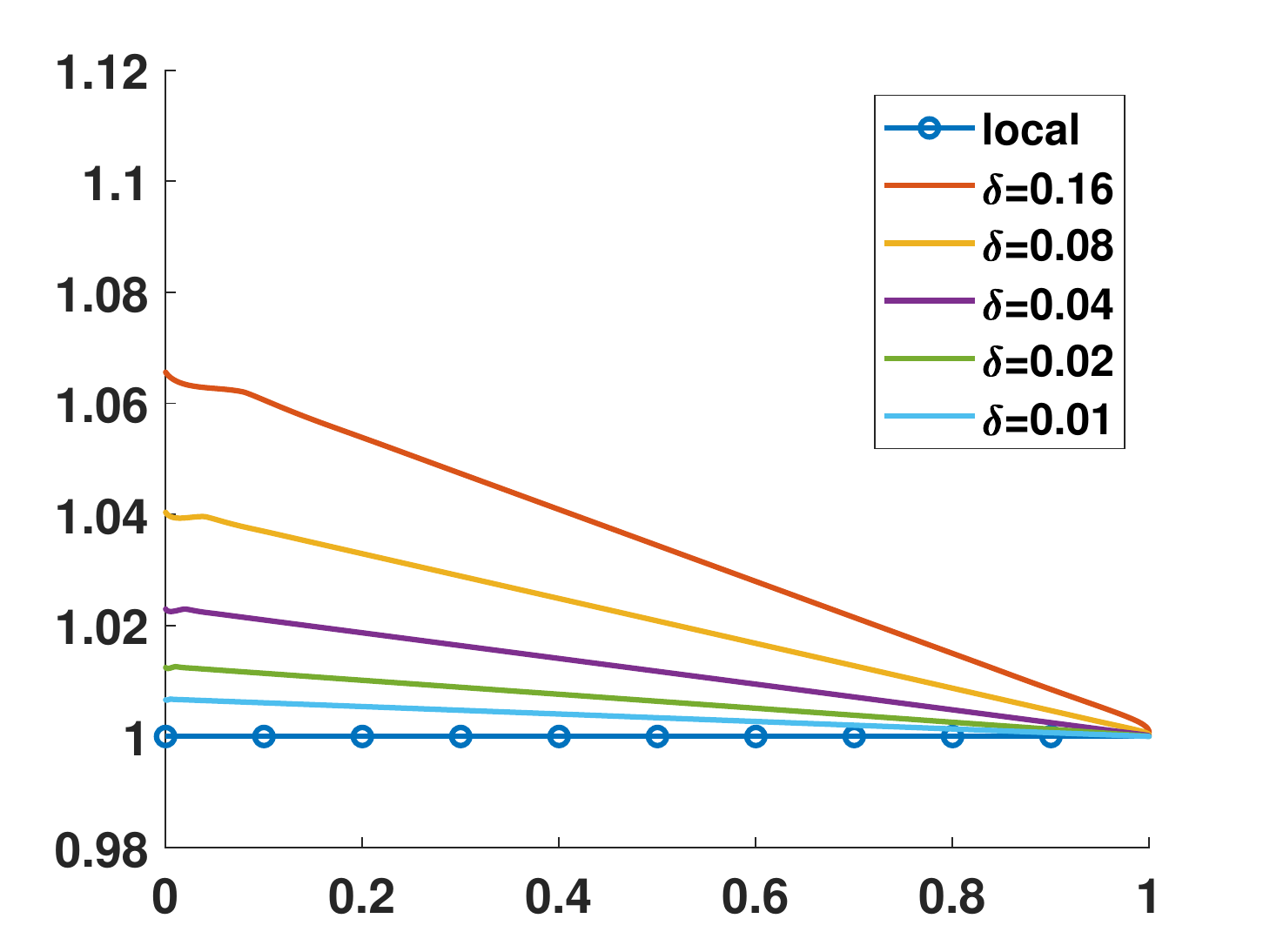} & \includegraphics[width=0.33\textwidth]{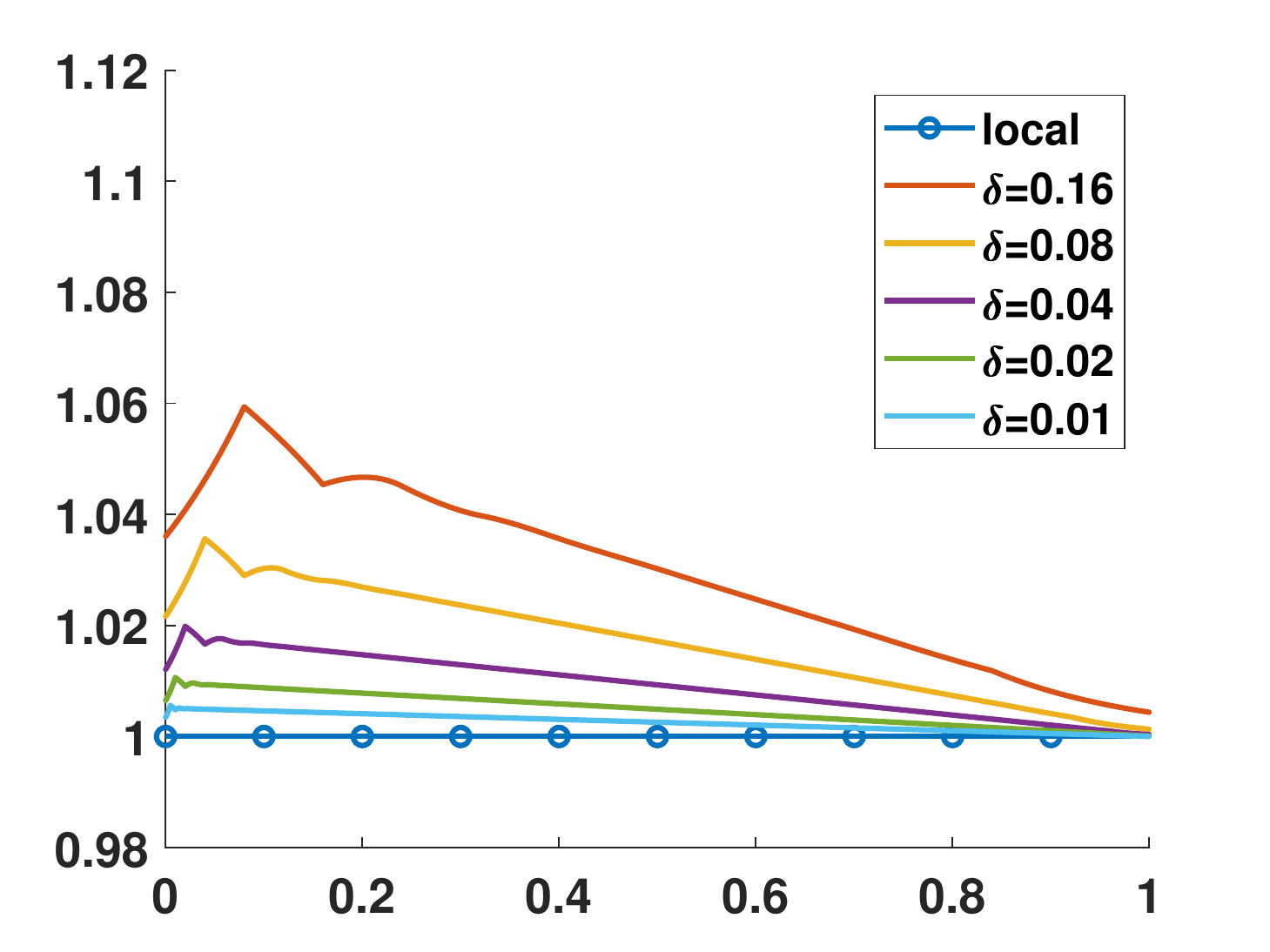}\\ 
        $s=0.75$  &  $s=0.25$ &  $s=-1$
   \end{tabular}
   \caption{Example 2 (a) (upper) and  2 (b) (lower). Solutions to the nonlocal diffusion problem with various horizon parameters $\delta$ and fractional powers $s$. }
   \label{fig:m1}
\end{figure}

\begin{figure}[hbt!]
\setlength{\tabcolsep}{0pt}
   \centering
   \begin{tabular}{ccc}
   \includegraphics[width=0.33\textwidth]{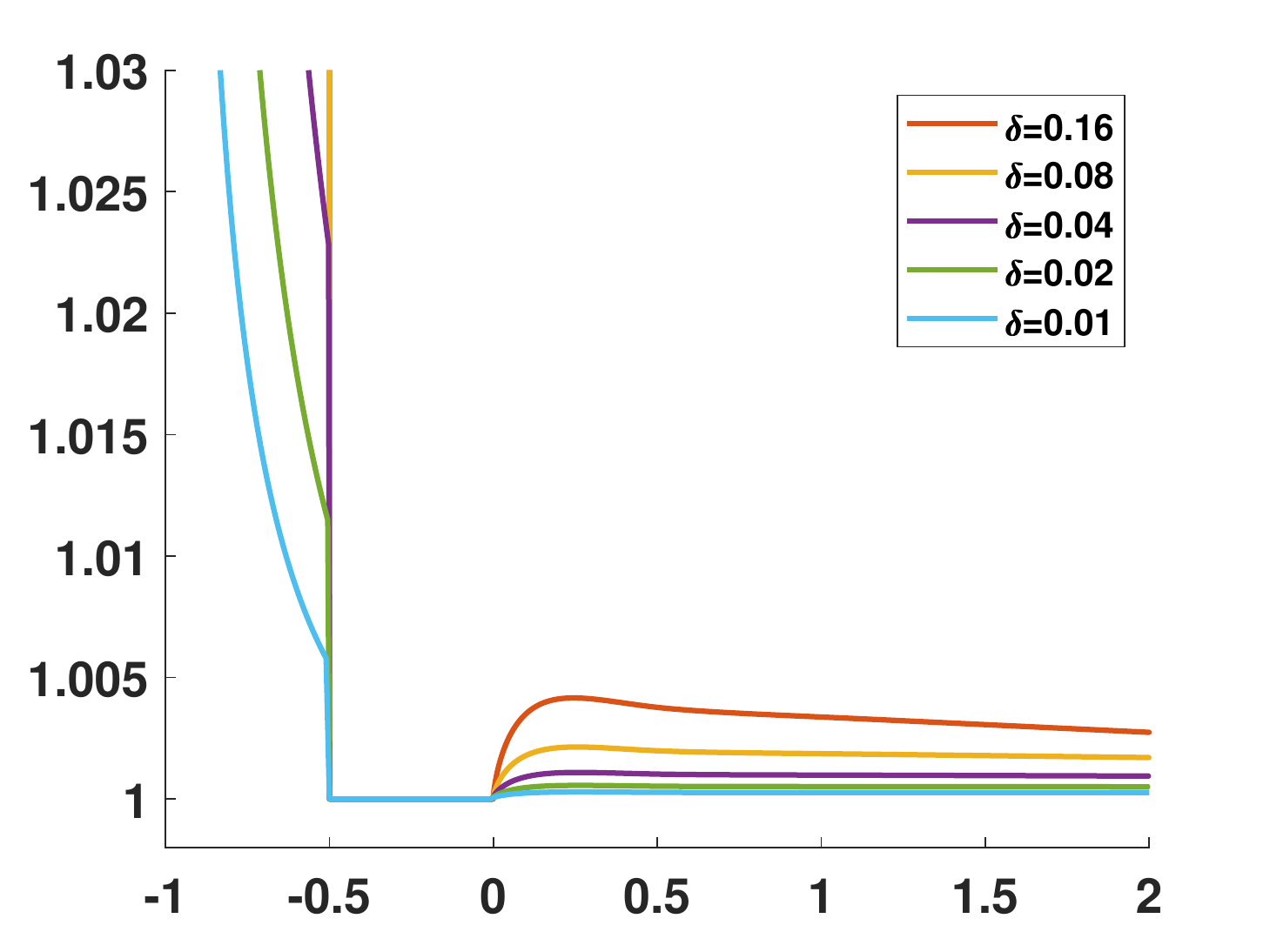}&
   \includegraphics[width=0.33\textwidth]{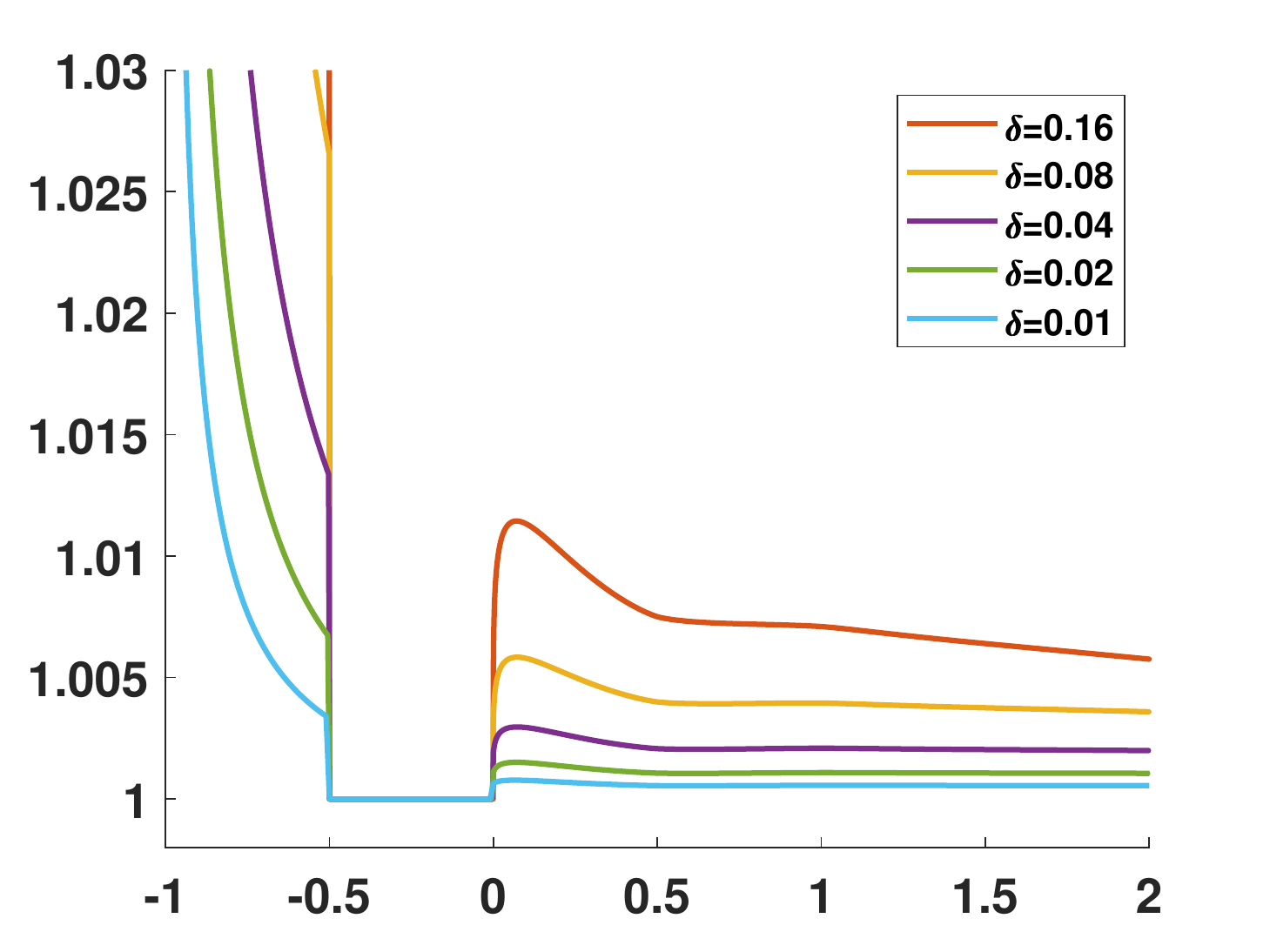} & \includegraphics[width=0.33\textwidth]{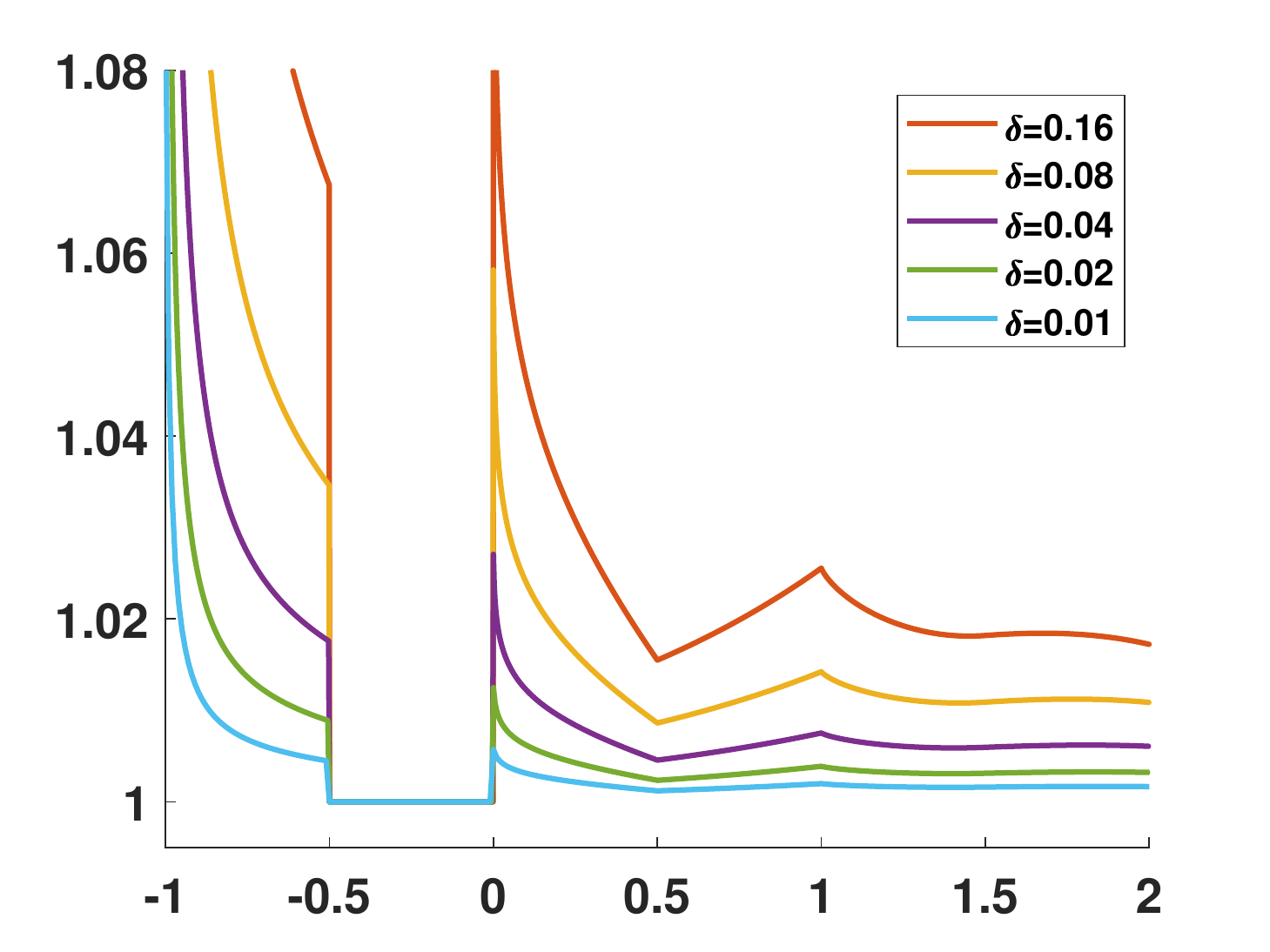}\\ 
   \includegraphics[width=0.33\textwidth]{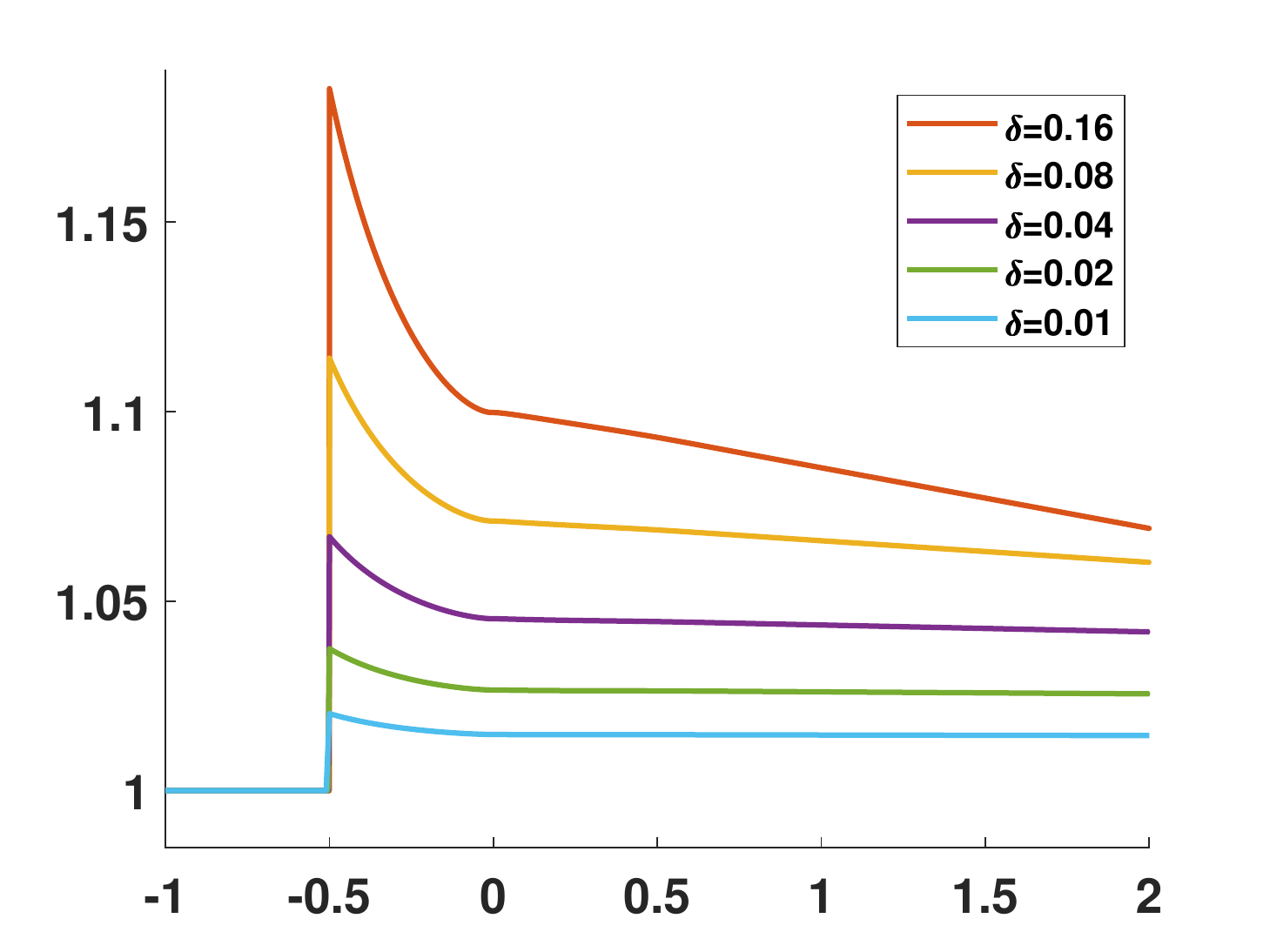}&
   \includegraphics[width=0.33\textwidth]{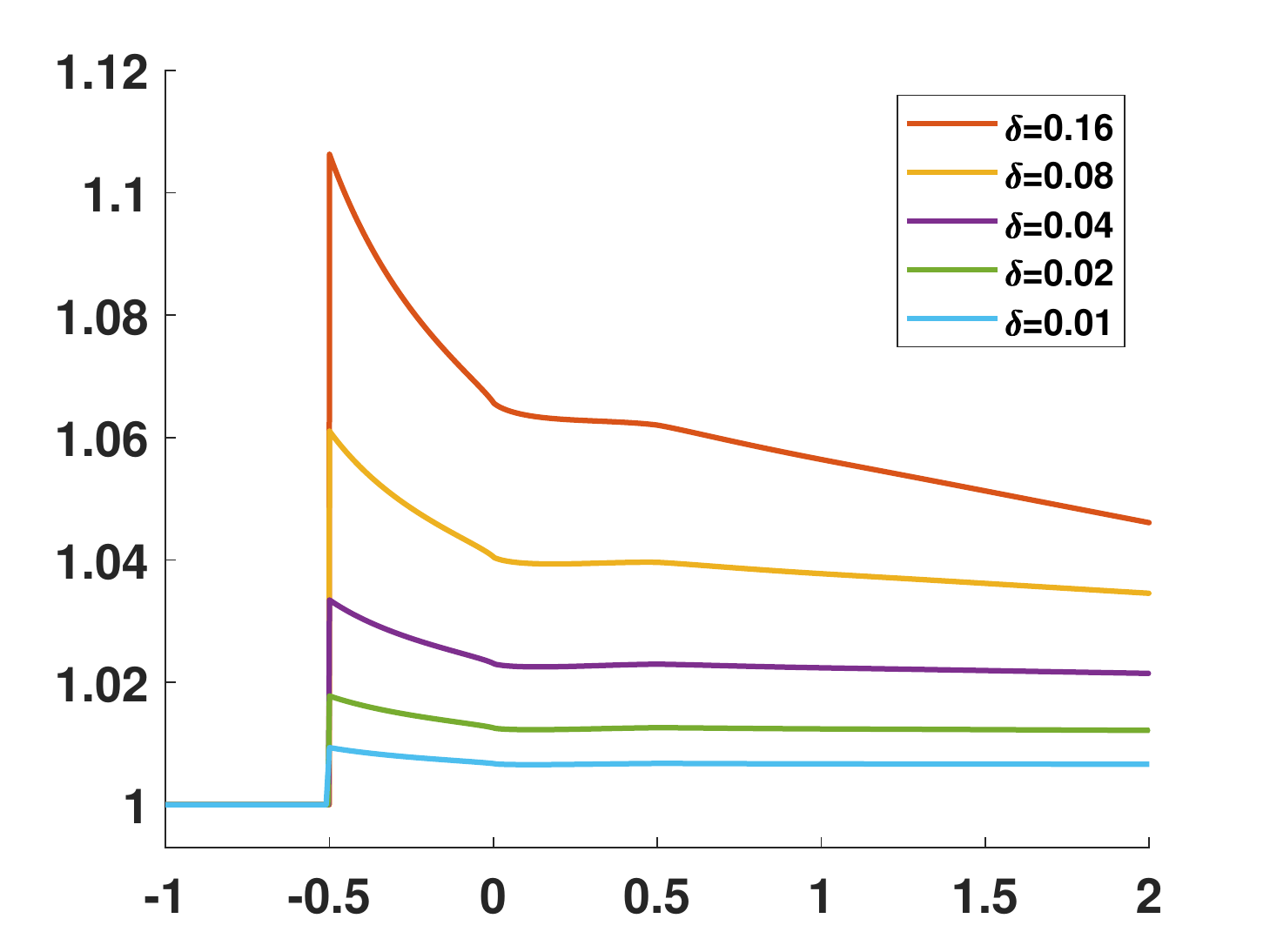} & \includegraphics[width=0.33\textwidth]{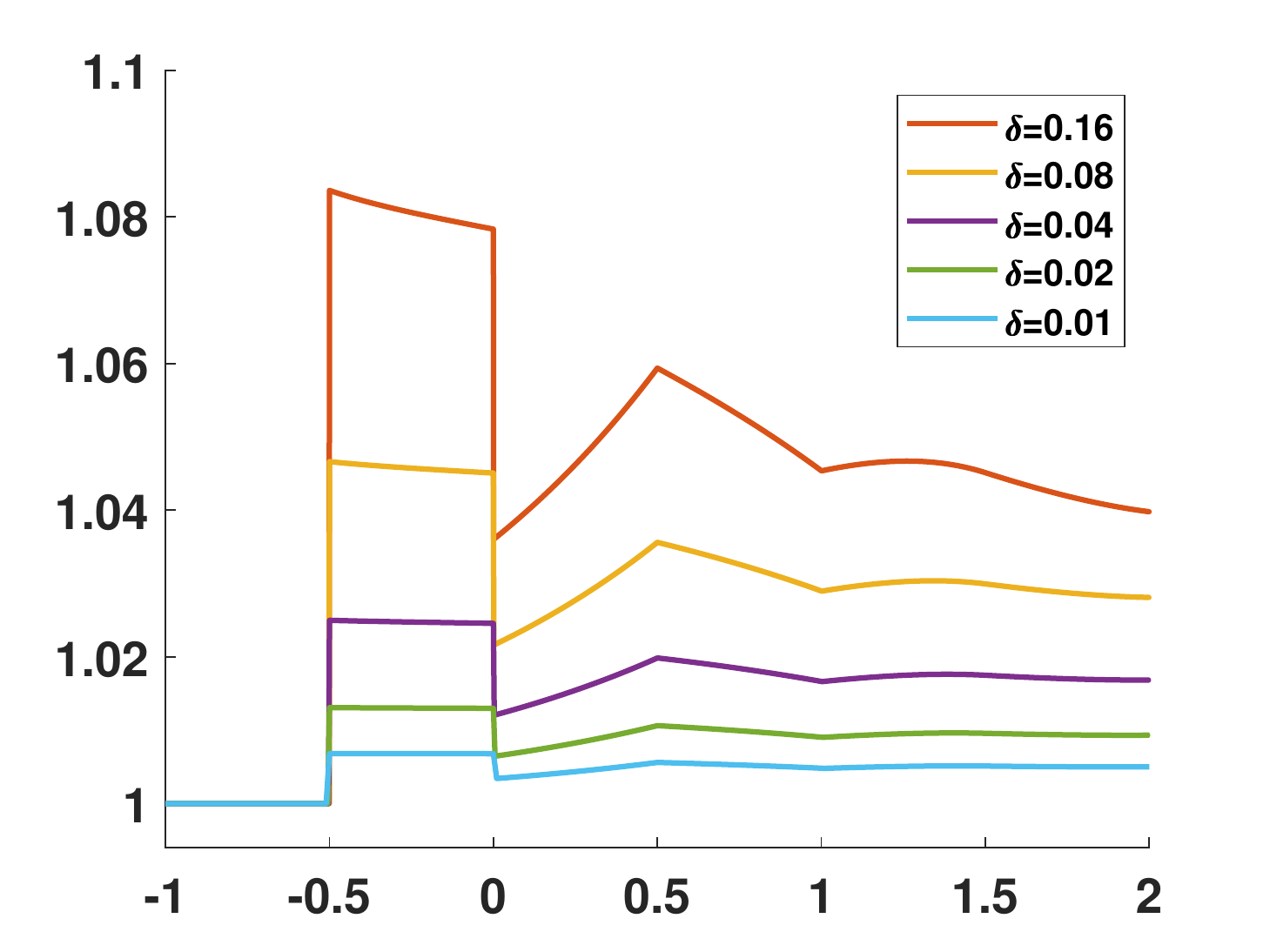}\\ 
        $s=0.75$  &  $s=0.25$ &  $s=-1$
  \end{tabular}
\caption{Example 2(a)(upper) and 2(b)(lower). Profiles of $u_\delta(\delta y)$, $y\in(-1,2)$, with various  horizon parameters $\delta$ and fractional powers $s$.}
\label{fig:m1-near0}
\end{figure}

\begin{table}[hbt!]
\centering
   \caption{Example 2 (a) and (b). $\|u_\delta - u_0\|_{L^2(\Omega)}$ with $s=0.75$, $0.25$ and $-1$.}\label{tab:Ex2-rate}
   \vskip-5pt
\begin{tabular}{|c|c|ccccc|c|}
\hline
&$s\backslash\delta$& 0.16& 0.08 & 0.04 & 0.02 & 0.01 &Rate ($\delta$)\\
\hline
&$ -1$ & 1.67E-2 & 8.04E-3  & 3.95E-3 & 1.96E-3& 9.84E-4 & 1.02\\
(a)&$  0.25$ &5.11E-3  &2.53E-3 &1.27E-3& 6.40E-4 & 3.27E-5 & 0.99 \\
&$ 0.75$  &2.32E-3 & 1.17E-3 & 5.92E-4 & 3.01E-4 &1.55E-4 & 0.98 \\
\hline
&$  -1$ & 3.33E-2 & 1.93E-2  & 1.06E-2 & 5.63E-3& 2.96E-3   & 0.87\\
(b)&$ 0.25$ &3.90E-2 & 2.38E-2  & 1.35E-2 & 7.31E-3 & 3.90E-3 & 0.83 \\
&$  0.75$  &5.87E-2 & 4.15E-2  & 2.63E-2 & 1.54E-2 & 8.56E-3 & 0.69 \\
\hline
\end{tabular}
\end{table}

\vskip5pt

\textbf{Example 3. Fractional limit.}
In this example, we shall numerically check the limit of nonlocal diffusion models as $\delta\rightarrow \infty$. To this end, we
consider the boundary value problem of steady fractional diffusion model
\begin{equation}\label{eq:fracPoisson-num}
\begin{cases}
-\mcLis u_\infty(x) + u_\infty = f(x), &\quad\forall\, x\in\omg{,}
\\
\mathcal{N}_\infty^s u_\infty(x)=g(x),&\quad\forall\, x\in
\omgc ,
\end{cases}
\end{equation}
where the fractional operator $\mcLis$ and the Neumann operator $\mathcal{N}_\infty^s$ are defined in \eqref{intfl} and \eqref{eq:mcns}, respectively. 
Since the closed form of the fractional model is unavailable, we assume that the solution is a Gaussian function: 
\begin{equation}\label{uinf}
    u_\infty(x) = \exp\big(-(x-\gfrac12)^2\big)\qquad \forall ~~x\in \mathbb{R}.
\end{equation}
Then it is easy to compute the source term $f$ and the Neumann boundary data $g$ numerically, since $u_\infty(x)$ decays double exponentially as $|x|\rightarrow \infty$.
In our computation, we evaluate the numerical solution of the following
boundary value problem of the steady nonlocal diffusion problem with a finite horizon:
\begin{equation}\label{eq:nonlocal-num}
\begin{cases}
-\mathcal{L}_\delta u_\delta(x) + u_\delta= f(x), &\quad\forall\, x\in\omg{,}
\\
\mathcal{N}_\delta u_\delta(x)=g(x),&\quad\forall\, x\in
\Omega_{\mathcal{I}_\delta},
\end{cases}
\end{equation}
Note that, by adding lower order terms $u_\infty$ and $u_\delta$ in \eqref{eq:fracPoisson-num} and \eqref{eq:nonlocal-num} respectively, we will not need to be concerned with data incompatibility for the corresponding fractional and nonlocal problems.

In Figure \ref{fig:frac}, we plot the numerical solutions $u_\delta$ with various $\delta$ and $s$. The numerical results clearly show the convergence of $u_\delta$ to $u_\infty$ as $\delta\rightarrow \infty$. Besides, for smaller $s$, we observe a slower convergence, which might be due to the slower decay of the kernel function in the case of smaller $s$. 

Similar to Example 1, we also consider the boundary value problem with another type of Neumann boundary condition:
\begin{equation}\label{eq:fracPoisson-num2}
\begin{cases}
-\mcLis u_\infty(x) + u_\infty = f(x), &\quad\forall\, x\in\omg{,}
\\
\widetilde{\mathcal{N}}_\infty^s u_\infty(x)=\widetilde g(x),&\quad\forall\, x\in
\omgc ,
\end{cases}
\end{equation}
where the fractional Neumann operator is defined by
\begin{equation}\label{eqn:Ns2}
  \widetilde{\mathcal{N}}_\infty^s u(x)= - C_{s,\infty}\int_{\mathbb{R}} \frac{u(y)-u(x)}{|x-y|^{1+2s}}   \,dy,  \quad\quad\forall\,  x\in{\omgc}.
\end{equation}
Let the solution to \eqref{eq:fracPoisson-num2} be of the form \eqref{uinf}. Then we are able to compute the new Neumann boundary data 
$\widetilde g(x) =\widetilde{\mathcal{N}}_\infty^s u_\infty(x)$.
The corresponding nonlocal model reads
\begin{equation}\label{eq:nonlocal-num2}
\begin{cases}
-\mathcal{L}_\delta u_\delta(x) + u_\delta= f(x), &\quad\forall\, x\in\omg{,}
\\
\mcN_{\delta,\omghd} 
u_\delta(x)=\widetilde g(x),&\quad\forall\, x\in
\Omega_{\mathcal{I}_\delta},
\end{cases}
\end{equation}
Again, our numerical results indicate that the functions $u_\delta$ converge to $u_\infty$ as $\delta \rightarrow \infty$, and the convergence rate deteriorates as $s\rightarrow 0$, see Figure \ref{fig:frac}.
For both cases, we observe that the convergence rate is around $O(\delta^{-2s})$(see Table \ref{tab:Ex3-rate}), which still awaits theoretical justification.

\begin{figure}[hbt!]
\setlength{\tabcolsep}{0pt}
   \centering
   \begin{tabular}{ccc}
   \includegraphics[width=0.33\textwidth]{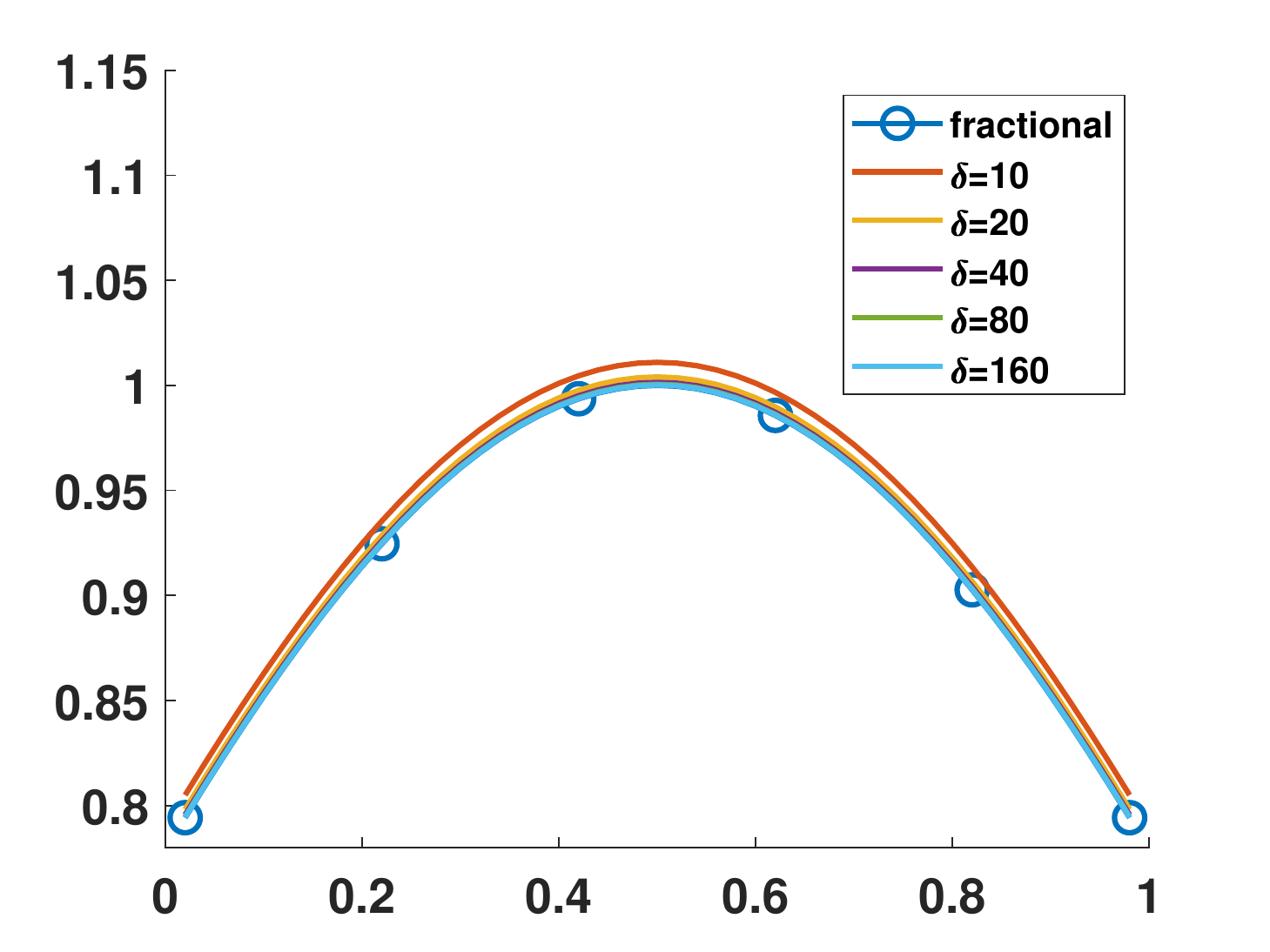}&
   \includegraphics[width=0.33\textwidth]{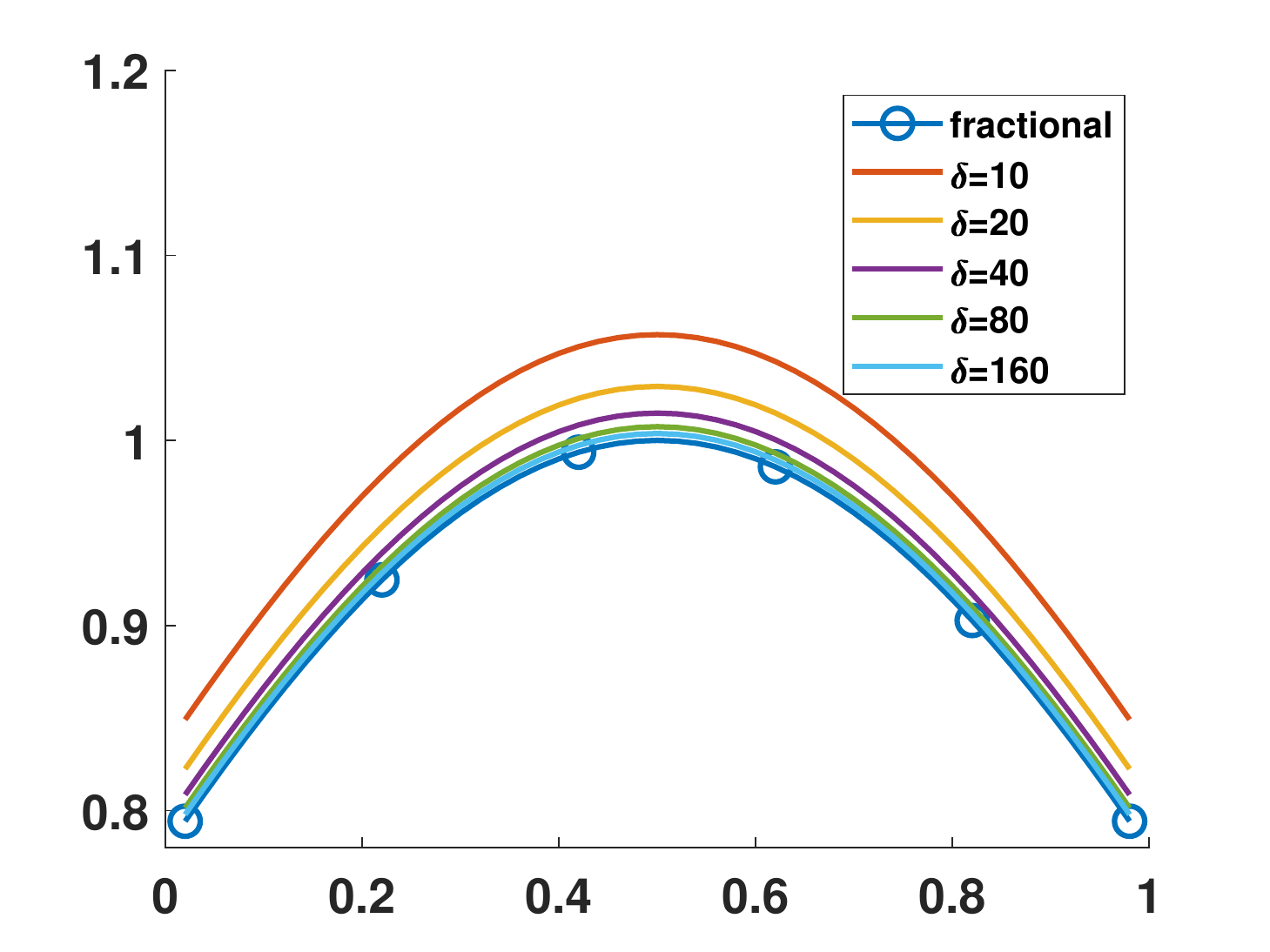} & \includegraphics[width=0.33\textwidth]{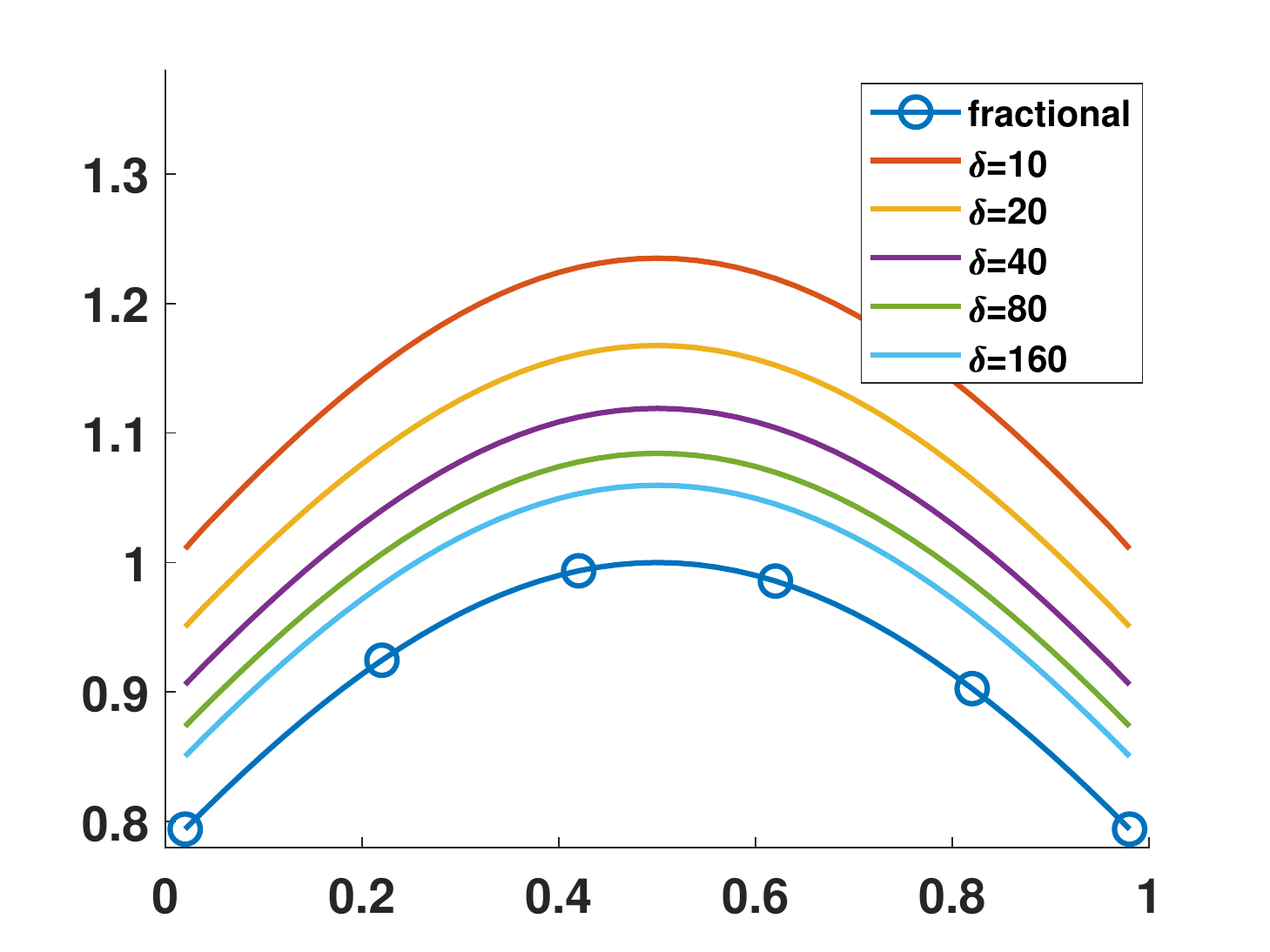}\\ 
      \includegraphics[width=0.33\textwidth]{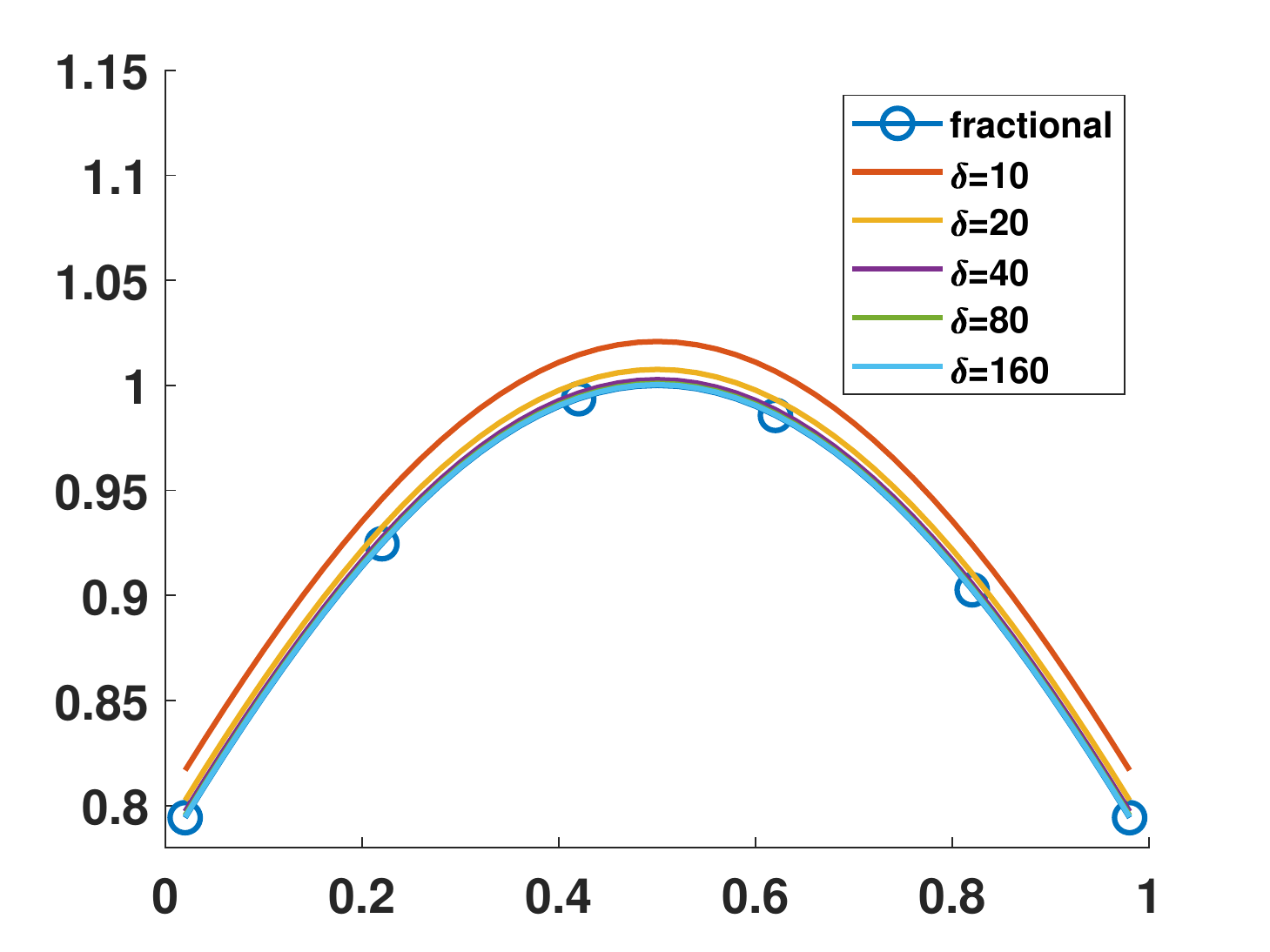}&
   \includegraphics[width=0.33\textwidth]{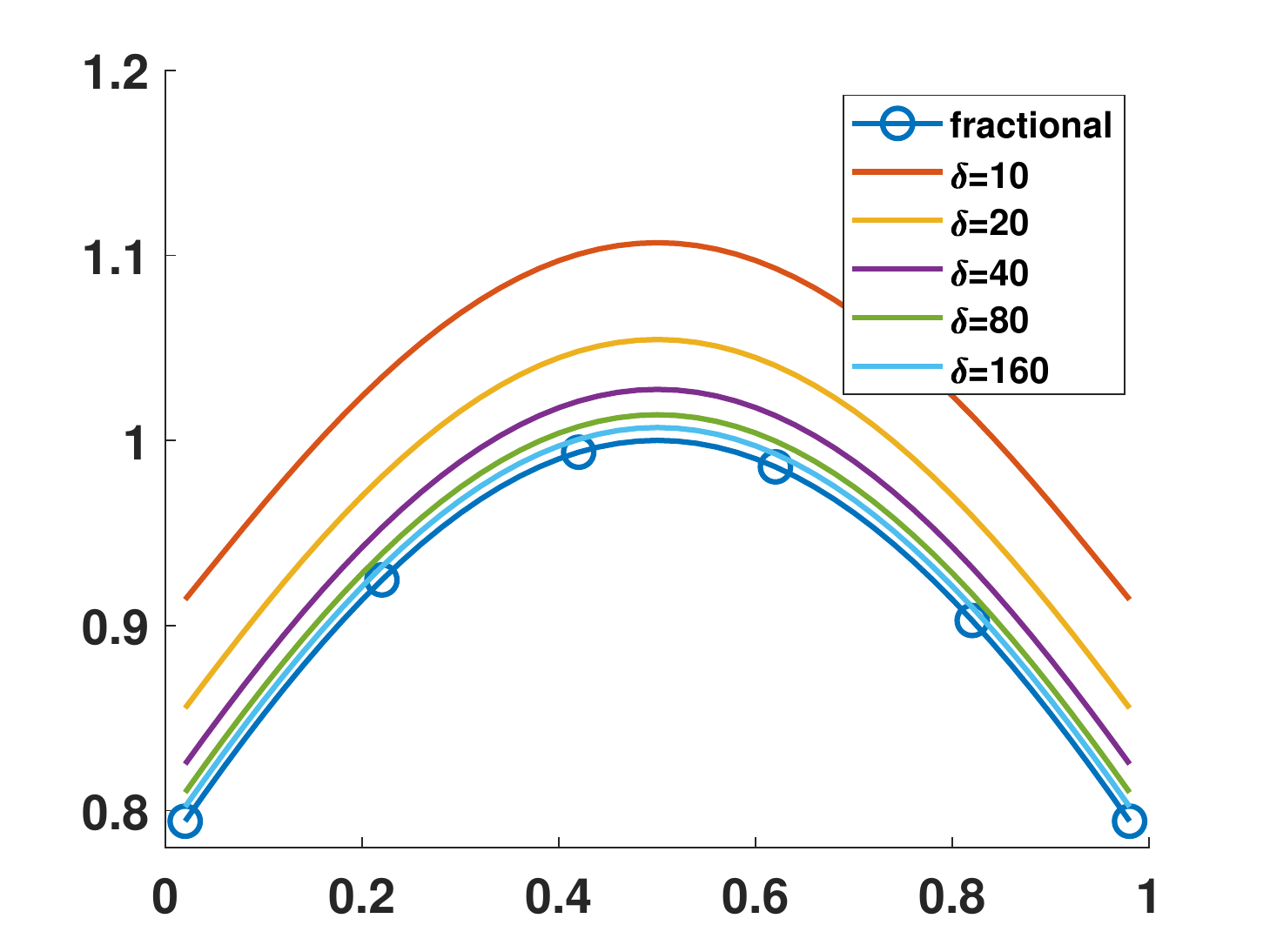} & \includegraphics[width=0.33\textwidth]{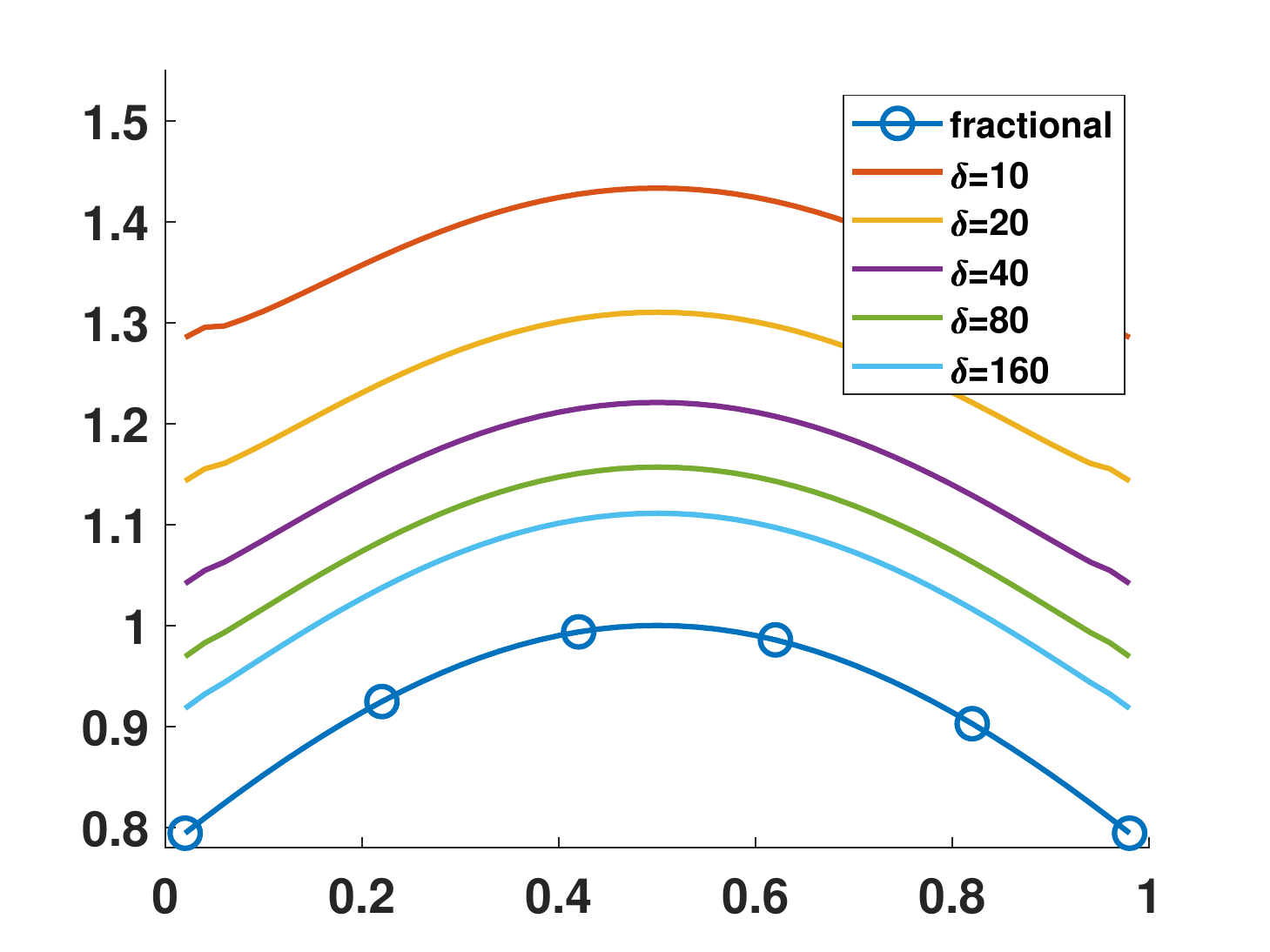}\\ 
        $s=0.75$  &  $s=0.5$ &  $s=0.25$
   \end{tabular}
   \caption{Example 3. Solutions to the nonlocal diffusion problem with inhomogeneous Neumann type boundary conditions, with operators $\mcNd$ (upper) and $\mcN_{\delta,\omghd}$ (lower).}
   \label{fig:frac}
\end{figure}

\begin{table}[htbp]
\centering
 \caption{Example 3. $\|u_\delta - u_\infty\|_{L^2(\Omega)}$ 
 with $s=0.25$, $0.5$ and $0.75$.}\label{tab:Ex3-rate}
   \vskip-5pt
\begin{tabular}{|c|c|ccccc|c|}
\hline
&$s\backslash$ $\delta$&  10 & 20 & 40 & 80&160 &Rate ($\delta$)\\
\hline
&0.25 & 2.25E-1 & 1.61E-1  & 1.15E-1 & 8.12E-2  & 5.75E-2& 0.49\\
$\mathcal N_\delta$ & 0.50 &5.55E-2  & 2.84E-2 & 1.44E-2 & 7.22E-3 & 3.62E-3& 0.98 \\
& 0.75  &1.07E-2 & 3.93E-3 & 1.41E-3 & 5.03E-4 &  1.78E-4& 1.48 \\
\hline
 & 0.25& 4.40E-1 & 3.15E-1  & 2.24E-1 & 1.59E-1 & 1.12E-1  & 0.49\\
$\mcN_{\delta,\omghd}$ & 0.50 &1.01E-1  & 5.56E-2 & 2.81E-2 & 1.41E-2 & 7.08E-3 & 0.99 \\
 &0.75  &2.11E-2 & 7.69E-3 & 2.76E-3 & 9.84E-3 & 3.48E-3& 1.48 \\
\hline
\end{tabular}
\end{table}

\section{Conclusion}

 The discussion here offers a glimpse into some of the recent studies of nonlocal models with a finite range of interactions defined on a bounded domain involving various types of nonlocal boundary conditions. In particular, we 
 provide a consistent presentation of such models to classical PDEs, seen as the local limits, and fractional PDEs, seen as the limit with an infinite range of interactions. 
 Obviously, there are many more known results in the literature and more interesting and open questions to be considered. For example, while the associated nonlocal boundary conditions are illustrated for problems defined in a bounded domain, the
inhomogeneous data are only prescribed in abstract function spaces. It is interesting to give more precise characterizations of these spaces by applying recently developed trace and extension theorems for the nonlocal function spaces
\cite{Trace21,dyda2019function}.  
Also, while we stress the consistent local and fractional limits, one can also further examine how the rates of convergence depend on the model and data in these limits \cite{dy19jsc,du2019uniform,d2020physically}. Moreover, the formulations are 
 only presented in the context of time-independent scalar equations, one can consider the extensions to nonlinear systems and evolution equations, particularly those connected to peridynamics \cite{Silling00,dyz17dcdsb,dtt18je,du19cbms}. 
 Another interesting topic, as previously mentioned in the discussion, is on the connection of the nonlocal problems to the stochastic process and their limits, from the local diffusion equation and Brownian motion to fractional diffusion and L\'evy jump processes \cite{Dipierro:2017,DuLZ:2020}.
 
 In addition, one may consider the numerical analysis of approximation methods and discuss, among many practical issues, the issue of asymptotic compatible approximation \cite{TiDu14}. Further numerical experiments can also be carried out to provide more comparisons of the nonlocal models and their various limits.

\bibliography{nonlocal}

\end{document}